\documentclass[12pt]{article}
\usepackage{latexsym, amsmath, amsfonts, amsthm, amssymb}
\usepackage{times}
\usepackage{a4wide}
\usepackage{graphicx, tocloft}
\usepackage{titlesec, titling}
\titleformat{\section}{\large\bfseries}{\thesection}{1em}{}
\titleformat{\subsection}{\normalsize\bfseries}{\thesubsection}{1em}{}

\def \ora {\overrightarrow}
\def \Hess {\text{\it Hess}}
\def \Trace {\text{\it Trace}}
\def \Strain {\text{\it Strain}}
\def \Diam {\text{\it Diam}}
\def \Inactive {\text{\it Loose}}
\def \Ends {\text{\it Ends}}
\def \Ric {\text{\it Ric}}
\def \RR {\mathbb R}
\def \QQ {\mathbb Q}

\def \eps {\varepsilon}
\def \vphi {\varphi}

\def \cS {\mathcal S}
\def \cL {\mathcal L}
\def \cM {\mathcal M}
\def \cF {\mathcal F}
\def \cJ {\mathcal J}

\def \cB {\mathcal B}
\def \cN {\mathcal N}

\def \cI {\mathcal I}

\def \sgn {{\rm sgn}}

\newtheorem{theorem}{Theorem}[section]
\newtheorem{lemma}[theorem]{Lemma}
\newtheorem{proposition}[theorem]{Proposition}
\newtheorem{corollary}[theorem]{Corollary}
\newtheorem{remark}[theorem]{Remark}
\newtheorem{example}[theorem]{Example}
\newtheorem{definition}[theorem]{Definition}
\newtheorem{conjecture}[theorem]{Conjecture}

\def\myffrac#1#2 in #3{\raise 2.6pt\hbox{$#3 #1$}\mkern-1.5mu\raise 0.8pt\hbox{$
#3/$}\mkern-1.1mu\lower 1.5pt\hbox{$#3 #2$}}

\def\qed{\hfill $\vcenter{\hrule height .3mm
\hbox {\vrule width .3mm height 2.1mm \kern 2mm \vrule width .3mm
height 2.1mm} \hrule height .3mm}$ \bigskip}

\def \Exp {{\rm Exp}}
\def \dExp {{\rm dExp}}
\def \dexp {{\rm dexp}}

\begin{document}

\pretitle{\vspace{-30pt}\begin{center}\LARGE}
\pagenumbering{gobble}
\title{Needle decompositions in Riemannian geometry}
\author{Bo'az Klartag\thanks{School of Mathematical Sciences, Tel Aviv University, Tel Aviv 69978, Israel. E-mail: { klartagb@tau.ac.il}}}
\date{}
\maketitle

\vspace{-20pt}
\abstract{ The
localization technique from
convex geometry is generalized to the setting of Riemannian manifolds
whose Ricci curvature is bounded from below.
In a nutshell, our method is based on the following observation: When the Ricci curvature is non-negative,  log-concave measures are obtained
when conditioning the Riemannian volume measure
with respect to an integrable geodesic foliation. The Monge mass transfer problem plays an important role in our analysis.
 }

\bigskip
\renewcommand\cftsecfont{\normalsize}
\renewcommand\cftsecpagefont{\normalsize}
\renewcommand{\cftsecleader}{\cftdotfill{\cftdotsep}}
\tableofcontents

\newpage
\section{Introduction}
\pagenumbering{arabic}
\label{sec_intro}

The {\it localization technique} in convex geometry is a method for reducing $n$-dimensional problems
to one-dimensional problems, that was developed by
Gromov and Milman \cite{GM}, Lov\'asz and Simonovits \cite{LS} and Kannan, Lov\'asz and Simonovits \cite{KLS}.
Its earliest appearance seems to be found in the work of Payne and Weinberger \cite{PW},
where
the following inequality is stated:
For any bounded, open, convex set $K \subset \RR^n$ and an integrable, $C^1$-function $f: K \rightarrow \RR$,
\begin{equation}  \int_K f = 0 \qquad \Longrightarrow  \qquad \int_K f^2 \, \leq \, \frac{\Diam^2(K)}{\pi^2} \int_K |\nabla f|^2,
\label{eq_1356}
\end{equation}
where $\Diam(K) = \sup_{x,y \in K} |x-y|$ is the diameter of $K$, and $| \cdot |$ is the standard Euclidean norm in $\RR^n$.
 The localization proof of (\ref{eq_1356}) goes roughly as follows: Given $f$ with $\int_{K} f = 0$,
one finds a hyperplane $H \subset \RR^n$ such that $ \int_{K \cap H^+} f = \int_{K \cap H^-} f = 0$, where $H^-, H^+ \subset \RR^n$
are the two half-spaces determined by the hyperplane $H$. The problem of proving (\ref{eq_1356}) is reduced
to proving the two inequalities:
$$  \int_{K \cap H^{\pm}} f^2 \, \leq \, \frac{\Diam^2(K \cap H^{\pm})}{\pi^2} \int_{K \cap H^{\pm}} |\nabla f|^2. $$
The next step is to again bisect  each of the two half-spaces separately, retaining the requirement that the integral of
$f$ is zero. Thus one recursively obtains finer and finer partitions of $\RR^n$ into convex cells. At the $k^{th}$ step,
the proof of (\ref{eq_1356}) is reduced to $2^k$ ``smaller'' problems of a similar nature. At the limit,
the original problem
is reduced to a lower-dimensional problem, and eventually even to a one-dimensional problem. This one-dimensional problem
has turned out to be relatively simple to solve.

\medskip This bisection technique has no clear analog
in the context of an abstract Riemannian manifold. The purpose of this
manuscript is to try and bridge this gap between convex geometry and Riemannian geometry.

\medskip There are only two parameters of a given Riemannian manifold that play a role in our analysis: the dimension of the manifold,
and a uniform lower bound $\kappa$ for its Ricci curvature.
We say that an $n$-dimensional Riemannian manifold $\cM$ satisfies the curvature-dimension condition $CD(\kappa, N)$
for $\kappa \in \RR$ and $N \in (-\infty, 1) \cup [n, +\infty ]$ if
\begin{equation} \Ric_{\cM}(v, v)  \geq \kappa \cdot g(v,v) \qquad \qquad \text{for} \ p \in \cM, v \in T_p \cM, \label{eq_I1000}
\end{equation}
 where $g$ is the Riemannian metric tensor and $\Ric_{\cM}$ is the Ricci tensor of $\cM$.
 The contribution of Bakry and \'Emery \cite{BE} has made it clear
 that
{\it weighted} Riemannian manifolds are
convenient
for the study of curvature-dimension conditions.
 A weighted Riemannian manifold is a triplet $ (\cM, d, \mu), $
where $\cM$ is an $n$-dimensional Riemannian manifold with Riemannian distance function $d$, and where the measure  $\mu$  has
a smooth, positive density $e^{-\rho}$ with respect to the Riemannian volume measure on $\cM$.
The {\it generalized Ricci tensor} of the weighted Riemannian
manifold $(\cM, d, \mu)$ is defined via
\begin{equation}  \Ric_{\mu}(v,v) := \Ric_{\cM}(v,v) \, + \, \Hess_\rho(v,v) \qquad \qquad \text{for} \ p \in \cM, v \in T_p \cM,
\label{eq_I303} \end{equation}
where
$\Hess_\rho$ is the Hessian form associated with  the smooth function $\rho: \cM \rightarrow \RR$.
For $N \in (-\infty, 1) \cup [n, +\infty], p \in \cM$ and $v \in T_p \cM$ we define the {\it generalized Ricci tensor with parameter $N$} as
follows:
\begin{equation} \Ric_{\mu, N}(v,v) := \left \{ \begin{array}{lccl} \Ric_{\mu}(v,v)  \, - \, \frac{(\partial_v \rho)^2}{N - n}  &  & & N \neq n, +\infty \\
\Ric_{\mu}(v,v)  & & & N = +\infty \\
\Ric_{\cM}(v,v) & & & N = n, \rho \equiv Const
\end{array} \right.
 \label{eq_I1206}
\end{equation}
The standard agreement is that $Ric_{\mu, n}(v,v)$ is undefined unless $\rho$ is a constant function.
For $\kappa \in \RR$ and $N \in (-\infty, 1) \cup [n, +\infty]$
we say that $(\cM, d, \mu)$ satisfies the curvature-dimension condition  $CD(\kappa, N)$ when
$$ \Ric_{\mu, N}(v, v)  \geq \kappa \cdot g(v,v) \qquad \qquad \text{for} \ p \in \cM, v \in T_p \cM. $$
For instance, the $CD(0, \infty)$-condition is equivalent to the requirement that the generalized Ricci tensor be non-negative.
We refer the reader to Bakry, Gentil and Ledoux \cite{BGL} for background
on weighted Riemannian manifolds of class $CD(\kappa, N)$.
In this manuscript, a {\it minimizing geodesic} is a curve $\gamma: A \rightarrow \cM$, where $A \subseteq \RR$ is a connected set, such that
$$  d(\gamma(s), \gamma(t)) = |s-t| \qquad \qquad \qquad \text{for all} \ s,t \in A.
$$

\begin{definition} Let $\kappa \in \RR, 1 \neq N \in \RR \cup \{ \infty \}$ and let $\nu$ be a  measure on
the Riemannian manifold $\cM$.
We say that $\nu$ is a ``$CD(\kappa,N)$-needle'' if there exist a non-empty, connected open set $A \subseteq \RR$,
a smooth function $\Psi: A \rightarrow \RR$ and a minimizing geodesic $\gamma:A \rightarrow \cM$ such that:
\begin{enumerate}
\item[(i)] Denote by $\theta$ the measure on $A \subseteq \RR$ whose density with respect to the Lebesgue measure
is $e^{-\Psi}$. Then $\nu$ is the push-forward of $\theta$ under the map $\gamma$.
\item[(ii)] The following inequality holds in the entire set $A$:
\begin{equation}
\Psi^{\prime \prime} \geq \kappa + \frac{(\Psi^{\prime})^2}{N-1},
\label{eq_B1952}
\end{equation}\end{enumerate}
where in the case $N = \infty$, we interpret the term $(\Psi^{\prime})^2 / (N-1)$ as zero.
\label{def_cd}
\end{definition}

Condition (\ref{eq_B1952}) is equivalent to condition $CD(\kappa, N)$ for
the weighted Riemannian manifold $(A, d,  \theta)$ with $d(x,y) = |x-y|$.
Examples of needles include:

\begin{enumerate}
\item {\it Log-concave needles} which are defined to be $CD(0, \infty)$-needles. In this case, $\Psi$ is a convex function.
Log-concave needles are valuable when studying the uniform measure on convex sets in $\RR^n$ for large $n$.

\item  A {\it{$\sin^n$-concave needle}} is a $CD(n-1, n)$-needle. These are relevant to the sphere $S^{n}$,
since the $n$-dimensional unit sphere is of class $CD(n-1,n)$.

\item The {\it $N$-concave needles} are  $CD(0, N+1)$-needles with $N > 0$. Here, $f^{1/N}$ is a concave function,
where $f = e^{-\Psi}$ is the density of the measure $\theta$. For $N < 0$, the $CD(0,N+1)$-condition is equivalent to the convexity
of $f^{-1/|N|}$.

\item  A {\it $\kappa$-log-concave needle} is a $CD(\kappa,\infty)$-needle.
\end{enumerate}

These examples are discussed by
Gromov \cite[Section 4]{gr_waist}.
We say that the Riemannian manifold $\cM$ is {\it geodesically-convex}
if any two points in $\cM$ may be connected by a minimizing geodesic. By the Hopf-Rinow theorem, any complete, connected Riemannian manifold is
geodesically-convex. A partition of $\cM$ is a collection of non-empty disjoint subsets of $\cM$ whose union equals $\cM$.

\begin{theorem}[``Localization theorem'']
Let $n \geq 2, \kappa \in \RR$ and $N \in (-\infty, 1) \cup [n, +\infty ]$.
Assume that $(\cM,d,\mu)$ is an $n$-dimensional weighted Riemannian manifold
of class $CD(\kappa, N)$ which is geodesically-convex.
Let $f: \cM \rightarrow \RR$ be a $\mu$-integrable function with $ \int_\cM f d \mu = 0$.
Assume that there exists a point $x_0 \in \cM$ with $\int_\cM |f(x)| \cdot  d(x_0, x) d \mu(x) < \infty$.

\medskip Then there exist a partition $\Omega$ of $\cM$,
a measure $\nu$ on $\Omega$
and a family $\{ \mu_{\cI} \}_{\cI \in \Omega}$ of measures on $\cM$
such that:
\begin{enumerate}
\item[(i)] For any Lebesgue-measurable set $A \subseteq \cM$,
$$ \mu(A) = \int_{\Omega} \mu_{\cI}(A) d \nu(\cI) $$
(In particular, the map $\cI \mapsto \mu_{\cI}(A)$ is well-defined $\nu$-almost everywhere and it is a $\nu$-measurable map).
In other words, we have a ``disintegration of the measure $\mu$''.
\item[(ii)] For $\nu$-almost any $\cI \in \Omega$, the set $\cI \subseteq \cM$ is the image of a minimizing geodesic,
the measure $\mu_{\cI}$ is supported on $\cI$, and either $\cI$ is a singleton or else $\mu_{\cI}$ is a $CD(\kappa, N)$-needle.
\item[(iii)] For $\nu$-almost any $\cI \in \Omega$ we have $ \int_{\cI} f d \mu_{\cI} = 0$.
\end{enumerate}
\label{thm_main}
\end{theorem}

We demonstrate in Section \ref{applications} that Theorem
\ref{thm_main} may be used in order to obtain alternative proofs of some  familiar inequalities from convex and Riemannian geometry.
These include the isoperimetric inequality, the Poincar\'e and log-Sobolev
inequalities, the Payne-Weiberger/Yang-Zhong inequality, the inequality of Cordero-Erausquin, McCann and Schmuckenschlaeger,
 among others.  Some of these inequalities are consequences
 of the following Riemannian analog of the  four functions theorem
of Kannan, Lov\'asz and Simonovits \cite{KLS}:

\begin{theorem}[``The four functions theorem''] Let $n \geq 2, \alpha, \beta > 0, \kappa \in \RR, N \in (-\infty, 1) \cup [n, +\infty]$.
Let $(\cM, d, \mu)$ be an $n$-dimensional weighted Riemannian manifold
of class $CD(\kappa, N)$ which is geodesically-convex. Let $f_1, f_2, f_3, f_4: \cM \rightarrow [0, +\infty)$ be measurable functions
such that there exists $x_0 \in \cM$ with
$$ \int_{\cM} \left( |f_1(x)| + |f_2(x)| + |f_3(x)| + |f_4(x)| \right) \cdot (1 + d(x_0, x))  d \mu(x) < \infty. $$
Assume that $f_1^{\alpha} f_2^{\beta} \leq f_3^{\alpha} f_4^{\beta}$ almost-everywhere in $\cM$ and
that for any probability measure $\eta$ on $\cM$ which is a $CD(\kappa, N)$-needle,
\begin{equation}  \left( \int_{\cM} f_1 d \eta \right)^{\alpha} \left( \int_{\cM} f_2 d \eta \right)^{\beta}
\leq \left( \int_{\cM} f_3 d \eta \right)^{\alpha} \left( \int_{\cM} f_4 d \eta \right)^{\beta} \label{eq_E951}
\end{equation}
whenever $f_1,f_2,f_3, f_4$ are $\eta$-integrable. Then,
\begin{equation}  \left( \int_{\cM} f_1 d \mu \right)^{\alpha} \left( \int_{\cM} f_2 d \mu \right)^{\beta}
\leq \left( \int_{\cM} f_3 d \mu \right)^{\alpha} \left( \int_{\cM} f_4 d \mu \right)^{\beta}. \label{eq_E945} \end{equation}
\label{prop_four_functions}
\end{theorem}

Theorem \ref{thm_main} was certainly known in the case where $\cM = \RR^n$ or
$\cM = S^{n-1}$.
However, even in these symmetric spaces, our proof of
Theorem \ref{thm_main} is very different from the traditional bisection proofs given in
Gromov and Milman \cite{GM} or Lov\'asz and Simonovits \cite{LS}.
The geodesic foliations that we construct in Theorem \ref{thm_main} are {\it integrable}, meaning that there is a function
$u: \cM \rightarrow \RR$ such that the geodesics appearing in the partition are integral curves of $\nabla u$.
This  integrability property makes the construction of the partition somewhat more ``canonical''.
In contrast, there are many arbitrary choices that one makes
during the bisection process, as there  could be many hyperplanes that bisect a domain in $\RR^n$ into two subsets of equal volumes.
For a function $u: \cM \rightarrow \RR$ we define its Lipschitz seminorm by
$$
  \| u \|_{Lip} = \sup_{x \neq y \in \cM} \frac{|u(x) - u(y)|}{d(x,y)}.
  $$
Given a $1$-Lipschitz function $u: \cM \rightarrow \RR$ and a point $y \in \cM$, we say that
$y$ is a {\it  strain point} of $u$ if there exist $x,z \in \cM$ for which
$$
 u(y) - u(x) = d(x,y) > 0, \quad u(z) - u(y) = d(y,z) > 0, \quad d(x,z) = d(x,y) + d(y, z).
 $$
 Write $\Strain[u] \subseteq \cM$ for the collection of all strain points of $u$. The set $\Strain[u]$ resembles the {\it transport set} defined at the beginning of Section 3 in  Evans and Gangbo \cite{EG}.
 It is explained below that $\Strain[u]$ is a measurable
subset of $\cM$.
  It is also proven  below that
the relation
$$  x \sim y \qquad \Longleftrightarrow \qquad |u(x) - u(y)| = d(x,y)
$$
is an equivalence relation on $\Strain[u]$, and that each equivalence class is the image of a minimizing geodesic.
Write $T^{\circ}[u]$ for the collection of all equivalence classes. It follows that for any $\cI \in T^{\circ}[u]$
there exists a minimizing geodesic $\gamma: A \rightarrow \cM$ with $\gamma(A) = \cI$ and
\begin{equation}  u(\gamma(t)) = t \qquad \qquad \qquad \text{for all} \ t \in A. \label{eq_I1220}
\end{equation}
Let $\pi: \Strain[u] \rightarrow T^{\circ}[u]$ be the partition map,
i.e., $x \in \pi(x) \in T^{\circ}[u]$ for all $x \in \Strain[u]$.
The conditioning
of $\mu$ with respect to the geodesic foliation
$T^{\circ}[u]$ is described in the following theorem:

\begin{theorem}
Let $n \geq 2, \kappa \in \RR$ and $N \in (-\infty, 1) \cup [n, +\infty ]$.
Assume that $(\cM,d,\mu)$ is an $n$-dimensional weighted Riemannian manifold
of class $CD(\kappa, N)$ which is geodesically-convex.
 Let $u: \cM \rightarrow \RR$ satisfy $\| u \|_{Lip} \leq 1$.
Then there exist a measure $\nu$ on the set $T^{\circ}[u]$
and a family $\{ \mu_{\cI} \}_{\cI \in T^{\circ}[u]}$ of measures on $\cM$
such that:
\begin{enumerate}
\item[(i)] For any Lebesgue-measurable set $A \subseteq \cM$, the map $\cI \mapsto \mu_{\cI}(A)$ is well-defined $\nu$-almost everywhere and is a $\nu$-measurable map.
If a subset $S \subseteq T^{\circ}[u]$ is $\nu$-measurable then $\pi^{-1}(S) \subseteq \Strain[u]$
is a measurable subset of $\cM$.
\item[(ii)] For any Lebesgue-measurable set $A \subseteq \cM$,
$$ \mu(A \cap \Strain[u]) = \int_{T^{\circ}[u]} \mu_{\cI}(A) d \nu(\cI). $$
\item[(iii)] For $\nu$-almost any $\cI \in T^{\circ}[u]$, the measure $\mu_{\cI}$ is a $CD(\kappa, N)$-needle supported on $\cI \subseteq \cM$.
Furthermore, the set $A \subseteq \RR$ and the minimizing geodesic $\gamma:A \rightarrow \cM$ from Definition \ref{def_cd} may be
selected so that $\cI = \gamma(A)$ and
so that (\ref{eq_I1220}) holds true.
\end{enumerate}
\label{thm_main2}
\end{theorem}

We call the $1$-Lipschitz function $u$
from Theorem \ref{thm_main2} the {\it guiding function} of the needle-decomposition.
In the case where the function $u$ from Theorem \ref{thm_main2} is the distance function from a smooth hypersurface, the
conclusion of Theorem \ref{thm_main2}
is essentially a classical computation in Riemannian geometry which may be found in Gromov \cite{G_levy, G_MS},
Heintze and Karcher \cite{HK} and Morgan \cite{morgan}. That computation is related to Paul Levy's proof
of the isoperimetric inequality.
It is beneficial to analyze arbitrary Lipschitz functions in Theorem \ref{thm_main2}, because
of the relation to
the dual Monge-Kantorovich problem presented in the following:

\begin{theorem}[``Localization theorem with a guiding function'']
Let $n \geq 2, \kappa \in \RR$ and $N \in (-\infty, 1) \cup [n, +\infty ]$.
Assume that $(\cM,d,\mu)$ is an $n$-dimensional weighted Riemannian manifold
of class $CD(\kappa, N)$ which is geodesically-convex.
Let $f: \cM \rightarrow \RR$ be a $\mu$-integrable function with $ \int_\cM f d \mu = 0$.
Assume that there exists a point $x_0 \in \cM$ with $\int_\cM |f(x)| \cdot  d(x_0, x) d \mu(x) < \infty$.
Then,
\begin{enumerate}
\item[(A)] There exists a $1$-Lipschitz function $u: \cM \rightarrow \RR$ such that
\begin{equation}  \int_{\cM} u f d \mu = \sup_{\| v \|_{Lip} \leq 1} \int_{\cM} v f d \mu. \label{eq_I2103} \end{equation}
\item[(B)] For any such function $u$, the function $f$ vanishes $\mu$-almost everywhere in
$\cM \setminus \Strain[u]$. Furthermore, let $\nu$ and $\{ \mu_{\cI} \}_{\cI \in T^{\circ}[u]}$ be measures on $T^{\circ}[u]$
and $\cM$, respectively, satisfying conclusions (i), (ii) and (iii) of Theorem \ref{thm_main2}.
Then for $\nu$-almost any $\cI \in T^{\circ}[u]$,
\begin{equation}  \int_{\cI} f d \mu_{\cI} = 0. \label{eq_I1002} \end{equation}
\item[(C)] For any such function $u$, there exist $\Omega, \nu, \{ \mu_{\cI} \}_{\cI \in \Omega}$ satisfying
the conclusions of Theorem \ref{thm_main}, which also satisfy the
following  property:
For $\nu$-almost any $\cI \in \Omega$, there exist a connected set $A \subseteq \RR$ and a minimizing geodesic
$\gamma: A \rightarrow \cM$ with $\gamma(A) = \cI$ and
$$ u(\gamma(t)) = t \qquad \qquad \qquad \text{for all} \ t \in A. $$
 \end{enumerate}
 \label{prop_intro}
\end{theorem}

Our manuscript owes much to previous investigations
of the Monge-Kantorovich problem. An integrable foliation by straight lines satisfying an analog of (\ref{eq_I1002})
was mentioned already by Monge in 1781, albeit on a heuristic level (see, e.g., Cayley's review of Monge's work \cite{cayley}).
The optimization problem  (\ref{eq_I2103})
entered the arena with the work of Kantorovich \cite[Section VIII.4]{KA}.

\medskip
An analytic resolution of the Monge-Kantorovich  problem which
is satisfactory for our needs is provided by
Evans and Gangbo \cite{EG}, with subsequent developments  by Ambrosio \cite{amb}, Caffarelli, Feldman and McCann \cite{CFM}, Feldman and McCann \cite{FM}
and Trudinger and Wang \cite{TW}.
Ideas from these papers have helped us in dealing with
the following  difficulty:  We
are obliged to work with the second fundamental form of the level set $\{ u = t_0 \}$ in order
to use the Ricci curvature and conclude that $\mu_{\cI}$ is a $CD(\kappa, N)$-needle.
However, the function $u$ is an arbitrary Lipschitz function, and it is not entirely clear
how to interpret its Hessian. Section \ref{sec_reg} is devoted to overcoming this difficulty,
 by showing that inside the set $\Strain[u]$ the function $u$
behaves as if it were a $C^{1,1}$-function.
The conditioning of $\mu$
with respect to the partition $T^{\circ}[u]$ is discussed in Section
\ref{sec_condition}, in which we prove Theorem \ref{thm_main2}.
Section \ref{sec_monge} is dedicated to the proofs of Theorem \ref{thm_main}
and Theorem \ref{prop_intro}.

\medskip Throughout this note, by a smooth function or manifold we always mean $C^{\infty}$-smooth.
All differentiable manifolds are assumed smooth and all of our Riemannian manifolds have smooth metric tensors. We do not consider
Riemannian manifolds with a boundary.
When we mention a measure $\nu$ on a set $X$ we implicitly consider
a $\sigma$-algebra of $\nu$-measurable subsets of $X$. All
of our measures in this paper are {\it complete}, meaning that if $\nu(A) = 0$
and $B \subseteq A$, then $B$ is $\nu$-measurable.
When we push-forward the measure $\nu$, we implicitly also push-forward
its $\sigma$-algebra.
Note that the concept of a Lebesgue-measurable subset of a differentiable manifold is well-defined
(e.g., Section \ref{lip_sec} below). When we write ``a measurable set'', without any reference to a specific measure, we simply mean Lebesgue-measurable.
We write $\log$  for the natural logarithm.

\medskip
{\it Acknowledgements.}
I would like to thank Emanuel Milman for introducing me to the subject of
Riemannian manifolds with lower bounds on their Ricci curvature.
Supported by a grant from the European Research Council.

\section{Regularity of geodesic foliations}
\label{sec_reg}
\setcounter{equation}{0}

\subsection{Transport rays}
\label{transport_rays}
\setcounter{equation}{0}

Let $\cM$ be an $n$-dimensional Riemannian manifold
which is geodesically-convex
and let $d$ be the Riemannian distance function on $\cM$.
As before, a curve $\gamma: I \rightarrow \cM$ is a minimizing geodesic
if $I \subseteq \RR$ is a connected subset and
$$ d(\gamma(s), \gamma(t)) = d(s,t) \qquad \qquad \qquad \text{for all } \ s,t \in I. $$
A curve $\gamma: J \rightarrow \cM$
is a {\it geodesic} if $J \subseteq \RR$ is connected, and for any $x \in J$ there exists a relatively-open subset $I \subseteq J$
containing $x$ such that $\gamma|_I$ is a minimizing geodesic. Thus, we only discuss geodesics of speed one, and not
of arbitrary speed as is customary. For the basic concepts in Riemannian geometry that we use here we refer the reader, e.g., to the first
ten pages of Cheeger and Ebin \cite{CE}. In particular, it is well-known
that all geodesic curves are smooth, and that for $p \in \cM$ and a unit vector $v \in T_{p} \cM$ there is a unique geodesic curve $\gamma_{p, v}$ with $\gamma_{p,v}(0) = p$
and $\dot{\gamma}_{p, v}(0) = v$. Let $I_{p,v} \subseteq \RR$ be the maximal set on which $\gamma_{p,v}$ is well-defined, which is an open, connected set
containing zero.
Denote
$$ \exp_{p}(tv) = \gamma_{p,v}(t) \qquad \qquad \qquad \text{for} \ t \in I_{p,v}. $$
The exponential map $\exp_{p}: T_{p} \cM \rightarrow \cM$ is a partially-defined function,
which is
well-defined
and smooth on an open subset of $T_{p} \cM$ containing the origin.

\begin{lemma} Let $A \subseteq \RR$ be an arbitrary subset, and let $\gamma: A \rightarrow \cM$ satisfy
\begin{equation} d(\gamma(s), \gamma(t)) = |s-t| \qquad \qquad \qquad \text{for all} \ s,t \in A.
\label{eq_1401} \end{equation}
Denote $conv(A) = \left \{ \lambda t + (1-\lambda) s \, ; \, s,t \in A, 0 \leq \lambda \leq 1 \right \}$.
Then there exists a minimizing geodesic $\tilde{\gamma}: conv(A) \rightarrow \cM$ with $\tilde{\gamma}|_A = \gamma$.
\label{lem_1522}
\end{lemma}

\begin{proof} We may assume that $\#(A) \geq 3$, because if $A$ contains only  two points
then we may connect them by a minimizing geodesic. Fix $s\in A$ with $\inf A < s < \sup A$.
According to (\ref{eq_1401}),  for any $r,t \in A$ with $r < s < t$,
\begin{equation}  d(\gamma(r), \gamma(s)) + d(\gamma(s), \gamma(t)) = d(\gamma(r), \gamma(t)). \label{eq_1502__} \end{equation}
Denote $a = \gamma(r), b = \gamma(s), c = \gamma(t)$.  Select any minimizing geodesic $\gamma_1$
from $a$ to $b$, and any minimizing geodesic $\gamma_2$ from $b$ to $c$.
We claim that $\gamma_1$ and $\gamma_2$ make a zero angle at the point $b$.
Indeed by (\ref{eq_1502__}), the concatenation
of the curves $\gamma_1$ and $\gamma_2$ forms a minimizing geodesic from $a$ to $c$,
which is necessarily smooth, hence the curves $\gamma_1$ and $\gamma_2$ must fit together at the point $b$.
We conclude that there exists a unit vector $v \in T_{\gamma(s)} \cM$, such that for any $x \in A \setminus \{ s \}$,
the vector $\sgn(x-s) v$ is tangent to any
 minimizing geodesic
from $\gamma(s)$ to $\gamma(x)$. Here, $\sgn(x)$ is the sign of $x \in \RR \setminus \{ 0 \}$.
Denote
$$ \tilde{\gamma}(x) = \exp_{\gamma(s)} ( (x - s) v ). $$
Then $\tilde{\gamma}$ is the geodesic emanating from $\gamma(s)$ in the direction of $v$, and
it satisfies $\tilde{\gamma}(x) = \gamma(x)$ for any $x \in A$. The geodesic curve $\tilde{\gamma}$ is thus
 well-defined on the interval $conv(A)$, with $\tilde{\gamma}|_A = \gamma$. Furthermore, it follows from (\ref{eq_1401}) that the geodesic $\tilde{\gamma}: conv(A) \rightarrow
 \cM$ is a minimizing geodesic, and the lemma is proven.
 \end{proof}

The following definition was proposed by Evans and Gangbo \cite{EGang} who worked under the assumption
that $\cM$ is a Euclidean space, see Feldman and McCann \cite{FM} for the generalization to complete Riemannian manifolds.

\begin{definition} Let $u: \cM \rightarrow \RR$ be a function with $\| u \|_{Lip} \leq 1$.
A subset $\cI \subseteq \cM$ is a ``transport ray'' associated with $u$ if
\begin{equation}  |u(x) - u(y)| = d(x,y) \qquad \qquad \qquad \text{for all} \, x, y \in \cI
\label{eq_1141} \end{equation}
and if for any $\cJ \supsetneq \cI$ there exist
$x,y \in \cJ$ with $|u(x) - u(y)| \neq d(x,y)$. In other words, $\cI$ is a maximal set that satisfies condition (\ref{eq_1141}).
We write $T[u]$ for the collection of all transport rays associated with $u$. \label{def_1050}
\end{definition}

By continuity, the closure of a transport ray is also a transport ray, and by maximality
any transport ray is a closed set. By Zorn's lemma, any subset $\cI \subseteq \cM$ satisfying (\ref{eq_1141})
is contained in a certain transport ray.
For the rest of this subsection, we fix a function  $u: \cM \rightarrow \RR$ with $\| u \|_{Lip} \leq 1$.
The following lemma shows that transport rays are geodesic arcs in $\cM$ on which $u$ grows at speed one.
For a map $F$ defined on a set $A$ we write $F(A) = \{ F(x) \, ; \, x \in A \}$.

\begin{lemma}
Any $\cJ \in T[u]$ is the image of a minimizing geodesic $\gamma: A \rightarrow \cM$,
where $A = u(\cJ)$ is a connected set in $\RR$,
and we have
\begin{equation}
u(\gamma(t)) = t \qquad \qquad \qquad \text{for} \ t \in A.
\label{eq_A1717}
\end{equation}
\label{lem_1102}
\end{lemma}

\begin{proof} Denote $A = u(\cJ) \subseteq \RR$. From (\ref{eq_1141}) the map $u: \cJ \rightarrow A$ is invertible.
By defining $\gamma(u(x)) = x$ for $x \in \cJ$, we see from (\ref{eq_1141}) that
\begin{equation}  d(\gamma(s), \gamma(t)) = |s-t| \qquad \qquad \qquad
\text{for any} \ s,t \in A. \label{eq_1015} \end{equation}
We may apply Lemma \ref{lem_1522} in view of (\ref{eq_1015}), and conclude that $\gamma$ may be extended to a curve $\tilde{\gamma}:conv(A) \rightarrow \cM$
which is a minimizing geodesic. Furthermore, since $\| u \|_{Lip} \leq 1$ with $u(\gamma(t)) = t$ for $t \in A$, then
necessarily
\begin{equation}
u(\tilde{\gamma}(t)) = t \qquad \qquad \qquad \text{for} \ t \in conv(A). \label{eq_A1049}
\end{equation}
The curve $\tilde{\gamma}$ is a minimizing geodesic, and its image $\cI = \tilde{\gamma}(conv(A))$ satisfies (\ref{eq_1141}),
thanks to (\ref{eq_A1049}). The maximality property of $\cJ$ entails  that $\cI = \cJ$ and $A = conv(A)$. Consequently $\cJ$ is
the image of the minimizing geodesic $\gamma \equiv \tilde{\gamma}$, and (\ref{eq_A1717}) follows from (\ref{eq_A1049}).
\end{proof}

Lemma \ref{lem_1102} states that we may identify between a transport ray $\cI \subseteq \cM$
and the image of a certain minimizing geodesic $\gamma: A \rightarrow \cM$. When we write that a unit vector $v \in T \cM$
is tangent to $\cI$ we mean that $v = \dot{\gamma}(t)$ for some $t \in A$.
We say that
$$ \{ \gamma(t) \, ; \, t \in int(A) \} $$ is the {\it relative interior} of the transport ray $\cI$,
where $int(A) \subseteq \RR$ is the interior of the set $A \subseteq \RR$.
Note that a transport ray $\cI$ could be a singleton, and then its relative interior turns out to be empty.
The set $$ \{ \gamma(t) \, ; \, t \in A \setminus int(A) \} $$ is defined to be the
{\it relative boundary} of the transport ray $\cI$.
Since $A \subseteq \RR$ is connected, then the relative boundary of any transport ray contains at most two points.
The short proof of the following lemma appears in  Feldman and McCann \cite[Lemma 10]{FM}:

\begin{lemma}
For any transport ray $\cI \in T[u]$ and a point $x$ in the relative interior of $\cI$, the function  $u$ is differentiable
at $x$, and $\nabla u(x)$ is a unit vector tangent to  $\cI$.
\label{lem_1046}
\end{lemma}

In this subsection we define
the set $\Strain[u] \subseteq \cM$ to be the union of all relative interiors of transport rays associated with $u$.
Very soon  we will show that  this definition, in fact, coincides with the definition of $\Strain[u]$ provided in Section \ref{sec_intro}.

\begin{lemma} For any $x \in \Strain[u]$ there exists a unique $\cI \in T[u]$ such that $x \in \cI$.
Furthermore, $x$ belongs to the relative interior of $\cI$.
\label{lem_A1031}
\end{lemma}

\begin{proof} From Lemma \ref{lem_1046} we know that $u$ is differentiable at $x$ and that $\nabla u(x)$ is a unit vector.
Consider the geodesic
\begin{equation}  \tilde{\gamma}(t) = \exp_{x} (t \nabla u(x))  \label{eq_B919_} \end{equation}
which is well-defined in a maximal subset $(a,b) \subseteq \RR$ containing zero. Define
\begin{equation}  A = \{ t \in (a,b) \, ; \, u(\tilde{\gamma}(t)) = u(x) + t \}. \label{eq_A1234_} \end{equation}
Note that $0 \in A$. Since $\tilde{\gamma}$ is a geodesic and $\| u \|_{Lip} \leq 1$,
then $A$ is necessarily connected
and $\tilde{\gamma}: A \rightarrow \cM$ is a minimizing geodesic.
In fact, by (\ref{eq_A1234_}) the set $\tilde{\gamma}(A)$ is contained in a certain transport ray.

\medskip We will show that $\tilde{\gamma}(A)$ is the unique transport ray containing $x$.
Indeed, $x \in \Strain[u]$ and hence there exists
 $\cI \in T[u]$ with $x \in \cI$. Since $x$ is contained in the relative interior
 of a certain  transport ray, then $\cI$ is not a singleton by the maximality property
of transport rays.
Note that $\nabla u(x)$ is necessarily tangent to $\cI$: this follows
from equation (\ref{eq_A1717}) of Lemma
\ref{lem_1102} and from the fact that $\nabla u(x)$ is a unit vector.
We conclude
from (\ref{eq_B919_}), (\ref{eq_A1234_}) and Lemma \ref{lem_1102} that $\cI \subseteq \tilde{\gamma}(A)$. However, we said earlier that $\tilde{\gamma}(A)$ is contained
in a  transport ray, and by maximality $\cI = \tilde{\gamma}(A)$. Therefore $\tilde{\gamma}(A)$ is the unique transport ray containing $x$.
Since $x \in \Strain[u]$ then the point $x$ necessarily belongs to the relative interior of the transport ray $\tilde{\gamma}(A)$.
\end{proof}

For a point $y \in \Strain[u]$  define
$$ \alpha_u(y) = u(y) - \inf_{z \in \cJ} u(z), \qquad  \beta_u(y) = \left[ \sup_{z \in \cJ} u(z) \right]  - u(y), $$
where $\cJ \in T[u]$ is the unique transport ray containing $y$.
For $y \not \in \Strain[u]$ we set $\alpha_u(y) = \beta_u(y) = -\infty$.
Thus, the functions $\alpha_u, \beta_u$ are positive on $\Strain[u]$, and equal to $-\infty$ outside $\Strain[u]$.
Lemma \ref{lem_1102} and Lemma \ref{lem_1046} admit the following immediate corollary:

\begin{corollary} Let $y \in \Strain[u]$. Set $A = (-\alpha_u(y), \beta_u(y)) \subseteq \RR$.
Then there exists a minimizing geodesic $\gamma: A \rightarrow \cM$ whose image is the
relative interior of a transport ray, such that $\gamma(0) = y$ and for all $t \in A$,
$$u(\gamma(t)) = u(y) + t, \qquad \dot{\gamma}(t) = \nabla u(\gamma(t)). $$
\label{cor_A1217}
\end{corollary}
Recall that the set $\Strain[u] = \left \{ x \in \cM \, ; \, \alpha_u(x) > 0 \right \}
= \left \{ x \in \cM \, ; \, \beta_u(x) > 0 \right \}
$ was defined a bit differently in Section \ref{sec_intro}. The equivalence of the two definitions follows
from our next little lemma:

\begin{lemma} Let $y \in \cM$. Then $\alpha_u(y)$ equals the supremum over all $\eps > 0$ for which there exist $x, z \in \cM$ with
\begin{equation} d(x,y) = u(y) - u(x) \geq \eps, \quad
d(y,z) = u(z) - u(y) > 0, \quad
d(x,y) + d(y,z) = d(x,z).  \label{eq_1751}
\end{equation}
The supremum over an empty set is defined to be $-\infty$.
\label{lem_A1053}
\end{lemma}

\begin{proof} Write $\tilde{\alpha}_u(y)$ for the supremum over all $\eps > 0$ for which there exist $x, z \in \cM$
such that (\ref{eq_1751}) holds. We need to show that
\begin{equation} \alpha_u(y) = \tilde{\alpha}_u(y) \qquad \qquad \qquad \text{for all} \ y \in \cM.
\label{eq_A1225} \end{equation}
Corollary \ref{cor_A1217} implies that $\alpha_u(y) \leq \tilde{\alpha}_u(y)$ for any $y \in \Strain[u]$.
Clearly $\alpha_u(y) \leq \tilde{\alpha}_u(y)$ for any $y \not \in \Strain[u]$, since $\alpha_u(y) = -\infty$
for such $y$. It thus remains to prove the ``$\geq$'' inequality between the terms in (\ref{eq_A1225}).
To this end, we fix $y \in \cM$ for which $\tilde{\alpha}_u(y) > -\infty$.
Then there exist
 $x,z \in \cM$ satisfying  (\ref{eq_1751}) with some $\eps > 0$. The triplet $\cI = \{ x, y, z \}$ satisfies (\ref{eq_1141}).
 By Zorn's lemma, $\cI$ is contained in a transport ray $\cJ$, and the point $y$ must belong to the relative interior of $\cJ$ as $$ u(x) < u(y) < u(z). $$
 By Lemma \ref{lem_A1031}, the point $y$ does not belong to any transport ray other than $\cJ$. Additionally, any points
 $x, z \in \cM$ satisfying (\ref{eq_1751})
 must belong to the transport ray $\cJ$. It follows from Corollary \ref{cor_A1217} that
 $\tilde{\alpha}_u(y) \leq \alpha_u(y)$, and (\ref{eq_A1225}) is proven.
\end{proof}

A transport ray which is a singleton is called a {\it degenerate} transport ray.
According to Lemma \ref{lem_1102}, a transport ray $\cI \in T[u]$ is non-degenerate if and only
if its relative interior is non-empty.

\begin{lemma} The following relation is an equivalence relation on $\Strain[u]$:
\begin{equation} x \sim y \qquad \Longleftrightarrow \qquad |u(x) - u(y)| = d(x,y).
\label{eq_B1109}
\end{equation}
As in Section \ref{sec_intro}, we write $T^{\circ}[u]$ for the collection of all equivalence classes.
Then $T^{\circ}[u]$ is the collection of all relative interiors
of  non-degenerate transport rays.
\label{lem_A317}
\end{lemma}

\begin{proof} According to Lemma \ref{lem_A1031}, The collection of all relative interiors of non-degenerate transport rays is a partition of
$\Strain[u]$. Let $x,y \in \Strain[u]$. We need to show that $x \sim y$ if and only if $x$ and $y$ belong to the relative interior
of the same transport ray.

\medskip Assume first that  $x \sim y$. Then $\cI = \{ x, y \}$
satisfies (\ref{eq_1141}), and hence  there exists a transport ray $\cJ \in T[u]$
such that $x,y \in \cJ$. However, $x, y \in \Strain[u]$ and
$\cJ$ is a transport ray containing $x$ and $y$. From Lemma
\ref{lem_A1031} we conclude that $x$ and $y$ belong to the relative interior of $\cJ$.
Conversely, suppose that $x,y \in \Strain[u]$ belong to the relative interior of a certain
transport ray $\cJ \in T[u]$. By (\ref{eq_B1109}) and Definition \ref{def_1050}, we have $x \sim y$. The proof is complete.
\end{proof}

A $\sigma$-compact set
is a countable union of compact sets.
A topological space is second-countable
if its topology has a countable basis of open sets.
Note that any geodesically-convex, Riemannian manifold $\cM$
is second-countable: Indeed, since $\cM$ is a metric space, it suffices
to find a countable, dense subset. Fix $a \in \cM$ and a countable, dense subset of $T_a \cM$.
Since $\cM$ is geodesically-convex,  the image of the latter subset under $\exp_a$ is
a countable, dense subset of $\cM$. Therefore $\cM$ is second-countable, and any open cover
of any subset $S \subseteq \cM$ has a countable subcover. Since $\cM$ is locally-compact
and second-countable,  it is $\sigma$-compact.

\medskip Define $\ell_u(y) = \min \{ \alpha_u(y), \beta_u(y) \}$ for $y \in \cM$.
Then $\ell_u$ is positive on $\Strain[u]$, and it equals $-\infty$ outside $\Strain[u]$.

\begin{lemma} The functions $\alpha_u, \beta_u, \ell_u: \cM \rightarrow \RR \cup \{ \pm  \infty \}$ are Borel-measurable.
\label{lem_1055}
\end{lemma}

\begin{proof} We will only prove that $\alpha_u$ is Borel-measurable. The argument for $\beta_u$ is similar, while
$\ell_u$ is Borel-measurable as $\ell_u = \min \{ \alpha_u, \beta_u \}$.
For $\eps, \delta > 0$ we define $A_{\eps, \delta}$ to be the collection of all triplets $(x,y,z) \in \cM^3$ with
$$ d(x,y) = u(y) - u(x) \geq \eps,
\quad
d(y,z) = u(z) - u(y) \geq \delta,
\quad
d(x,y) + d(y,z) = d(x,z). $$
Then $A_{\eps, \delta}$ is a closed set, by the continuity of $u$ and of the distance function. The Riemannian manifold
$\cM$ is $\sigma$-compact, hence  there exist compacts $K_1 \subseteq K_2 \subseteq \ldots$ such that $\cM = \cup_i K_i$.
Define
$$ A_{i, \eps, \delta} = A_{\eps, \delta} \cap (K_i \times K_i \times K_i) \qquad \qquad \qquad (i \geq 1, \eps > 0, \delta > 0).  $$
Note that $A_{i, \eps, \delta}$ is compact and hence $\pi(A_{i, \eps, \delta})$ is also compact, where $\pi(x,y,z) = y$.
Clearly, $A_{\eps, \delta} = \cup_i A_{i, \eps, \delta}$.
Let $\alpha_{i, \eps, \delta}: \cM \rightarrow \RR \cup \{-\infty\}$ be the function that equals $\eps$ on the compact set
$\pi(A_{i, \eps, \delta})$ and equals $-\infty$ otherwise. Then $\alpha_{i, \eps, \delta}$ is a Borel-measurable function and by
Lemma \ref{lem_A1053}, for any $y \in \cM$,
$$ \alpha_u(y) =  \sup \left \{ \eps > 0 \, ; \, \exists \delta > 0, \, y \in \pi(A_{\eps, \delta}) \right \} = \sup \left \{ \alpha_{i, \eps, \delta}(y) \, ; \,  \eps, \delta \in \QQ \cap (0,\infty), i \geq 1 \right \}. $$
Hence $\alpha_u$ is the supremum of countably many Borel-measurable functions, and is thus necessarily Borel-measurable.
\end{proof}

For $\eps > 0$ denote
$ \Strain_{\eps}[u] = \left \{ x \in \cM \, ; \, \ell_u(x) > \eps \right \}$. Thus,
$$ \Strain[u] = \bigcup_{\eps > 0} \Strain_{\eps}[u] = \{ x \in \cM \, ; \, \ell_u(x) > 0 \}. $$
The function $u$ is basically an arbitrary Lipschitz function, yet the following theorem asserts higher regularity of $u$ inside
the set $\Strain[u]$. Denote $B_{\cM}(p, \delta) = \left \{ x \in \cM \, ; \, d(x,p) < \delta \right \}$.

\begin{theorem} Let $\cM$ be a geodesically-convex Riemannian manifold. Let $u: \cM \rightarrow \RR$ be a function with $\| u \|_{Lip} \leq 1$.
Let $p \in \cM, \eps_0 > 0$.
Then there exist
$\delta > 0$ and a $C^{1,1}$-function $\tilde{u}: B_{\cM}(p, \delta) \rightarrow \RR$
such that for any $x \in \cM$,
\begin{equation}  x \in
B_{\cM}(p, \delta) \cap \Strain_{\eps_0}[u] \qquad \Longrightarrow \qquad \tilde{u}(x) = u(x), \ \ \nabla \tilde{u}(x) = \nabla u(x). \label{eq_1921} \end{equation}
\label{prop_939}
\end{theorem}

Section \ref{sec_939} contains the standard background on $C^{1,1}$-functions.
In Section \ref{charts} we discuss the Riemann normal coordinates, and in Section \ref{proof_c11} we complete the proof
of Theorem \ref{prop_939}.
Our proof of Theorem \ref{prop_939}
is related to the arguments of Evans and Gangbo \cite{EG}
and to the contributions by Ambrosio \cite{amb}, Caffarelli, Feldman and McCann \cite{CFM},
Feldman and McCann \cite{FM} and Trudinger and Wang \cite{TW}.
The new ingredient in our analysis is the use
of Whitney's extension theorem.

\subsection{Whitney's extension theorem for $C^{1,1}$}
\label{sec_939}
\setcounter{equation}{0}

Given a function $f: \RR^n \rightarrow \RR$ we write $\partial_i f = \partial f / \partial x_i$ for its $i^{th}$ partial
derivative, so that $\nabla f = (\partial_1 f, \ldots, \partial_n f)$.
Denote by $| \cdot |$ the standard Euclidean
norm in $\RR^n$, and $x \cdot y$ is the usual scalar product of $x,y \in \RR^n$.
For an open, convex set $K \subseteq \RR^n$ and a $C^1$-function $\vphi = (\vphi_1,\ldots,\vphi_m): K \rightarrow \RR^m$  we set
\begin{equation}  \left \| \vphi \right \|_{C^{1,1}} = \sup_{x \in K} \left( |\vphi(x)| + \| \vphi^{\prime}(x) \|_{op} \right)
\, + \, \sup_{x \neq y \in K} \frac{\| \vphi^{\prime}(x) \, - \, \vphi^{\prime}(y) \|_{op}}{|x-y|},
\label{eq_1017} \end{equation}
where the derivative $\vphi^{\prime}(x)$ is an $m \times n$ matrix whose $(i,j)$-entry is $\partial_j \vphi_i(x)$,
and $$ \| A \|_{op} = \sup_{0 \neq v \in \RR^n} |Av|/|v| $$ is the operator norm.
Similarly, we may define the $C^{1,1}$-norm of a function $\vphi: K \rightarrow Y$,
where $X$ and $Y$ are  finite-dimensional linear spaces with inner products
and where $K \subseteq X$ is an open, convex set. In fact, formula (\ref{eq_1017}) remains valid in the latter scenario, yet in this case we need to interpret $\vphi^{\prime}(x)$
as  a linear map from $X$ to $Y$ and not as a matrix.
For an open set $U \subseteq \RR^n$, we say that $f: U \rightarrow \RR^m$ is a $C^{1,1}$-function if for any $x \in U$ there exists $\delta > 0$ such that
 $$ \left \| \left. f \right|_{B(x, \delta)} \right \|_{C^{1,1}} < \infty $$
 where $\left. f \right|_{B(x, \delta)}$ is the restriction of $f$ to the open ball $B(x, \delta) = \{ y \in \RR^n \, ; \, |y-x| < \delta \}$.
 In other words, a   $C^1$-function $f: U \rightarrow \RR^m$ is a $C^{1,1}$-function if and only if the derivative $f^{\prime}$ is
 a locally-Lipschitz map into the space of $m \times n$ matrices.
 Any $C^2$-function $f: U \rightarrow \RR^m$ is automatically a $C^{1,1}$-function.
 A  map $\vphi: U \rightarrow V$ is a $C^{1,1}$-diffeomorphism, for open sets $U,V \subseteq \RR^n$, if $\vphi$ is an invertible $C^{1,1}$-map and
 the inverse map $\vphi^{-1}: V \rightarrow U$ is also $C^{1,1}$.
 The $C^1$-version of the following lemma  may be found in any textbook on multivariate calculus.

\begin{lemma} \begin{enumerate}
\item[(i)] Let $U_1 \subseteq \RR^n$ and $U_2 \subseteq \RR^m$ be open sets. Let $f_2: U_2 \rightarrow \RR^k$ and $f_1: U_1 \rightarrow U_2$
be $C^{1,1}$-functions. Then  $f_2 \circ f_1$ is also a $C^{1,1}$-function.
\item[(ii)] Let $U \subseteq \RR^n$ be an open set and let $f: U \rightarrow \RR^n$ be a $C^{1,1}$-function. Assume that $x_0 \in U$ is such that
$\det f^{\prime}(x_0) \neq 0$. Then there exists $\delta > 0$ such that $f|_{B(x_0,\delta)}$ is a $C^{1,1}$-diffeomorphism onto
some open set $V \subseteq \RR^n$.
\item[(iii)] Let $U \subseteq \RR^n$ be an open set and let $f: U \rightarrow \RR$ be a $C^{1,1}$-function. Assume that $x_0 \in U$ is such that
$\nabla f(x_0) \neq 0$. Then there exists an open set $V \subseteq U$ containing the point $x_0$,
an open set $\Omega \subseteq \RR^{n-1} \times \RR$ of the form $\Omega = \Omega_0 \times (a,b) \subseteq \RR^{n-1} \times \RR$ and a $C^{1,1}$-diffeomorphism
$G: \Omega \rightarrow V$ such that for any $(y,t) \in \Omega$,
$$ f(G(y,t)) = t. $$
\end{enumerate}
\label{lem_1345}
\end{lemma}

\begin{proof} \begin{enumerate} \item[(i)] We know that $h = f_2 \circ f_1$ is a $C^1$-function.
The map $x \mapsto f_2^{\prime}(f_1(x))$ is locally-Lipschitz, since it is the composition of two locally-Lipschitz maps.
Since $f_1^{\prime}$ is locally-Lipschitz,  the product $h^{\prime}(x) = f_2^{\prime}(f_1(x)) \cdot f_1^{\prime}(x)$ is also locally-Lipschitz. Hence $h$ is a $C^{1,1}$-function.
\item[(ii)] The usual inverse function theorem for $C^1$ guarantees the existence of $\delta > 0$ and an open set $V \subseteq \RR^n$ such that $f: B(x_0,\delta) \rightarrow V$ is a $C^1$-diffeomorphism. Let $g: V \rightarrow B(x_0, \delta)$ be the inverse map.
    The map $g^{\prime}(x) = \left( f^{\prime}(g(x)) \right)^{-1}$ is the composition of three locally-Lipschitz maps, hence it is locally-Lipschitz and $g$ is  $C^{1,1}$.
    \item[(iii)] This follows from (ii) in exactly the same way that the implicit function theorem follows from the inverse function theorem in the $C^1$ case, see e.g. Edwards
    \cite[Chapter III.3]{edwards}. \qedhere
    \end{enumerate}
\end{proof}

Lemma \ref{lem_1345}(i) shows that the concept of a $C^{1,1}$-function on a differentiable manifold
is well-defined:

\begin{definition} Let $\cM$ and $\cN$ be differentiable manifolds. A function $f: \cM \rightarrow \cN$
is a $C^{1,1}$-function if $f$ is $C^{1,1}$ in any local chart. A $C^{1,1}$-function $f: \cM \rightarrow \cN$ is a $C^{1,1}$-diffeomorphism if it is invertible
and the inverse function $f^{-1}: \cN \rightarrow \cM$ is also $C^{1,1}$.
\label{def_1741}
\end{definition}

Let $K \subseteq \RR^n$ be an open, convex set  and  let $f: K \rightarrow \RR$ satisfy $ M := \| f \|_{C^{1,1}} < \infty$.
It follows from the definition (\ref{eq_1017}) that for $x,y \in K$,
\begin{equation} |\nabla f(x) - \nabla f(y)| \leq M |x-y|. \label{eq_1057}
\end{equation}
For $x, y \in K$ we also have, denoting $x_t = (1-t) x + t y$,
\begin{equation}
\left| f(x) + \nabla f(x) \cdot (y - x) - f(y) \right|
= \left| \int_0^1 \left[ \nabla f( x ) - \nabla f(x_t) \right] \cdot (y-x) dt \right| \leq \frac{M}{2} |x-y|^2.
\label{eq_1042}
\end{equation}
Conditions (\ref{eq_1057}) and (\ref{eq_1042}), which are basically  Taylor's theorem for $C^{1,1}$-functions,
capture the essence of
the concept of a $C^{1,1}$-function, as is demonstrated in Theorem \ref{thm_1052} below.
For points $x,y \in \RR^n$ and for $f: \{ x, y \} \rightarrow \RR$ and $V: \{ x, y \} \rightarrow \RR^n$
we define $\| (f, V) \|_{x,y}$ to be the infimum over all $M \geq 0$ for which the following three conditions hold:
\begin{enumerate}
\item[(i)] $\displaystyle |f(x)| \leq M, \, |V(x)| \leq M$,
\item[(ii)] $\displaystyle \left| V(y) - V(x) \right| \leq M |y-x|$,
\item[(iii)] $\displaystyle \left| f(x) + V(x) \cdot (y-x) -  f(y) \right| \leq M |y-x|^2$.
\end{enumerate}
This infimum is in fact a minimum. Note that $\| (f, V) \|_{x,y}$ is not
necessarily the same as $\| (f, V) \|_{y,x}$.

\begin{theorem}[Whitney's extension theorem for $C^{1,1}$]
Let $A \subseteq \RR^n$ be an arbitrary set, let $f: A \rightarrow \RR$ and $V: A \rightarrow \RR^n$. Assume that
\begin{equation} \sup_{x, y \in A} \left \| (f, V) \right \|_{x,y} < \infty. \label{eq_goth}
\end{equation}
Then there exists a $C^{1,1}$-function $\tilde{f}: \RR^n \rightarrow \RR$ such that for any $x \in A$,
$$ \tilde{f}(x) = f(x), \quad \nabla \tilde{f}(x) = V(x). $$
\label{thm_1052}
\end{theorem}

For a proof of Theorem \ref{thm_1052} see Stein \cite[Chapter VI.2.3]{stein}
or the original paper by Whitney \cite{W}.
Whitney's theorem is usually stated under the additional assumption that $A \subseteq \RR^n$
is a closed set, but it is straightforward to extend $f$ and $V$ from $A$ to the closure $\overline{A}$ by continuity,
preserving the validity of assumption (\ref{eq_goth}).

\medskip
Given a differentiable manifold $\cM$ and a subset $A \subseteq \cM$, a $1$-form on $A$ is a map $\omega: A \rightarrow T^* \cM$
with $\omega(x) \in T_x^* \cM$ for $x \in A$.
Let $\cM, \cN$ be differentiable manifolds and let $\vphi: \cM \rightarrow \cN$ be a $C^1$-map.
For a $1$-form $\omega$ on $A \subseteq \cN$ we write $\vphi^* \omega$ for the pull-back of $\omega$
under the map $\vphi$. Thus $\vphi^* \omega$ is a $1$-form on $\vphi^{-1}(A)$.
Write $\RR^{n*}$ for the space of all linear functionals from $\RR^n$ to $\RR$.
With any $\ell \in \RR^{n*}$
we associate the vector $V_{\ell} \in \RR^n$ which satisfies
$$ \ell(x) = x \cdot V_{\ell} \qquad \qquad \text{for any} \ x \in \RR^n. $$
Since $T_x^* (\RR^n)$ is canonically isomorphic to $\RR^{n*}$, any $1$-form $\omega$ on a subset $A \subseteq \RR^n$ may be identified with a map $\omega: A \rightarrow \RR^{n*}$.
Defining $V_{\omega}(x) := V_{\omega(x)} \in \RR^n$ we recall the formula
\begin{equation}
V_{\vphi^* \omega}(x) =  \vphi^{\prime}(x)^{*} \cdot V_{\omega}(\vphi(x)),
\label{eq_1821_} \end{equation}
where $B^{*}$ is the transpose of the matrix $B$. Here,  $\omega$ is a $1$-form on a subset $A \subseteq \RR^m$,
the function $\vphi$ is a $C^1$-map from an open set $U \subseteq \RR^n$ to $\RR^m$, and the formula (\ref{eq_1821_})
is valid for any $x \in \vphi^{-1}(A)$.
For $x, y \in \RR^n$ and for $f: \{ x, y \} \rightarrow \RR, \omega: \{ x, y \} \rightarrow \RR^{n*}$ we define
$$ \left \| (f, \omega) \right \|_{x,y} =
\left \| (f, V_\omega) \right \|_{x,y}. $$

\begin{lemma} Let $K_1, K_2 \subseteq \RR^n$ be open, convex sets. Let $R \geq 1$ and let $\vphi: K_1  \rightarrow K_2$ be a $C^1$-diffeomorphism with
\begin{equation}  \| \vphi^{-1} \|_{C^{1,1}} \leq R. \label{eq_937} \end{equation}
Let $x, y \in K_2$, denote $A = \{ x, y \}$, let $f: A \rightarrow \RR$, and let $\omega: A \rightarrow \RR^{n*}$ be a $1$-form on $A$. Denote
$\tilde{A} = \vphi^{-1}(A), \tilde{\omega}
= \vphi^* \omega, \tilde{f} = f \circ \vphi$, and $\tilde{x} = \vphi^{-1}(x), \tilde{y} = \vphi^{-1}(y)$.
Then,
$$ \left \| (f, \omega) \right \|_{x,y} \leq C_{n,R} \left \| (\tilde{f}, \tilde{\omega}) \right \|_{\tilde{x}, \tilde{y}}, $$
where $C_{n,R} > 0$ is a constant depending solely on $n$ and $R$. \label{lem_1245}
\end{lemma}

\begin{proof} It follows from (\ref{eq_1017}), (\ref{eq_937}) and the convexity of $K_2$ that the map $\psi := \vphi^{-1}$ is $R$-Lipschitz. Thus,
\begin{equation}
|\tilde{y} - \tilde{x}| = |\psi(y) - \psi(x)| \leq R |y-x|. \label{eq_957_}
\end{equation}
Set $V = V_{\omega}: A \rightarrow \RR^n$ and $\tilde{V} = V_{\tilde{\omega}}: \tilde{A} \rightarrow \RR^n$.
Since $\tilde{\omega} = \vphi^* \omega$ then $\omega = \psi^* \tilde{\omega}$ and from (\ref{eq_1821_}),
$$ V(x) = \psi^{\prime}(x)^* \cdot \tilde{V}({\tilde{x}}). $$
Denote $M =
\| (\tilde{f}, \tilde{\omega})  \|_{\tilde{x}, \tilde{y}}
= \| (\tilde{f}, \tilde{V})  \|_{\tilde{x}, \tilde{y}}$.
It suffices to show that
$f$ and $V$ satisfy conditions (i), (ii) and (iii) from the definition of $\| (f, V) \|_{x,y}$ with $M$ replaced by $2M (R^2 + nR + 1)$.
To that end, observe that
\begin{equation}
|f(x)| = |\tilde{f}(\tilde{x})| \leq M, \qquad |V(x)| = |\psi^{\prime}(x)^* \cdot \tilde{V}(\tilde{x})| \leq M R.
\label{eq_926}
\end{equation}
Thus condition (i) is satisfied. To prove condition (ii), we  compute that
\begin{align} \label{eq_2155}
|V(y) &- V(x)|  = \left| \psi^{\prime}(y)^* \tilde{V}(\tilde{y}) - \psi^{\prime}(x)^* \tilde{V}(\tilde{x})
\right| \\ & \leq \left| \psi^{\prime}(y)^*  (\tilde{V}(\tilde{y}) - \tilde{V}(\tilde{x}))
\right| + \left| ( \psi^{\prime}(y)^* - \psi^{\prime}(x)^* ) \tilde{V}(\tilde{x})
\right| \leq R M \left( |\tilde{y} - \tilde{x}| + |y-x| \right). \nonumber
\end{align}
Condition (ii) holds in view of (\ref{eq_957_}) and (\ref{eq_2155}).
Denote $\psi = (\psi_1,\ldots,\psi_n)$. From (\ref{eq_1042}) and (\ref{eq_957_}), \begin{align*}
| & f(x)  + V(x) \cdot (y-x) - f(y) | = | \tilde{f}(\tilde{x}) + \psi^{\prime}(x)^* \tilde{V}(\tilde{x})\cdot (y-x) - \tilde{f}(\tilde{y}) |
\\ & \leq | \tilde{f}(\tilde{x}) + \tilde{V}(\tilde{x}) \cdot (\tilde{y} - \tilde{x}) - \tilde{f}(\tilde{y}) | \, + \,
| \tilde{V}(\tilde{x}) |  \cdot \left|  \psi^{\prime}(x) (y - x) - (\psi(y) - \psi(x))  \right| \\ & \leq M |\tilde{x} - \tilde{y}|^2 +
M \sum_{i=1}^n \left| \nabla \psi_i(x) \cdot (y - x) - ( \psi_i(y) - \psi_i(x) ) \right| \leq (M R^2 + n M R) |y-x|^2.
\end{align*}
Condition (iii) is thus satisfied and the lemma is proven.
\end{proof}

\begin{corollary} Let $\cM$ be an $n$-dimensional  differentiable manifold, let $R \geq 1$ and let $U \subseteq \cM$ be an open set.
Assume that for any $a \in U$ we are given a convex, open set $U_a \subseteq \RR^n$ and a $C^{1,1}$-diffeomorphism $\vphi_a: U_a \rightarrow U$.
Suppose that for any $a, b \in U$,
\begin{equation} \| \vphi_b^{-1} \circ \vphi_a \|_{C^{1,1}} \leq R.
\label{eq_1243} \end{equation}

\medskip \noindent
Let $A \subseteq U$. Let $f: A \rightarrow \RR$ and let $\omega$ be a $1$-form on $A$.
For $a \in U$ set $f_a = f \circ \vphi_a$ and $w_a = \vphi_a^* w$.
Suppose that for any $x,y \in A$ there exists $a \in U$ for which
\begin{equation}  \| (f_a, \omega_a) \|_{\vphi_a^{-1}(x), \vphi_a^{-1}(y)} \leq R. \label{eq_1244} \end{equation}
Then there exists a $C^{1,1}$-function $\tilde{f}: U \rightarrow \RR$ with
\begin{equation}  \tilde{f} |_A = f, \quad  d \tilde{f} |_A = \omega, \label{eq_1257} \end{equation}
where $d \tilde{f}$ is the differential of the function $\tilde{f}$. \label{cor_1320}
\end{corollary}

\begin{proof}
Fix $b \in U$ and denote $A_b = \vphi_b^{-1}(A) \subseteq U_b \subseteq \RR^n$.
Abbreviate $\vphi_{b, a} = \vphi_a^{-1} \circ \vphi_b$. Let $x, y \in A_{b} \subseteq \RR^n$. According to (\ref{eq_1244}) there exists
$a \in U$ for which
\begin{equation}  \| (f_a, \omega_a) \|_{\vphi_{b,a}(x), \vphi_{b,a}(y)} \leq R. \label{eq_1249} \end{equation}
We may apply Lemma \ref{lem_1245}, thanks to (\ref{eq_1243}) and (\ref{eq_1249}), and conclude that for any $x,y \in A_b$,
\begin{equation}  \| (f_{b}, \omega_{b}) \|_{x,y} \leq C_{n,R}, \label{eq_1247} \end{equation}
for some $C_{n, R} > 0$ depending only on $n$ and $R$. Recall that for any linear functional $\ell \in \RR^{n*}$ there corresponds
a vector $V_{\ell} \in \RR^n$ defined via
$$ \ell(z) = V_{\ell} \cdot z \qquad \qquad \qquad (z \in \RR^n). $$
In particular, for $x \in A_b$
we have $\omega_b(x) \in \RR^{n*}$ and let us set $ V_b(x) := V_{\omega_b(x)} \in \RR^n$.
According to (\ref{eq_1247}),
the function $f_b: A_b \rightarrow \RR$ and the vector field $V_b : A_b \rightarrow \RR^n$
satisfy
$$ \sup_{x,y \in A_b} \left \| (f_b, V_b) \right \|_{x,y} \leq C_{n,R} < \infty. $$
Theorem \ref{thm_1052} thus produces a $C^{1,1}$-function $\tilde{f}_b: U_b \rightarrow \RR$ with
$$ \tilde{f}_b(x) = f_b(x), \quad \nabla \tilde{f}_b(x) = V_b(x) \qquad \qquad \qquad (x \in A_b). $$
In particular $ d \tilde{f}_b |_{A_b} = \omega_b$. Setting $\tilde{f}(x) = \tilde{f}_b(\vphi_b^{-1}(x))$ for $x \in U$, we obtain a function $\tilde{f}: U \rightarrow \RR$
satisfying (\ref{eq_1257}). The function $\tilde{f}$ is a $C^{1,1}$-function since it is the composition of two $C^{1,1}$-functions.
\end{proof}

\begin{remark}{\rm Corollary \ref{cor_1320} admits the following formal generalization:
Rather than stipulating that $U_a$ is a subset of $\RR^n$ for any $a \in U$, we may assume
that $U_a \subseteq X_a$, where $X_a$ is an $n$-dimensional linear space with an inner product.
This generalization is completely straightforward, and it does not involve any substantial modifications to neither the
formulation
nor the proof of Corollary \ref{cor_1320}.
}\label{rem_1320} \end{remark}

\subsection{Riemann normal coordinates}
\label{charts}
\setcounter{equation}{0}

 Let $\cM$ be an $n$-dimensional Riemannian manifold with  Riemannian distance function $d$.
 For $a \in \cM$ we write $\langle \cdot, \cdot \rangle$ for
 the Riemannian scalar product in $T_a \cM$, and $| \cdot |$ is the norm induced by this scalar product.
 Given a $C^2$-function $g: T_a \cM \rightarrow \RR$ and a point $X \in T_a \cM$ we
 may speak of the gradient $\nabla g(X) \in T_a \cM$ and of the
 Hessian operator $\nabla^2 g(X): T_a \cM \rightarrow
 T_a \cM$, which is a symmetric operator such that
 \begin{equation}  g(Y) = g(X) + \langle \nabla g(X), Y - X \rangle + \frac{1}{2} \left \langle \nabla^2 g(X) (Y - X), Y - X \right \rangle + o(| Y - X|^2).
 \label{eq_A1817_} \end{equation}
On a very formal level, since $T_a \cM$ is a linear space,  we canonically identify $T_X (T_a \cM) \cong T_a \cM$
for any $X \in T_a \cM$.  Therefore the gradient $\nabla g(X)$ belongs to $T_a \cM \cong T_X (T_a \cM)$.
 A subset $U \subseteq \cM$  is {\it strongly convex} if for any two points $x,y \in U$
there exists a {\it unique} minimizing geodesic in $\cM$ that connects $x$ and $y$,
and furthermore this minimizing geodesic is contained in $U$, while there are no other geodesic curves contained
in $U$ that join $x$ and $y$. See, e.g., Chavel \cite[Section IX.6]{Ch} for more information.
The following standard lemma expresses the fact that a Riemannian manifold
is ``locally-Euclidean''.

\begin{lemma} Let $\cM$ be a Riemannian manifold and let $p \in \cM$.
Then there exists $\delta_0 = \delta_0(p) > 0$
such that the following hold:
\begin{enumerate}
\item[(i)] For any $x \in B_{\cM}(p, \delta_0)$ and $0 < \delta \leq \delta_0$, the
ball $B_{\cM}(x, \delta)$ is strongly convex and its closure is compact.
\item[(ii)] Denote $U = B_{\cM}(p, \delta_0 / 2)$ and for $a \in U$ set $U_a = \exp_a^{-1}(U)$.
Then $U_a \subseteq T_a \cM$ is a bounded, open set
and $\exp_a$ is a smooth diffeomorphism between $U_a$ and $U$.
\item[(iii)] Define $ f_{a, X}(Y) = \frac{1}{2} \cdot d^2( \exp_a X, \exp_a Y)$ for $a \in U, X,Y \in U_a$.
Then $f_{a,X}: U_a \rightarrow \RR$ is a smooth function, and its Hessian operator $\nabla^2 f_{a,X}$ satisfies
\begin{equation}  \frac{1}{2} \cdot Id \leq \nabla^2 f_{a,X}(Y) \leq 2 \cdot Id \qquad \qquad \qquad (a \in U, X,Y \in U_a),
\label{eq_1030} \end{equation}
in the sense of symmetric operators, where $Id$ is the identity operator.
\item[(iv)] For any $a,x \in U$ and $0 < \delta \leq \delta_0$, the
set $\exp_a^{-1} \left( B_{\cM}(x, \delta)  \right)$ is a convex subset of $T_a \cM$. In particular, $U_a$ is convex.
\item[(v)] For any $a \in U, X, Y \in U_a$,
$$  \frac{1}{2} \cdot |X - Y| \leq d(\exp_a X, \exp_a Y) \leq 2 \cdot |X-Y|. $$
\item[(vi)] For $a, b \in U$ consider the transition map $\vphi_{a,b}: U_a \rightarrow U_b$
defined by $\vphi_{a,b} = \exp_b^{-1} \circ \exp_a$. Then,
\begin{equation} \sup_{a,b \in U} \| \vphi_{a,b} \|_{C^{1,1}} < \infty. \label{eq_1059_} \end{equation}
\end{enumerate}\label{lem_1315} \end{lemma}

\begin{proof} We will see that the conclusions of the lemma hold for any sufficiently small $\delta_0$, i.e.,
there exists $\tilde{\delta}_0 > 0$
such that the conclusions of the lemma hold for any $0 < \delta_0 < \tilde{\delta}_0$.
For $a \in \cM, X \in T_a \cM$ and $\delta > 0$ we define $B_{T_a \cM}(X, \delta) = \{ Y \in T_a \cM \, ; \, |X - Y| < \delta \}$.

\medskip Item (i) is the content of
Whitehead's theorem, see \cite[Theorem 5.14]{CE} or \cite[Theorem IX.6.1]{Ch}.
Regarding (ii), the openness
of $U_a$ and the fact that $\exp_a: U_a \rightarrow U$ is a smooth diffeomorphism are standard, see \cite[Chapter I]{CE}.
Furthermore, $U_a \subseteq B_{T_a \cM}(0, \delta_0)$, and hence $U_a$ is bounded and  (ii) holds true.

\medskip We move to item (iii).
The function
 $f_{a,X}(Y) := d^2(\exp_a X, \exp_a Y)/2$ is a smooth function, which depends smoothly   also on $a \in U$ and $X \in U_a$.
 The Hessian operator of $f_{p,0}$ at the point $0 \in T_{p} \cM$ is precisely the identity, as follows from (\ref{eq_A1817_}) and \cite[Corollary 1.9]{CE}.
By smoothness, the Hessian operator of $f_{a, X}$ at the point $Y \in T_a \cM$ is at least $\frac{1}{2} \cdot Id$
and at most $2 Id$, whenever $a$ is sufficiently close
to $p$ and $X,Y$ are sufficiently close to zero. In other words, assuming that $\delta_0$
is at most a certain positive constant determined by $p$, we know that for $a \in B_{\cM}(p, 2\delta_0)$ and $X,Y \in B_{T_a \cM}(0, 2\delta_0)$,
\begin{equation}
 \frac{1}{2} \cdot Id \leq \nabla^2 f_{a,X}(Y) \leq 2 \cdot Id. \label{eq_11139}
 \end{equation}
Thus (iii) is proven. It follows from (\ref{eq_11139}) that the function $f_{a,X}$ is convex in the Euclidean ball $B_{T_a \cM}(0, 2\delta_0)$.
Let $a,x \in U$ and $0 < \delta \leq \delta_0$. Then $B_{\cM}(x, \delta) \subseteq B_{\cM}(a, 2\delta_0)$.
Denoting $X = \exp_a^{-1}(x)$ we observe that
\begin{equation}  \{ Y \in T_a \cM \, ; \, f_{a, X}(Y) \leq \delta^2/2 \}
= \exp_a^{-1} \left( B_{\cM}(x, \delta)  \right)
\subseteq B_{T_a \cM}(0, 2 \delta_0). \label{eq_11141} \end{equation}
Since $f_{a,X}$ is convex in $B_{T_a \cM}(0, 2 \delta_0)$, then (\ref{eq_11141}) implies that
the set $\exp_a^{-1} \left( B_{\cM}(x, \delta)  \right)$ is convex. Therefore (iv) is proven.
Thanks to the convexity of $U_a$ we may use Taylor's theorem, and conclude from (\ref{eq_1030})
that for $a \in U, X, Y \in U_a$,
\begin{equation}  \frac{1}{4} \cdot |X-Y|^2 \leq  |f_{a,X}(Y) - (f_{a,X}(X) + \nabla f_{a,X}(X) \cdot (Y-X))| \leq |X-Y|^2. \label{eq_957} \end{equation}
However $f_{a,X}(X) = 0$, and also $\nabla f_{a,X}(X) = 0$ since  $Y \mapsto f_{a,X}(Y)$ attains its minimum at the point $X$.
Therefore (v) follows from (\ref{eq_957}).
Finally, the smooth map $\vphi_{a,b} = \exp_b^{-1} \circ \exp_a: U_a \rightarrow U_b$ smoothly depends  also on $a, b \in U$.
Since the closure of $U$
is compact, the continuous function $\| \vphi_{a,b} \|_{C^{1,1}}$ is bounded over $a, b \in U$, and (\ref{eq_1059_}) follows.
\end{proof}

For the rest of this subsection, we fix a point $p \in \cM$,
and let $\delta_0 > 0$ be the radius whose existence is guaranteed by Lemma \ref{lem_1315}.
Set $U = B_{\cM}(p, \delta_0 / 2)$
and $U_a = \exp_a^{-1}(U)$ for $a \in U$. When we say that a constant $C$ depends  on $p$,
we implicitly allow this constant to depend on the choice of $\delta_0$, on the Riemannian structure of $\cM$
and on the dimension $n$.

\medskip Since $T_X(T_a \cM) \cong T_a \cM$ for any $a \in \cM$ and $X \in T_a \cM$,  we may view the differential of the map $\exp_a$ at the point $X \in T_a \cM$
as a map
$$ \dexp_X: T_a \cM \rightarrow T_x \cM, $$
where $x = \exp_a(X)$. We define $\Pi_{x,a}: T_x \cM \rightarrow T_a \cM$ to be the adjoint map,
where we identify $T_x \cM \cong T_x^* \cM$ and $T_a \cM \cong T_a^* \cM$ by using the Riemannian scalar products.
In other words, for $V \in T_{x} \cM$ we define $\Pi_{x,a}(V) \in T_a \cM$ via
\begin{equation}
\langle \Pi_{x,a}(V), W \rangle_a = \langle V, \dexp_X(W) \rangle_{x} \qquad \qquad \qquad \text{for all} \ W \in T_a \cM.
\label{eq_1728}
\end{equation}
Here, $\langle \cdot, \cdot \rangle_a$ is the Riemannian scalar product in $T_a \cM$, and
$\langle \cdot, \cdot \rangle_x$ is the Riemannian scalar product in $T_x \cM$.
Following Feldman and McCann \cite{FM}, for $a \in U$ and $X, Y \in U_a$ we denote $x = \exp_a(X), y = \exp_a(Y)$ and define
$$ F_a(X, Y) := \exp_{x}^{-1} y. $$
It follows from  Lemma \ref{lem_1315} that the vector $F_a(X,Y) \in U_x$ is well-defined, as $x, y \in U$ and $\exp_x: U_x \rightarrow U$
is a diffeomorphism. Equivalently,  $F_a(X, Y)$
is the unique vector $V \in U_x \subseteq T_x \cM$ for which $\exp_x(V) = y$.
Given $a \in U$ and $X, Y \in U_a$ we define
\begin{equation} \ora{XY} = \Pi_{x,a}(F_a(X,Y)) \in T_a \cM. \label{eq_A1647}
\end{equation}
Intuitively, we think of $\ora{XY}$ as
a vector in $T_a \cM$ which represents ``how $\exp_a(Y)$ is viewed from $\exp_a(X)$''.

\begin{lemma} Let $f: U \rightarrow \RR, t \in \RR, a \in U$ and $X,Y \in U_a$. Denote $x = \exp_a(X), y = \exp_a(Y)$.
Assume that $f$ is differentiable at $x$ with $\nabla f(x) = t \cdot F_a(X,Y)$ and set $f_a = f \circ \exp_a$.
Then $\nabla f_a(X) = t \cdot \ora{XY}$.
\label{lem_grad}
\end{lemma}

\begin{proof} Let us pass to $1$-forms.
Then $d f_a = \exp_a^* (df)$, and for any $W \in T_a \cM$,
\begin{align} \label{eq_A1649} \langle \nabla f_a(X), W \rangle_a & = (d f_a)_X (W) = (df)_x \left( \dexp_X(W) \right) \\
& = \langle \nabla f(x), \dexp_X(W) \rangle_x =
 \langle t F_a(X,Y), \dexp_X(W) \rangle_x.  \nonumber
\end{align}
From (\ref{eq_1728}) and (\ref{eq_A1649}) we obtain that $\nabla f_a(X) = \Pi_{x,a}(t F_a(X)) = t \Pi_{x,a}(F_a(X))$.
The lemma thus follows from (\ref{eq_A1647}).
\end{proof}

\begin{lemma} Let $a \in U, X,Y \in U_a$. Assume that there exists $\alpha \in \RR$ such that $X = \alpha Y$. Then,
\begin{equation}
\ora{XY} = Y - X,
\label{eq_2224}
\end{equation}
and
\begin{equation}
|\ora{XY}| = d(\exp_a X, \exp_a Y).
\label{eq_1738}
\end{equation}
\label{lem_1742}
\end{lemma}

\begin{proof} Let $Z \in T_a \cM$ be a unit vector such that $X$ and $Y$ are proportional to $Z$.
Write $\gamma(t) = \exp_a(t Z)$ for the geodesic leaving $a$ in direction $Z$. Then $\exp_a(X)$ and $\exp_a(Y)$
lie on this geodesic and by the strong convexity of $U$,
$$
d(\exp_a(X), \exp_a(Y)) = |X - Y|.
$$
Therefore (\ref{eq_1738}) would follow once we prove (\ref{eq_2224}). In order to prove (\ref{eq_2224})
we denote $x =\exp_a(X)$ and claim that
\begin{equation}
\langle Y-X, Z \rangle_a = \langle F_a(X,Y), \dexp_X(Z) \rangle_x.
\label{eq_1716}
\end{equation}
Indeed, $F_a(X,Y) \in T_x \cM$ is a vector of length $d(\exp_a X, \exp_a Y) = |Y - X|$ which is tangential to the curve $\gamma$.
The vector $\dexp_X(Z) \in T_x \cM$ is a unit tangent to $\gamma$.
Therefore
$F_a(X,Y)$ is proportional to the unit vector $\dexp_X(Z)$, in exactly the same way that $Y-X$ is proportional
to the unit vector $Z$. Thus (\ref{eq_1716}) follows.
The Gauss lemma \cite[Lemma 1.8]{CE} states that for any $W \in T_a \cM$,
\begin{equation} \langle Z, W \rangle_a = 0 \qquad \Longrightarrow \qquad \langle \dexp_X(Z), \dexp_{X}(W) \rangle_x = 0. \label{eq_1711}
\end{equation}
Recall that $\ora{XY} = \Pi_{x,a}(F_a(X,Y))$ and that
$F_a(X,Y)$ is proportional to
the unit vector $\dexp_X(Z)$.
From (\ref{eq_1728}) and (\ref{eq_1711}) we learn that
$\ora{XY} = \beta Z$ for some $\beta \in \RR$. From (\ref{eq_1728}) and (\ref{eq_1716}),
\begin{equation}  \langle Y-X, Z \rangle_a = \langle F_a(X,Y), \dexp_X(Z) \rangle_x = \langle \ora{XY}, Z \rangle_a =  \langle \beta Z, Z \rangle_a = \beta.
\label{eq_A332} \end{equation}
Since $X$ and $Y$ are proportional to the unit vector $Z$, then $\ora{XY} = \langle Y-X, Z \rangle_a \cdot Z = Y-X$
according to  (\ref{eq_A332}). Thus (\ref{eq_2224}) is proven.
\end{proof}

 \begin{lemma} Let $a \in U$ and $t_0 \in \RR$.
Assume that $V, Z \in U_a$ are such that
$t_0 V \in U_a$. Then, in the notation of Lemma \ref{lem_1315}(iii),
\begin{equation}
 f_{a, t_0 V}(Z) \leq  f_{a, t_0 V}(V) +  \langle (1-t_0) V, Z - V \rangle  +  |Z - V|^2.
 \label{eq_935_}
\end{equation} \label{cor_1205}
\end{lemma}

\begin{proof} Fix $X_0, Y_0 \in U_a$ and define $x_0 = \exp_a(X_0) \in U, y_0 = \exp_a(Y_0) \in U$.
Consider the function $g_{x_0}(y) = \frac{1}{2} \cdot d(x_0, y)^2$, defined for $y \in U$.
Then $\nabla g_{x_0}(y_0)$ equals the vector $V \in U_{y_0} \subseteq T_{y_0} \cM$
for which $x_0 = \exp_{y_0}(-V)$. Consequently,
\begin{equation}  \nabla g_{x_0}(y_0) = - \exp_{y_0}^{-1}(x_0) = -F_a(Y_0,X_0). \label{eq_A336} \end{equation}
Since $f_{a, X_0}  =  g_{x_0} \circ \exp_a$, then from (\ref{eq_A336}) and Lemma \ref{lem_grad},
\begin{equation}  \nabla f_{a, X_0}(Y_0) = -\ora{Y_0 X_0}. \label{eq_goth2}
\end{equation}
According to (\ref{eq_goth2}) and Lemma \ref{lem_1742}, if $X, Y \in U_a$
lie on the same line through the origin, then
$$ \nabla f_{a, X}(Y) = -\ora{YX} = -(X-Y) = Y-X. $$
In particular,
\begin{equation}
\nabla f_{a,t_0 V}(V) =  V - t_0 V = (1 - t_0) V.
\label{eq_912_}
\end{equation}
We may use Taylor's theorem in the convex set $U_a \subseteq T_a\cM$, and deduce from
the bound (\ref{eq_1030}) in Lemma \ref{lem_1315}(iii) that
\begin{equation}
\left| f_{a, t_0 V}(Z) \, - \, \left(  f_{a, t_0 V}(V) +  \langle \nabla f_{a,t_0 V}(V), Z - V \rangle \right) \right| \leq \frac{1}{2} \cdot 2 \cdot |Z - V|^2.
\label{eq_1200} \end{equation}
Now (\ref{eq_935_}) follows from (\ref{eq_912_}) and (\ref{eq_1200}).
\end{proof}

\begin{lemma} Let $a \in U$ and $X, X_1, X_2, Y, Y_1, Y_2 \in U_a$. Then,
\begin{equation} \left| \ora{X Y_2} - \ora{X Y_1} \, - \, (Y_2 - Y_1) \right| \leq C_{p} \cdot |X| \cdot |Y_2 - Y_1|,
\label{eq_1734}
\end{equation}
and
\begin{equation}  \left| \ora{X_1 Y} - \ora{X_2 Y} \, - \, (X_2 - X_1) \right| \leq C_{p} \cdot |Y| \cdot |X_2 - X_1|.
\label{eq_1735} \end{equation}
Here, $C_{p} > 0$ is a constant depending on $p$.
\label{lem_1737}
\end{lemma}

\begin{proof} For $a \in U, X,Y \in U_a$ denote
\begin{equation}  H_{a, X}(Y) = \ora{XY} - Y. \label{eq_A1844_} \end{equation} Then $H_{a, X}: U_a \rightarrow T_a \cM$
is a smooth function. Since $T_a \cM$ is a linear space, then at the point $Y \in U_a$ the derivative
$ H_{a,X}^{\prime}(Y) $
is a linear operator from the space $T_a \cM$ to itself. We claim that there exists a constant $C_p > 0$ depending on $p$ such that
\begin{equation}
 \left \| H_{a, X_2}^{\prime}(Y) - H_{a, X_1}^{\prime}(Y) \right \|_{op}
 \leq  C_{p} \cdot |X_2 - X_1| \qquad \text{for} \ a \in U, X_1, X_2, Y \in U_a,
 \label{eq_1821}
\end{equation}
where $\| S \|_{op} = \sup_{0 \neq V} |S(V)| / |V|$ is the operator norm. Write $\cL(T_a \cM)$ for the space
of linear operators on $T_a \cM$, equipped with the operator norm.
For $a \in U, Y \in U_a$ the map
\begin{equation}  U_a \ni X \mapsto H_{a,X}^{\prime}(Y) \in \cL(T_a \cM) \label{eq_A1825_} \end{equation}
is a smooth map.
In fact, the map in (\ref{eq_A1825_}) may be extended smoothly to the larger domain $a \in B_{\cM}(p, \delta_0), X,Y \in \exp_a^{-1}(B_{\cM}(p, \delta_0))$.
Since $U_a$ is convex with a compact closure, the smooth map in (\ref{eq_A1825_}) is necessarily a Lipschitz map, and  the Lipschitz constant of this map
 depends continuously on $a \in U$ and $Y \in U_a$.
Since the closure of $U$ is compact, the Lipschitz constant of the map in (\ref{eq_A1825_})
is bounded  over
$a \in U$ and $Y \in U_a$. This completes the proof of (\ref{eq_1821}).
From (\ref{eq_A1844_}) and Lemma \ref{lem_1742},
\begin{equation} H_{a, 0}(Y) = 0 \qquad \qquad \qquad \text{for any} \ Y \in U_a. \label{eq_1830}
\end{equation}
From (\ref{eq_1830}) we have $H_{a, 0}^{\prime}(Y) = 0$ for any $Y \in U_a$. The set $U_a$ is convex, and by applying (\ref{eq_1821}) with $X_2 = X$ and $X_1 = 0$ we obtain
$$ \sup_{Y_1, Y_2 \in U_a \atop{Y_1 \neq Y_2}} \frac{|H_{a, X}(Y_2) - H_{a,X}(Y_1)|}{|Y_2 - Y_1|}
= \sup_{Y \in U_a}  \left \| H_{a, X}^{\prime}(Y) \right \|_{op} \leq C_{p} \cdot |X| \qquad \text{for all} \ X \in U_a, $$
and (\ref{eq_1734}) is proven. In order to prove (\ref{eq_1735}), one  needs to analyze
$\tilde{H}_{a, Y}(X) = \ora{XY} + X$. According to Lemma \ref{lem_1742} we know that $\tilde{H}_{a,0}(X) = 0$ for any $X \in U_a$.
The latter equality  replaces (\ref{eq_1830}), and the rest of the proof of (\ref{eq_1735}) is entirely parallel  to the analysis of $H_{a,X}$ presented above.
\end{proof}

\subsection{Proof of the regularity theorem}
\label{proof_c11}
\setcounter{equation}{0}

In this subsection we  prove Theorem  \ref{prop_939}.
We begin with a  geometric lemma:

\begin{lemma}[Feldman and Mccann \cite{FM}]
Let $\cM$ be a Riemannian manifold with distance function $d$, and let $p \in \cM$. Then there exists  $\delta_1 = \delta_1(p) > 0$
with the following property: Let
$x_0,x_1, x_2, y_0, y_1, y_2 \in B_{\cM}(p, \delta_1)$.
Assume that there exists $\sigma > 0$ such that
\begin{equation}  d(x_i, x_j) = d(y_i, y_j) = \sigma |i-j| \leq d(x_i, y_j) \qquad \qquad \text{for} \ i,j \in \{ 0,1,2 \}.\label{eq_0845} \end{equation}
Then,
\begin{equation}
\max \left \{ d(x_0, y_0),  d(x_2,y_2) \right \}  \leq 10 \cdot d(x_1, y_1). \label{eq_1502}
 \end{equation}
\label{lem_849}
\end{lemma}

Together with Whitney's extension theorem, Lemma \ref{lem_849} is the central
ingredient in our proof of Theorem  \ref{prop_939}.
The proof of Lemma \ref{lem_849} provided by Feldman and McCann in \cite[Lemma 16]{FM} is very clear and detailed, yet the notation is a bit
different from ours. For the convenience of the reader, their proof is reproduced
in the Appendix below.

\medskip
Let us recall the assumptions of Theorem \ref{prop_939}. The Riemannian manifold
$\cM$ is geodesically-convex and the function $u: \cM \rightarrow \RR$ satisfies $\| u \|_{Lip} \leq 1$.
We are given a point $p \in \cM$ and a number $\eps_0 > 0$. Set:
\begin{equation}
\delta_2 = \min \left \{ \frac{1}{10 C_{p}}, \frac{\delta_0}{2}, \delta_1 \right \} > 0 \label{eq_1038}
\end{equation}
where
$C_{p}$ is the constant from Lemma \ref{lem_1737},
the constant  $\delta_0 = \delta_0(p)$ is provided by Lemma \ref{lem_1315},
 and $\delta_1 = \delta_1(p)$ is the constant from Lemma \ref{lem_849}.
As before, we denote for $a \in U$,
$$ U = B_{\cM}(p, \delta_0/2), \qquad \qquad U_a = \exp_a^{-1}(U) \subseteq T_a \cM. $$
Recall from the previous subsection that $U \subseteq \cM$ is strongly convex, and that for $a \in U$ and $X, Y \in U_a$
we defined a certain vector $\ora{XY} \in T_a \cM$.

\begin{lemma} Let $\eps, \sigma > 0$.
Let $x, x_0,x_1, x_2, y_0, y_1, y_2 \in B_{\cM}(p, \delta_2) \subseteq U$.
Assume that $d(x, y_1) = \eps$, that $x$ lies on the geodesic arc between $x_0$ and $x_2$, and
that for $i,j \in \{ 0,1,2 \}$,
\begin{equation}  d(x_i, x_j) = d(y_i, y_j) = \sigma |i-j| \leq d(x_i, y_j).\label{eq_916} \end{equation}
Denote $a = x_0$ and let $X, X_0,X_1, X_2, Y_0, Y_1, Y_2 \in U_a = \exp_a^{-1}(U)$
be such that $x = \exp_a(X)$ and $x_i = \exp_a(X_i), y_i = \exp_a(Y_i)$ for $i=0,1,2$.
Then,
\begin{equation}
 \left| \ora{Y_1 Y_2}  \, - \, \ora{X_1 X_2} \right| \leq 100 \cdot \eps, \label{eq_15022}
 \end{equation}
and
\begin{equation}
|\langle X_1, Y_1 - X_1 \rangle| \leq 2000 \cdot \eps^2. \label{eq_949}
\end{equation}
\label{lem_1444}
\end{lemma}

\begin{proof} From (\ref{eq_916}), the point
$x_1$ is the midpoint of the geodesic arc between $x_0$ and $x_2$. The point $x$ also lies on the geodesic between $x_0$
and $x_2$. Let $K \in \{ 0, 2 \}$ be such that $x$ lies on the geodesic from $x_1$ to $x_K$.
According to (\ref{eq_916}),
\begin{equation}  d(x_1, x) + d(x, x_K) = d(x_1, x_K) = \sigma. \label{eq_1741_}
\end{equation}
From  (\ref{eq_916}) and (\ref{eq_1741_}),
\begin{equation} \sigma \leq   d(x_K, y_1) \leq d(x_K, x) + d(x, y_1) =  (\sigma - d(x, x_1)) + d(x, y_1). \label{eq_11622} \end{equation}
By using (\ref{eq_11622}) and our assumption that $d(x,y_1) = \eps$ we obtain
\begin{equation}
d(x_1, y_1) \leq d(x_1, x) + d(x, y_1)   \leq 2 d(x, y_1) = 2 \eps.
\label{eq_918}
\end{equation}
We would like to apply Lemma \ref{lem_849}. Recall from (\ref{eq_1038})
that $\delta_2 \leq \delta_1$, where $\delta_1 = \delta_1(p)$ is the constant from Lemma \ref{lem_849}.
Therefore $x_0,x_1,x_2,y_0,y_1, y_2 \in B_{\cM}(p, \delta_1)$. Moreover, assumption (\ref{eq_0845})
holds in view of (\ref{eq_916}). We may therefore apply Lemma \ref{lem_849}, and according to
its conclusion,
\begin{equation}
d(x_i, y_i) \leq 10 \cdot d(x_1, y_1) \leq 20 \eps \qquad \qquad \qquad (i=0,1,2),
\label{eq_992}
\end{equation}
where we used (\ref{eq_918}) in the last passage.
By Lemma \ref{lem_1315}(v), the inequality (\ref{eq_992}) yields
\begin{equation}
|X_i - Y_i| \leq 40 \eps \qquad \qquad \qquad (i=0,1,2).
\label{eq_999}
\end{equation}
Since $a = x_0$ and $\exp_a(X_0) = x_0$, then $X_0 = 0$.
According to Lemma \ref{lem_1742},  for $i=0,1,2$,
\begin{equation} |Y_i| = |\ora{X_0 Y_i}| = d(x_0, y_i) \leq 2 \delta_2,
\qquad |X_i| = |\ora{X_0 X_i}| = d(x_0, x_i) \leq 2 \delta_2,
\label{eq_922_} \end{equation}
as $x_0, x_1, x_2, y_0, y_1, y_2 \in B_{\cM}(p, \delta_2)$.
From Lemma \ref{lem_1737} combined with  (\ref{eq_999}) and (\ref{eq_922_}),
\begin{equation}
\left| \ora{Y_1 Y_2} - \ora{X_1 Y_2}  - (X_1 - Y_1)  \right| \leq C_{p} \cdot |Y_2| \cdot |Y_1 - X_1| \leq C_{p} \cdot 2 \delta_2 \cdot 40 \eps \leq 10 \eps,
\label{eq_1121}
\end{equation}
where we used the fact that $\delta_2 C_{p} \leq 1/10$ in the last passage, as follows from (\ref{eq_1038}). Similarly, according to Lemma \ref{lem_1737} and the inequalities (\ref{eq_999}) and
(\ref{eq_922_}),
\begin{equation}
\left| \ora{X_1 Y_2} - \ora{X_1 X_2}  - (Y_2 - X_2)  \right| \leq C_{p} \cdot |X_1| \cdot |Y_2 - X_2| \leq C_{p} \cdot 2 \delta_2 \cdot 40 \eps \leq 10\eps.
\label{eq_1136}
\end{equation}
Finally, by using (\ref{eq_999}), (\ref{eq_1121}) and (\ref{eq_1136}),
\begin{align*}  |\ora{Y_1 Y_2}  & - \ora{X_1 X_2}| = | (\ora{Y_1 Y_2} - \ora{X_1 Y_2}) \, + \, (\ora{X_1 Y_2} - \ora{X_1 X_2})  | \\ & \leq 20\eps +
|(X_1 - Y_1) + (Y_2 - X_2)| \leq 20\eps + |X_1 - Y_1| + |Y_2 - X_2| \leq 100 \eps, \end{align*}
and (\ref{eq_15022}) is proven. We move on to the proof of (\ref{eq_949}).
For $a \in U$ and $W,Z \in U_a$ define
\begin{equation}  d_{a}(W,Z) :=  d(\exp_a W, \exp_a Z). \label{eq_946} \end{equation}
Then $d_{a}^2(W,Z) = 2 f_{a,W}(Z)$, in the notation of
 Lemma \ref{lem_1315}(iii).
  Using Lemma \ref{cor_1205} with $V = X_1, t_0 =0$ and $Z = Y_1$,
\begin{equation}
d_a^2(X_0, Y_1) \leq d_a^2(X_0, X_1) + \langle 2 X_1, Y_1 - X_1 \rangle + 2 |Y_1 - X_1|^2.
\label{eq_940}
\end{equation}
From (\ref{eq_916}) and (\ref{eq_946}),
$$ d_a(X_0, Y_1) = d(x_0, y_1) \geq d(x_0, x_1) = d_a(X_0, X_1). $$ Therefore (\ref{eq_940}) entails
\begin{equation}
\langle X_1, Y_1 - X_1 \rangle \geq - |Y_1 - X_1|^2. \label{eq_949_}
\end{equation}
Since
$x_1$ is the midpoint of the geodesic between $a=x_0$ and $x_2$, then
$x_2 = \exp_a(X_2) = \exp_a(2 X_1)$. Hence $X_2 = 2 X_1$.
By using Lemma \ref{cor_1205} with $V = X_1, t_0 =2$ and $Z = Y_1$ we obtain
\begin{equation}
d_a^2(X_2, Y_1) \leq d_a^2(X_2, X_1) + \langle -2 X_1, Y_1 - X_1 \rangle + 2|Y_1 - X_1|^2.
\label{eq_941_}
\end{equation}
As before, from (\ref{eq_916}) and (\ref{eq_946}) we deduce that $d_a(X_2, Y_1) \geq d_a(X_2, X_1)$. Therefore (\ref{eq_941_})
leads to
\begin{equation}
\langle X_1, Y_1 - X_1 \rangle  \leq |Y_1 - X_1|^2.
\label{eq_941}
\end{equation}
The desired conclusion (\ref{eq_949}) follows from (\ref{eq_999}), (\ref{eq_949_}) and (\ref{eq_941}).
\end{proof}

\begin{proof}[Proof of Theorem \ref{prop_939}]
Denote
\begin{equation} \sigma = \min \{ \eps_0 / 2, \delta_2 / 3\}.
\label{eq_A2157} \end{equation} We will prove the theorem with
\begin{equation} \delta = \min \{ \sigma/2, 1 \}. \label{eq_B1336A} \end{equation}
We would like to apply Whitney's extension theorem, in the form
of Corollary \ref{cor_1320} and Remark \ref{rem_1320}. Denote $\vphi_a = \exp_a: U_a \rightarrow U$ for any $a \in U$.
Then $\vphi_a$ is a smooth diffeomorphism between the convex, open set $U_a \subseteq T_a \cM$
and the open set $U \subseteq \cM$.
Thanks to
Lemma \ref{lem_1315}(vi), there exists a constant $R = R_{p} > 0$
depending  on $p$ with the following property: For any $a,b \in U$, condition (\ref{eq_1243}) from Corollary \ref{cor_1320} holds true.
Furthermore, since $u$ is a Lipschitz function,
\begin{equation}
R_2 := 1 + \sup_{x \in B_{\cM}(p, \delta)} |u(x)| < \infty. \label{eq_1213}
\end{equation}
Denote
\begin{equation}  A = \{ x \in B_{\cM}(p, \delta) \, ; \, \ell_u(x) > \eps_0 \} = B_{\cM}(p, \delta) \cap \Strain_{\eps_0}[u].
\label{eq_A2158} \end{equation}
Then $A \subseteq U = \cB_{\cM}(p, \delta_0/2)$ according to
(\ref{eq_1038}), (\ref{eq_A2157}) and (\ref{eq_B1336A}).
The function $u$ is differentiable on the entire set $A$, according to Lemma
\ref{lem_1046}. Define a $1$-form $\omega$ on  $A$ by setting
 $ \omega = du|_A$.
 We will verify that the scalar function $u: A \rightarrow \RR$
 and the $1$-form $\omega$ on the set $A$ satisfy condition
 (\ref{eq_1244}) from Corollary \ref{cor_1320}. In fact, for any $x, y \in A$ we will show that
there exists $a \in U$ for which
 \begin{equation}
 \| (u_a, \omega_a) \|_{\vphi_a^{-1}(x), \vphi_a^{-1}(y)} \leq \max \left \{ R_2, \frac{10^4}{\sigma} \right \}, \label{eq_A1819}
 \end{equation}
 where $u_a = u \circ \vphi_a$ and $\omega_a = \vphi_a^* \omega$. Once we prove (\ref{eq_A1819}),
 the theorem easily follows:
The right-hand side of (\ref{eq_A1819}) depends on the point $p$ and on the function $u$,
but not on the choice of $x, y \in A$. Thus condition (\ref{eq_1244}) of Corollary \ref{cor_1320} is satisfied.
From the conclusion of Corollary \ref{cor_1320}, there exists a $C^{1,1}$-function $\tilde{u}: U \rightarrow \RR$ with
\begin{equation} \tilde{u}|_A = u|_A, \qquad d \tilde{u}|_A  = \omega = du|_A.
\label{eq_A2037} \end{equation}
Since $U \supseteq B_{\cM}(p, \delta)$, the theorem follows from
(\ref{eq_A2158}) and (\ref{eq_A2037}). Therefore, all that remains is
to show that for any $x, y \in A$ there exists $a \in U$ for which (\ref{eq_A1819}) holds true.

\medskip
Let us fix $x,y \in A$.
Since $\ell_u(x) > \eps_0 \geq 2 \sigma$ and also $\ell_u(y) > 2 \sigma$ then by Corollary \ref{cor_A1217}
there exist minimizing geodesics $\gamma_x, \gamma_y: (-2 \sigma, 2 \sigma) \rightarrow \cM$
with $\gamma_x(0) = x, \gamma_y(0) = y$ such that
\begin{equation}
u(\gamma_x(t)) = u(x) + t, \qquad u(\gamma_y(t)) = u(y) + t, \qquad \qquad \text{for} \ t \in (-2 \sigma, 2 \sigma),
\label{eq_A1159}
\end{equation}
and such that
\begin{equation}  \nabla u(\gamma_x(t)) = \dot{\gamma}_x(t), \qquad
\nabla u(\gamma_y(t)) = \dot{\gamma}_y(t) \qquad \qquad  \text{for} \ t \in (-2 \sigma, 2 \sigma).
\label{eq_A1234} \end{equation}
Recall that $x,y \in A \subseteq B_{\cM}(p, \delta)$. Denote
\begin{equation} \eps := d(x, y) < 2 \delta \leq \sigma.
\label{eq_1212}
\end{equation}
Set $t_0 = u(y) - u(x)$. Since $u$ is $1$-Lipschitz, then (\ref{eq_1212}) implies that $|t_0| < \sigma$.
We now define
\begin{equation} x_i = \gamma_x \left( t_0 + (i-1) \sigma \right), \qquad y_i = \gamma_y \left( (i-1) \sigma \right) \qquad
\qquad \qquad \text{for} \ i=0,1,2.
\label{eq_A1202}
\end{equation}
Since $|t_0| < \sigma$ then $t_0 + (i-1) \sigma \in (-2 \sigma, 2 \sigma)$ and the points
$x_0,x_1,x_2,y_0,y_1,y_2$ are well-defined. Since $t_0 = u(y) - u(x)$ then (\ref{eq_A1159}) and (\ref{eq_A1202}) yield
\begin{equation} u(x_i) = u(y_i) = u(x_0) + i \sigma \qquad \qquad \qquad \text{for} \ i=0,1,2.
\label{eq_A2110}
\end{equation}
Recall that $\| u \|_{Lip} \leq 1$ and that
$\gamma_x, \gamma_y$
are minimizing geodesics. We deduce from (\ref{eq_A1202}) and (\ref{eq_A2110})
that for $i,j \in \{ 0,1,2 \}$,
\begin{equation}  d(x_i, x_j) = d(y_i, y_j) = \sigma |i-j| = |u(x_i) - u(y_j)| \leq d(x_i, y_j). \label{eq_1217} \end{equation}
Since $\gamma_x(0) = x$ and $|t_0| < \sigma$, then by (\ref{eq_A1202})
the points $x_0,x_1,x_2$ are of distance at most
$2 \sigma$ from $x$. Similarly,
 the points $y_0,y_1,y_2$ are of distance at most
$\sigma$ from $y = y_1$. Since $x,y \in \cB_{\cM}(p, \delta)$ we obtain
\begin{equation} x, x_0, x_1, x_2, y_0, y_1, y_2 \in B(p, \delta_2) \subseteq U, \label{eq_1209}
\end{equation}
as $\delta \leq \sigma/2 \leq \delta_2/6$.
Recall from (\ref{eq_A1202}) that  $x_0 = \gamma_x(t_0 - \sigma)$ and $x_2 = \gamma_x(t_0 + \sigma)$.
Since $\gamma_x(0) = x$ and $|t_0| < \sigma$, the point $x$ lies on the geodesic arc from $x_0$ to $x_2$.
Furthermore, $x \not \in \{ x_0, x_2 \}$.
Thus all of the requirements of Lemma \ref{lem_1444} are satisfied: This follows from
(\ref{eq_1212}), (\ref{eq_1217}) and (\ref{eq_1209}), as $y = y_1$. We are therefore permitted
to use the conclusions of Lemma \ref{lem_1444}.
Denote $$ a = x_0. $$
As in Lemma \ref{lem_1444} we define $X, X_0,X_1, X_2, Y_0, Y_1, Y_2 \in U_a$
via $x = \exp_a(X)$ and $x_i = \exp_a(X_i), y_i = \exp_a(Y_i)$ for $i=0,1,2$.
Thus $X_0 = 0$. According to (\ref{eq_1212}) and Lemma \ref{lem_1315}(v),
\begin{equation}
\eps = d(x,y) = d(x, y_1) \leq 2 |X - Y_1|.
\label{eq_1300} \end{equation}
The four points $a = x_0, x_1, x_2, x$ lie on the minimizing geodesic $\gamma_x$,
according to (\ref{eq_A1202}). Therefore the four vectors $0 = X_0, X_1, X_2, X$ lie on a line through the origin
in $T_a \cM$.
Furthermore, since $x$ and $x_1$ lie on the geodesic arc between $x_0$ and $x_2$, then
 $X$ and $X_1$ belong to the line segment
between $X_0$ and $X_2$.
Since
$x_1$ is the midpoint of the geodesic between $x_0$ and $x_2$, then
$x_2 = \exp_a(X_2) = \exp_a(2 X_1)$. Hence,
\begin{equation} X_2 = 2 X_1. \label{eq_A529_}
\end{equation}
Since $X$ lies on the line segment between the point $0 = X_0$ and the point $X_2 = 2 X_1$ while $X \not \in \{ X_0, X_2 \}$, then
there exists $t \in (0, 2 \sigma)$ such that $X_2 = X + (t / \sigma) \cdot X_1$.
We claim that
\begin{equation}
\gamma_x(t) = \exp_a( X + (t/ \sigma) \cdot  X_1 ) \qquad \qquad \text{for} \ t \in (-2 \sigma, 2 \sigma).
\label{eq_A2224} \end{equation}
Indeed, since $\exp_a(X_1) = x_1$ then $|X_1| = d(a, x_1) = d(x_0, x_1) = \sigma$
according to (\ref{eq_1217}) and the strong convexity of $U$. Therefore $t \mapsto \exp_a( X + (t/ \sigma) \cdot  X_1 )$ is a geodesic of unit speed. Since $\gamma_x(0) = x = \exp_a(X)$, then the equality in (\ref{eq_A2224}) holds true  when $t = 0$.
The two unit speed geodesics $t \mapsto \gamma_x(t)$ and $t \mapsto \exp_a( X + (t/ \sigma) \cdot  X_1 )$
visit the point $x$ at time $t =0$, and at a later time $t \in (0, 2 \sigma)$ they visit the point $x_2$.
By strong convexity, these two geodesics coincide, and (\ref{eq_A2224}) is proven.
Next, from (\ref{eq_1217}), (\ref{eq_A529_}) and Lemma \ref{lem_1742},
\begin{equation}
\ora{XX_2} = X_2 - X =   |X_2 - X| \cdot \frac{X_1 }{|X_1|} =
|X_2 - X| \cdot \frac{X_2 - X_1}{|X_1|} = d(x, x_2) \cdot \frac{\ora{X_1 X_2}}{\sigma}.
\label{eq_1222}
\end{equation}
From (\ref{eq_A1234}) we see that $\nabla u(x)$ is the unit tangent to the geodesic from $x$ to $x_2$.
Similarly, $\nabla u(y)$ is the unit tangent to the geodesic from $y = y_1$ to $y_2$.
Thus,
\begin{equation} \nabla u(x) = \frac{F_p(X, X_2)}{d(x, x_2)}, \qquad \qquad \nabla u(y) = \frac{F_p(Y_1, Y_2)}{d(y_1, y_2)} = \frac{F_p(Y_1, Y_2)}{\sigma},
\label{eq_1220_} \end{equation}
where we used (\ref{eq_1217}) in the last equality.
Recall that $u_a(Z) = u(\vphi_a(Z)) = u(\exp_a(Z))$ for $Z \in U_a$. According to Lemma  \ref{lem_grad}, (\ref{eq_1222}) and (\ref{eq_1220_}),
\begin{equation}  \nabla u_a(X) = \frac{\ora{X X_2}}{d(x, x_2)} = \frac{\ora{X_1 X_2}}{\sigma} = \frac{X_1}{\sigma}, \qquad \qquad \nabla u_a(Y_1) = \frac{\ora{Y_1 Y_2}}{\sigma}.
\label{eq_1255} \end{equation}
From (\ref{eq_A1159}) and (\ref{eq_A2224}), the function $u_a = u \circ \exp_a$ satisfies that $u_a( (t / \sigma) X_1 ) = u_a(0) + t$
for all $t \in [0,2 \sigma]$. Since both $X_1$ and $X$ belong to the line segment between $X_0 = 0$ and $X_2 = 2 X_1$ then
\begin{align} u_a(X_1) - u_a(X)  =  \left \langle X_1, \frac{X_1}{\sigma} \right \rangle
- \left \langle X, \frac{X_1}{\sigma} \right \rangle  =
 \left \langle X_1 - X, \frac{X_1}{\sigma} \right \rangle. \label{eq_1200_} \end{align}
According to (\ref{eq_1255}) and  conclusion (\ref{eq_15022}) of Lemma \ref{lem_1444},
\begin{equation}
|\nabla u_a(X) - \nabla u_a(Y_1)| = \frac{1}{\sigma} \cdot |\ora{X_1 X_2} - \ora{Y_1 Y_2}| \leq 100  \eps  / \sigma \leq  \frac{200}{\sigma} \cdot |X - Y_1|,
\label{eq_1301}
\end{equation}
where we used (\ref{eq_1300}) in the last passage. Furthermore,
conclusion (\ref{eq_949}) of Lemma \ref{lem_1444} implies that
\begin{equation}
|\langle X_1, Y_1 - X_1 \rangle| \leq 2000 \eps^2 \leq 10^4  |X - Y_1|^2, \label{eq_1307}
\end{equation}
where again we used (\ref{eq_1300}) in the last passage.
From (\ref{eq_A2110}) we know that $u_a(X_1) = u(x_1) = u(y_1) = u_a(Y_1)$.
According to  (\ref{eq_1255}), (\ref{eq_1200_}) and (\ref{eq_1307}),
\begin{align}   \label{eq_1918}
|&u_a(X) +  \left \langle \nabla u_a(X), Y_1 - X \right \rangle - u_a(Y_1)| \\ & = |u_a(X) + \left \langle \nabla u_a(X), X_1 - X \right \rangle + \left \langle \nabla u_a(X), Y_1 - X_1 \right \rangle - u_a(Y_1) | \nonumber \\ & = \left|u_a(X_1) +
\left \langle \frac{X_1}{\sigma}, Y_1 - X_1 \right \rangle - u_a(Y_1)\right|  = \left| \left \langle \frac{X_1}{\sigma}, Y_1 - X_1 \right \rangle \right| \leq \frac{10^4}{\sigma} |X - Y_1|^2.
\nonumber   \end{align}
From (\ref{eq_1213}), (\ref{eq_1217}) and (\ref{eq_1255}),
\begin{equation}
|u_a(X)| \leq R_2, \qquad |\nabla u_a(X)| = \frac{|X_1 - X_0|}{\sigma} = \frac{d(x_0, x_1)}{\sigma} = 1 \leq R_2.
\label{eq_1215_}
\end{equation}
Recall that
$\vphi_a = \exp_a$
and that
$\omega_a = \vphi_a^* \omega = \vphi_a^* (du|_A) = d u_a|_{\vphi_a^{-1}(A)}$.
The inequalities (\ref{eq_1301}), (\ref{eq_1918}) and (\ref{eq_1215_}) mean precisely that
$$
 \| (u_a, \omega_a) \|_{\vphi_a^{-1}(x), \vphi_a^{-1}(y)} = \| (u_a, \omega_a) \|_{X,Y_1}
 = \| (u_a, \nabla u_a) \|_{X, Y_1}
\leq \max \left \{ R_2, \frac{10^4}{\sigma} \right \}.
$$
To summarize, given the arbitrary points $x,y \in A$, we found
$a \in U$ for which (\ref{eq_A1819}) holds true. The proof is thus complete.
\end{proof}

By using a partition of unity and a standard argument, we may deduce from Theorem
\ref{prop_939} the following corollary (which will not be needed here):

\begin{corollary} Let $\cM$ be a geodesically-convex Riemannian manifold. Let $u: \cM \rightarrow \RR$ satisfy $\| u \|_{Lip} \leq 1$ and let $\eps_0 > 0$.
Then there exists
a $C^{1,1}$-function $\tilde{u}: \cM \rightarrow \RR$
such that for any $x \in \cM$,
$$  x \in
 \Strain_{\eps_0}[u] \qquad \Longrightarrow \qquad \tilde{u}(x) = u(x), \ \ \nabla \tilde{u}(x) = \nabla u(x).
 $$ \end{corollary}

\section{Conditioning a measure with respect to  an integrable geodesic foliation}
\label{sec_condition}
\setcounter{equation}{0}

Let $(\cM, d, \mu)$ be a weighted Riemannian manifold of dimension $n$ which is geodesically-convex.
In this section we describe the conditioning of $\mu$ with respect to the partition $T^{\circ}[u]$
 associated with
a given $1$-Lipschitz function $u$.
 The conditioning is based on ``ray clusters'' which are defined
 in Section \ref{lip_sec}. Analogous constructions appear in
Caffarelli, Feldman and McCann
 \cite{CFM}, Evans and Gangbo \cite{EG}, Feldman and McCann \cite{FM} and Trudinger and Wang \cite{TW}.
Section \ref{sec_decomp} explains that the set $\Strain[u]$  may be partitioned into countably
 many ray clusters. The connection with curvature
  appears on Section  \ref{ricci}.

\subsection{Geodesics emanating from a $C^{1,1}$-hypersurface}
\label{lip_sec}
\setcounter{equation}{0}

In what follows we prefer to work with a slightly different normalization of the exponential map. For $t \in \RR$ set
$$ \Exp_t(v) = \exp_p(t v) \qquad \qquad \qquad (p \in \cM, v \in T_p \cM). $$
Then  $\Exp_t: T \cM \rightarrow \cM$
is a partially-defined map which is well-defined and smooth on a maximal open set containing the zero section.
That is, for any $v \in T \cM$
there is a maximal connected set $I \subseteq \RR$ containing the origin such that $\Exp_t (v)$
is well-defined for $t \in I$. This maximal connected subset $I$ is always open,
and if $t \in I$, then $\Exp_s(w)$ is well-defined for any  $(w,s) \in T \cM \times \RR$ which is sufficiently close to
$(v,t) \in T \cM \times \RR$.
 Write $\dExp_t: T(T \cM) \rightarrow T \cM$ for the differential of the map $\Exp_t: T \cM \rightarrow \cM$. The maps $\Exp_t$ and $\dExp_t$ are smooth in all of their variables, including the $t$-variable.

\medskip
Let $\gamma:(a,b) \rightarrow \cM$ be a smooth curve with $a,b \in \RR \cup \{ \pm \infty \}$.
We say that $J: (a,b) \rightarrow T \cM$ is a smooth vector field along $\gamma$ if $J$ is smooth and $J(t) \in T_{\gamma(t)} \cM$ for any $t \in (a,b)$.
As in Cheeger and Ebin \cite[Section 1.1]{CE}, we may use the Riemannian connection and consider the {\it covariant derivative} of $J$ along $\gamma$, denoted by
$$ J^{\prime} = \nabla_{\dot{\gamma}} J. $$
Then $J^{\prime}: (a,b) \rightarrow T \cM$ is a well-defined, smooth vector field along $\gamma$. Assume that $\gamma: (a,b) \rightarrow \cM$
is a geodesic. We say that a smooth vector field $J$ along $\gamma$ is a {\it Jacobi field} if
\begin{equation} J''(t) = R(\dot{\gamma}(t), J(t)) \dot{\gamma}(t) \qquad \qquad \qquad \text{for} \ t \in (a,b),
\label{eq_840}
\end{equation}
where $R$ is the {\it Riemann curvature tensor}. We refer the reader to Cheeger and Ebin \cite[Chapter I]{CE}
for background on the Jacobi equation  (\ref{eq_840}). The space of Jacobi
fields along the fixed geodesic curve $\gamma$ is a linear space of dimension $2n$. In fact,
we may parameterize the space of Jacobi fields along $\gamma$ by the $(2n)$-dimensional vector space
$T_{\dot{\gamma}(0)} (T \cM)$. The parametrization is defined as follows: For $\xi \in T_{\dot{\gamma}(0)} (T \cM)$
we define a Jacobi field $J$ via
\begin{equation} J(t) = \dExp_t (\xi) \qquad \qquad \qquad \text{for} \ t \in (a,b).
\label{eq_B1624} \end{equation}

Let $V: \cM \rightarrow T \cM$ be a vector field on $\cM$,
i.e., $V(p) \in T_p \cM$ for any $p \in \cM$. Assume that $V$ is
differentiable at the point $p \in \cM$.
For $w \in T_p  \cM$ we write  $\partial_w V \in T_{V(p)} (T \cM)$ for the usual directional derivative of the map $V: \cM \rightarrow
T \cM$.
We write $\nabla_w V \in T_p \cM$ for the covariant derivative of $V$ with respect to the Riemannian connection.
Note the formal difference between the directional derivative $\partial_w V \in T_{V(p)} (T \cM)$ and the covariant derivative
 $\nabla_w V \in T_p \cM$.
In the case where $\cM = \RR$, the relation between $\partial_w V$ and $\nabla_w V$ is rather like the relation
between the tangent to the plane curve $t \mapsto (t,f(t))$ and the derivative of the scalar-valued function $t \mapsto f(t)$.

\begin{lemma}
Let $a \in [-\infty, 0), b \in (0, +\infty]$, let $\gamma: (a, b) \rightarrow \cM$ be a geodesic and let $\xi \in T_{\dot{\gamma}(0)} (T \cM)$. Let $J(t)$
 be the Jacobi field
along $\gamma$ that is given by (\ref{eq_B1624}).
Assume that $V$ is a vector field on $\cM$ that is differentiable at the point $\gamma(0) \in \cM$ and satisfies  $\partial_{J(0)} V = \xi$. Then,
$$  J^{\prime}(0) = \nabla_{J(0)} V. $$
\label{lem_B1815}
\end{lemma}

\begin{proof}
Let $\beta: (-1, 1) \rightarrow T \cM$ be a smooth, one-to-one curve satisfying
$\beta(0) = \dot{\gamma}(0)$ and $\dot{\beta}(0) = \xi = \partial_{J(0)} V$. A moment of contemplation reveals that
$$ \nabla_{J(0)} \beta = \nabla_{J(0)} V, $$
where  we use the conventions from \cite[Section 1.1]{CE} regarding vector fields along a smooth map
 and their covariant derivatives.
 Set $\alpha(s,t) = \Exp_t(\beta(s))$. Then $\alpha$ is smooth in $(s,t) \in \RR^2$ near the origin,
while $J(t) = \frac{\partial \alpha}{\partial s} (0, t)$ and
 $\beta(s) = \frac{\partial \alpha}{\partial t} (s,0)$.
As in \cite[Section 1.5]{CE} we abbreviate $S = d \alpha( \frac{\partial}{\partial s})$
 and $T = d \alpha (\frac{\partial}{\partial t})$, which are smooth vector fields along the map $\alpha$
 with $S(0, t) = J(t)$ and $T(s, 0) =  \beta(s)$.    Then,
 \begin{equation} J^{\prime}(0) = \left. \nabla_{T}
S  \right|_{t,s=0}, \qquad \nabla_{J(0)} V = \nabla_{J(0)} \beta
 = \left. \nabla_{S}
 T \right|_{t,s=0}. \label{eq_B936_} \end{equation}
Since $\left[\frac{\partial}{\partial s}, \frac{\partial}{\partial t} \right] = 0$
then $[S,T] = 0$ and consequently $\nabla_S T = \nabla_T S$. The lemma thus follows from (\ref{eq_B936_}).
\end{proof}

We say that a $C^{1}$-function $f: \cM \rightarrow \RR$ is {\it twice differentiable with a symmetric Hessian}
at the point $p \in \cM$ if the vector field $\nabla f$ is differentiable at $p$ and
$$  \langle \nabla_v (\nabla f),  w \rangle \, = \,
\langle \nabla_w (\nabla f), v \rangle \qquad \qquad \qquad \text{for} \ v,w \in T_p \cM.
$$

The notation of the next  lemma will accompany us now for several pages.
We will consider geodesics orthogonal to the level set $\{ \tilde{u} = r_0 \}$,
where $\tilde{u}: \cM \rightarrow \RR$ is usually twice differentiable with a symmetric Hessian.
This level set is locally parameterized by a $C^1$-function $f: \Omega_0 \rightarrow \cM$ where $\Omega_0 \subseteq \RR^{n-1}$ is an open set.
The geodesics are denoted by $\tilde{F}(y,t) = \Exp_t(\nabla \tilde{u}(f(y))$.
Later on, the restriction of $\tilde{F}$ to a certain set will be denoted by $F$,
while $\tilde{u}$ will be the function provided by Theorem \ref{prop_939}.
 By differentiating $\tilde{F}(y,t)$ with respect to $y_i$ we obtain a Jacobi field $J_i$, as is
 precisely stated  in the following lemma:

\begin{lemma} Let $r_0 \in \RR$ and let $\tilde{u}: \cM \rightarrow \RR$ be a $C^{1}$-function.
Let $\Omega_0 \subseteq \RR^{n-1}$ be an open set and let $y_0 \in \Omega_0$. Let $f: \Omega_0 \rightarrow \cM$ be a $C^{1}$-map,
and assume that the function $\tilde{u}$ is twice differentiable with a symmetric Hessian
at the point $f(y_0)$. For $y \in \Omega_0$ and $t \in \RR$ set
 $$ \tilde{F}(y,t) = \Exp_t(\nabla \tilde{u}(f(y))), \qquad N(y,t) = \frac{\partial \tilde{F}}{\partial t}(y,t). $$
 Our Riemannian manifold is not necessarily complete, and we assume that $t \mapsto \tilde{F}(y, t)$ is well-defined  in a maximal subset $(a_y,b_y) \subseteq \RR$ containing the origin.
 Suppose that $B_0 \subseteq \Omega_0$ is a measurable set containing $y_0$, such that $y_0$ is a Lebesgue density point of $B_0 \subseteq \RR^{n-1}$,
 and
\begin{equation}  \tilde{u}(f(y)) = r_0, \quad |\nabla \tilde{u}(f(y))| = 1 \qquad \qquad \qquad \text{for} \ y \in B_0. \label{eq_B1627}
\end{equation}
Then,
\begin{enumerate}
\item[(i)] For any $t \in (a_{y_0},b_{y_0})$ the map $\tilde{F}$ is
 differentiable at the point $(y_0, t) \in \Omega_0 \times \RR$.
 (We note that  $\tilde{F}$ is  well-defined in an open neighborhood of $(y_0, t)$
in $\RR^{n-1} \times \RR$).
\item[(ii)] There exist
 Jacobi fields $J_1(y_0,t),\ldots,$ $J_{n-1}(y_0,t)$ along the geodesic
curve $t \mapsto \tilde{F}(y_0,t)$, which are well-defined in the entire interval $t \in (a_{y_0}, b_{y_0})$, such that
$$ J_i(y_0,t) = \frac{\partial \tilde{F}}{\partial y_i}(y_0,t) \qquad \qquad \text{for all} \ i=1,\ldots, n-1, \ t \in (a_{y_0}, b_{y_0}). $$
\item[(iii)] At the point $(y_0,0)
\in \Omega_0 \times \RR$ we have
\begin{equation} \langle J_i, N \rangle = \langle J_i', N \rangle = 0 \qquad \qquad (i=1,\ldots,n-1), \label{eq_B1537} \end{equation}
and
\begin{equation}
 \langle J_i', J_k \rangle =\langle J_k', J_i \rangle \qquad \qquad (i,k=1,\ldots,n-1). \label{eq_B1536} \end{equation}
 Here, $J_i^{\prime}(y_0, t)$ is the covariant derivative of the Jacobi field
 $t \mapsto J_i(y_0, t)$ along the geodesic curve $t \mapsto \tilde{F}(y_0, t)$ for $t \in (a_{y_0}, b_{y_0})$. \end{enumerate}
 \label{lem_jacobi}
\end{lemma}

\begin{proof}
The curve $t \mapsto \tilde{F}(y_0,t)$ is a geodesic curve of speed one since $|\nabla \tilde{u}(f(y_0))| = 1$ as
follows from (\ref{eq_B1627}) and the fact that  $y_0 \in B_0$.
 The vector field $t \mapsto N(y_0,t)$ is the unit tangent along this geodesic, with
$ N(y_0,0) = \nabla \tilde{u}(f(y_0)). $
The equation
\begin{equation} \tilde{F}(y,t) = \Exp_t(\nabla \tilde{u}(f(y)))
\label{eq_B1113} \end{equation}
is valid in an open set in $\Omega_0 \times \RR$ containing  $\{ y_0 \} \times (a_{y_0}, b_{y_0})$.
Note also that $\tilde{F}(y,0) = f(y)$ for $y \in \Omega_0$.
Since $f$ is a $C^1$-function,
$$ \frac{\partial f}{\partial y_i}(y_0) = \frac{\partial \tilde{F}}{\partial y_i}(y_0, 0) \qquad \qquad \qquad \text{for} \ i=1,\ldots,n-1. $$
Differentiating (\ref{eq_B1113}) at the point $y = y_0$ yields
\begin{equation}  J_i(y_0,t) := \frac{\partial \tilde{F}}{\partial y_i} (y_0,t)
= \dExp_t \left( \xi_{y_0,i} \right) \qquad \text{for} \ t \in (a_{y_0}, b_{y_0}), i=1,\ldots,n-1, \label{eq_B1811} \end{equation}
where
\begin{equation} \xi_{y_0,i} = \frac{\partial [ (\nabla \tilde{u}) \circ f ]}{\partial y_i} (y_0) = \partial_{J_i(y_0, 0)} \nabla \tilde{u} \in T_{N(y_0,0)} (T \cM)
 \qquad \text{for} \ i=1,\ldots,n-1. \label{eq_B1001} \end{equation}
This differentiation is legitimate since $f$ is a $C^{1}$-map and
since the vector field $\nabla \tilde{u}: \cM \rightarrow T \cM$ is differentiable at the point $f(y_0)$.
We conclude that
for any $t \in (a_{y_0},b_{y_0})$, the map $\tilde{F}$ is differentiable at $(y_0,t)$, and (i) is proven.
From (\ref{eq_B1811}) we learn that the vector fields $J_1(y_0, t),\ldots,J_{n-1}(y_0, t)$
have the form (\ref{eq_B1624}), and hence they are Jacobi fields along the geodesic $t \mapsto \tilde{F}(y_0, t)$.
This proves (ii). Thanks to (\ref{eq_B1811}) and (\ref{eq_B1001})
we may apply Lemma \ref{lem_B1815} with $V = \nabla \tilde{u}, \xi = \xi_{y_0, i}$ and $J(t) = J_i(y_0,t)$, and  conclude  that
\begin{equation}
J_i^{\prime}(y_0, 0) = \nabla_{J_i(y_0, 0)} \nabla \tilde{u}
\qquad \qquad \text{for} \ i=1,\ldots,n-1.
\label{eq_B1652_}
\end{equation}
Since $y_0$ is a Lebesgue density point of $B_0$, then (\ref{eq_B1627}) entails that for $i=1,\ldots,n-1$,
\begin{equation}   \left. \frac{\partial \tilde{u}(f(y))}{\partial y_i} \right|_{y = y_0} = 0 \qquad \text{and} \qquad \left. \frac{\partial |\nabla \tilde{u}(f(y))|}{\partial y_i} \right|_{y = y_0} = 0.
\label{eq_B1114} \end{equation}
Since $J_i(y_0,0) = \frac{\partial \tilde{F}}{\partial y_i} (y_0, 0) = \frac{\partial f}{\partial y_i} (y_0)$ and $ N(y_0,0) = \nabla \tilde{u}(f(y_0))$, we may
rewrite (\ref{eq_B1114}) as
\begin{equation}  \langle N(y_0, 0), J_i(y_0, 0) \rangle = 0 \qquad \text{and} \qquad
\langle \nabla_{J_i(y_0, 0)} \nabla \tilde{u}, N(y_0, 0) \rangle = 0,
  \label{eq_B1639}
\end{equation}
for $i=1,\ldots,n-1$. Now (\ref{eq_B1537}) follows from
(\ref{eq_B1652_}) and (\ref{eq_B1639}). As for the proof of (\ref{eq_B1536}): in view of (\ref{eq_B1652_}) we actually need to prove that
$$ \langle \nabla_{J_i(y_0, 0)} \nabla \tilde{u}, J_k(y_0, 0) \rangle =
\langle \nabla_{J_k(y_0, 0)} \nabla \tilde{u}, J_i(y_0, 0) \rangle \qquad \qquad \text{for} \ i,k=1,\ldots,n-1. $$
The latter relations hold as $\tilde{u}$ is twice differentiable with a symmetric Hessian at
the point $f(y_0) = \tilde{F}(y_0, 0)$.
\end{proof}

Recall the definitions of $\Strain[u], \Strain_{\eps_0}[u]$ and $\alpha_u, \beta_u$ from Section \ref{transport_rays}.

\begin{definition} Let
 $u :\cM \rightarrow \RR$ satisfy $\| u \|_{Lip} \leq 1$ and let $R_0 \subseteq \cM$ be a Borel set.
We say that $R_0$ is a ``seed of a ray cluster'' associated with $u$
if there exist numbers $r_0 \in \RR, \eps_0 > 0$, open sets $U \subseteq \cM, \Omega_0 \subseteq \RR^{n-1}$
and $C^{1,1}$-functions
$\tilde{u}: U \rightarrow \RR, f: \Omega_0 \rightarrow \cM$
for which the following hold:
\begin{enumerate}
\item[(i)] For any $x \in U \cap \Strain_{\eps_0}[u]$ we have that $ \tilde{u}(x) = u(x)$ and $\nabla \tilde{u}(x) = \nabla u(x)$.
\item[(ii)] The $C^{1,1}$-map $f: \Omega_0 \rightarrow \cM$ is one-to-one with
$ f(\Omega_0) = \{ x \in U \, ; \, \tilde{u}(x) = r_0 \}$. The inverse map $f^{-1}: f(\Omega_0) \rightarrow \Omega_0$ is continuous.
\item[(iii)] For almost any point $y \in \Omega_0$, the function $\tilde{u}$ is twice differentiable with a symmetric Hessian at the point $f(y)$.
\item[(iv)] $\displaystyle R_0 \subseteq \left \{ x \in U \cap \Strain_{\eps_0}[u] \, ; \, \tilde{u}(x) = r_0 \right \}$.
\end{enumerate}

If the functions $\alpha_u, \beta_u: R_0 \rightarrow \RR \cup \{ \pm \infty \}$ are continuous,
then we say that $R_0$ is a ``seed of a ray cluster of continuous length''.
\label{def_B1216}
\end{definition}

Note that any Borel set which is contained in a seed of a ray cluster, is in itself a seed of a ray cluster.
Recall from Lemma \ref{lem_A317} that $T^{\circ}[u]$
is the collection of all relative interiors of non-degenerate transport rays associated with $u$,
and that
$T^{\circ}[u]$ is a partition of $\Strain[u]$.

\begin{definition} Let $u :\cM \rightarrow \RR$ satisfy $\| u \|_{Lip} \leq 1$.
A  subset $R \subseteq \Strain[u]$ is a ``ray cluster'' associated with $u$ if there exists $R_0 \subseteq \cM$ which is a seed of
a ray cluster such that
\begin{equation}   R = \left \{ x \in \cM \, ; \, \exists \cI \in T^{\circ}[u] \ \text{such that} \ x \in \cI \ \text{and} \ \cI \cap R_0 \neq \emptyset \right \}. \label{eq_B2030}
\end{equation}
We say that $R$ is a ``ray cluster of continuous length'' if $R_0$ is a seed of a ray cluster of continuous length.
\label{def_B1215}
\end{definition}

When $A \subseteq \RR^n$ is a measurable set and
 $f: A \rightarrow \RR^m$ is locally-Lipschitz, the function $f$ maps measurable sets to measurable sets:
Indeed, any measurable set equals the union of a Lebesgue-null set
and countably many compacts, hence also its image under a locally-Lipschitz map
is the union of a Lebesgue-null set
and countably many compacts.
Therefore, the concept
of a measurable subset of a differentiable manifold $\cM$ is well-defined.
Similarly, the
concepts of a Lebesgue-null set and a Lebesgue density point of a measurable set
in a differentiable manifold $\cM$ are well-defined.
The Lebesgue theorem, stating that almost any point of a measurable set $A$
is a Lebesgue density point of $A$, also applies  in the context of an abstract  differentiable manifold.

\medskip
For a subset $A \subseteq \RR^n$, a function $f: A \rightarrow \RR^m$ and a point $x_0 \in A$, we
say that $f$ is differentiable at $x_0$ if there is a {\it unique} linear map $T: \RR^n \rightarrow \RR^m$ such that
$$ \lim_{A \ni x \rightarrow x_0} |f(x_0) + T (x-x_0) - f(x)| / |x - x_0| = 0. $$
In this case we may speak of the differential of $f$ at $x_0$. For instance, if $f: A \rightarrow \RR^m$
is differentiable at the point $x \in A \subseteq \RR^n$, and $B \subseteq A$ is a measurable
set containing $x$ such that $x$ is a Lebesgue density point of $B$, then $f|_B$ is differentiable at $x$.
In what follows we will usually consider the differential of a function $f: A \rightarrow \RR^m$ only at Lebesgue density points of $A$.

\medskip Similarly, given differentiable manifolds $\cM$ and $\cN$, a subset $A \subseteq \cM$
and a function $f: A \rightarrow \cN$, we may speak about the differentiability of $f$
at the point $p_0 \in A$.
When $f$ is differentiable at $p_0$, we may consider the differential of $f$ at $p_0$,
and we may also consider the  directional derivatives $\partial_v f$ for $v \in T_{p_0} \cM$.
A function defined in a subset of a differentiable manifold is said to be locally-Lipschitz when it is
locally-Lipschitz in any chart.
By the Rademacher theorem and the Kirszbraun theorem (see, e.g., Evans and Gariepy \cite[Section 3.1]{EG}),
any locally-Lipschitz function defined on a measurable subset $A$ of a differentiable manifold,
is differentiable almost-everywhere in $A$.

\medskip
A {\it parallel line-cluster} is a subset $B \subseteq \RR^{n-1} \times \RR$
of the following form: There exist a measurable set $B_0 \subseteq \RR^{n-1}$ and continuous functions $a: B_0 \rightarrow [-\infty, 0)$
and $b: B_0 \rightarrow (0, +\infty]$ such that
\begin{equation} B = \left \{ (y,t) \in \RR^{n-1} \times \RR \, ; \, y \in B_0, \ a_y < t < b_y \right \}, \label{eq_917}
\end{equation}
where $a_y = a(y)$ and $b_y = b(y)$ for $y \in B_0$.
Note that when $y \in B_0$ is a Lebesgue density point of $B_0$, the point  $(y, t) \in B$
is a Lebesgue density point of $B$ for any $t \in (a_{y}, b_y)$.

\medskip
An {\it almost line-cluster} is a subset $B \subseteq \RR^{n-1} \times \RR$
of the form (\ref{eq_917}) where $B_0 \subseteq \RR^{n-1}$ is measurable
and the functions
$a: B_0 \rightarrow [-\infty, 0)$ and $b: B_0 \rightarrow (0, +\infty]$
are only assumed to be measurable, and not continuous.
Note that a parallel line-cluster is always measurable, as well as an almost line-cluster.
We say that a map $F$ is invertible if it is one-to-one and onto.

\begin{proposition} Let $u :\cM \rightarrow \RR$ satisfy $\| u \|_{Lip} \leq 1$.
Suppose that $R \subseteq \Strain[u]$ is a non-empty ray cluster of continuous length.
Then there exist a parallel line-cluster $B \subseteq \RR^{n-1} \times \RR$,
a measurable set $B_0 \subseteq \RR^{n-1}$, functions $a,b: B_0 \rightarrow \RR \cup \{ \pm \infty \}$
 and a locally-Lipschitz, invertible map $F: B \rightarrow R$
with the following properties:
\begin{enumerate}
\item[(i)] The relation (\ref{eq_917}) holds true.
Write $f(y) = F(y,0)$ for $y \in B_0$. Then
the set $R_0 = f( B_0)$ is a seed of a ray cluster satisfying (\ref{eq_B2030}). Additionally,
 \begin{equation} a_y = -\alpha_u(f(y)), \quad b_y = \beta_u(f(y)) \qquad \qquad \qquad \text{for all} \ y \in B_0. \label{eq_B1043A} \end{equation}
\item[(ii)] For any $y \in B_0$, the curve
$$ t \mapsto F(y,t) \qquad \qquad \qquad \qquad t \in (a_y, b_y) $$
is a minimizing geodesic whose image is the relative interior of a transport ray associated with $u$. Furthermore, there exists $r_0 \in \RR$
such that
\begin{equation}
 u(F(y,t)) = t + r_0 \qquad \qquad \qquad \text{for all} \ (y,t) \in B.
 \label{eq_B1104} \end{equation}
  \item[(iii)] For almost any Lebesgue density point $y_0 \in B_0$ the following hold:
The map $F$ is differentiable at $(y_0, t)$ for all
 $t \in (a_{y_0}, b_{y_0})$, and there exist
 Jacobi fields $J_1(y_0,t),\ldots,$ $J_{n-1}(y_0,t)$ along the geodesic
$t \mapsto F(y_0,t)$ in the entire interval $t \in (a_{y_0}, b_{y_0})$ such that for $i=1,\ldots,n-1$,
\begin{equation} J_i(y_0,t) = \frac{\partial F}{\partial y_i}(y_0,t) \qquad \qquad \text{for all} \ t \in (a_{y_0}, b_{y_0}). \label{eq_B1807} \end{equation}
Denoting $N(y_0,t) = \frac{\partial F}{\partial t}(y_0,t) $
we have, at the point $(y_0,0) \in B$,
\begin{equation}  \langle J_i, N \rangle = \langle J_i^{\prime}, N \rangle =0 \qquad \qquad (i=1,\ldots,n-1),
\label{eq_B1459} \end{equation}
and
\begin{equation}  \langle J_i^{\prime}, J_k \rangle = \langle J_k^{\prime}, J_i \rangle \qquad \qquad (i,k=1,\ldots,n-1).
\label{eq_B1621} \end{equation}
Here, $J_i^{\prime}(y_0, 0)$ is the covariant derivative at $t= 0$ of the Jacobi field $t \mapsto J_i(y_0, t)$ along the geodesic curve $t \mapsto
F(y_0, t)$.
\item[(iv)] For $(y,t) \in B$ denote $T(y,t) = \left \{ \langle J_i(y,t), J_k(y,t) \rangle \right \}_{i,k=1,\ldots,n}$, where $J_n := N$.
Then the symmetric matrix $T(y,t)$ is well-defined and positive semi-definite almost everywhere in $B$, and for any Borel set $A \subseteq R$,
\begin{equation} \lambda_\cM(A) = \int_{F^{-1}(A)} \sqrt{\det T(y,t)} dy dt, \label{eq_B1314A} \end{equation}
where $\lambda_{\cM}$ is the Riemannian volume measure in $\cM$.
\end{enumerate}
\label{prop_B1122}
\end{proposition}

\begin{proof} Let $R_0 \subseteq \cM$ be the seed of a ray cluster
of continuous length given by Definition \ref{def_B1215}. Then $R_0$ is a Borel set with
\begin{equation} R = \left \{ x \in \cM \, ; \, \exists \cI \in T^{\circ}[u] \ \text{such that} \ x \in \cI \ \text{and} \ \cI \cap R_0 \neq \emptyset \right \}.
\label{eq_B1228} \end{equation}
Since $R_0$ is a seed of a ray cluster, Definition \ref{def_B1216}
provides us with certain numbers $r_0 \in \RR, \eps_0 > 0$, open sets $U \subseteq \cM, \Omega_0 \subseteq \RR^{n-1}$ and
$C^{1,1}$-functions $\tilde{u}: U \rightarrow \RR, f: \Omega_0 \rightarrow \cM$ such that
\begin{equation}
R_0 \subseteq \left \{ x \in U \cap \Strain_{\eps_0}[u] \, ; \, \tilde{u}(x) = r_0 \right \}.
\label{eq_B1524}
\end{equation}
Additionally, $f$ is a one-to-one map with
$ f(\Omega_0) = \{ x \in U \, ; \, \tilde{u}(x) = r_0 \}.
$
In particular, $R_0 \subseteq f(\Omega_0)$. Denote
$$ B_0 := f^{-1}(R_0) \subseteq \Omega_0.
$$
Since $R_0 \subseteq f(\Omega_0)$ then
\begin{equation} f(B_0) = R_0.
\label{eq_B1234}
\end{equation}
 Since $B_0$ is the preimage of the Borel set $R_0$ under the continuous map $f$, then  $B_0 \subseteq \RR^{n-1}$ is measurable.
 According to (\ref{eq_B1524}) and (\ref{eq_B1234}), for each $y \in B_0$, the point $f(y)$ belongs to $\Strain_{\eps_0}[u] \subseteq
 \Strain[u]$. Since $T^{\circ}[u]$ is a partition of $\Strain[u]$,
 then for any $y \in B_0$  there exists a
unique $\cI = \cI(y) \in T^{\circ}[u]$ for which $f(y) \in \cI$.
In view of (\ref{eq_B1234}), we may rewrite (\ref{eq_B1228}) as follows:
\begin{equation} R = \bigcup_{y \in B_0} \cI(y). \label{eq_B1242} \end{equation}
For any $y \in B_0$, the set $\cI(y)$ is the relative interior of a non-degenerate transport ray.
According to
Corollary \ref{cor_A1217}
there exists
an open set $(a_y, b_y) \subseteq \RR$ containing the origin, with $a_y = -\alpha_u(f(y)), b_y = \beta_u(f(y))$, such that
\begin{equation} \cI(y) = \left \{ \Exp_t \left( \nabla u(f(y)) \right) \, ; \, t \in (a_y, b_y) \right \} \qquad \qquad \qquad  \text{for} \ y \in B_0, \label{eq_B1434} \end{equation}
and such that $t \mapsto \Exp_t( \nabla u(f(y)) )$ is a minimizing geodesic in $t \in (a_y, b_y)$ with
\begin{equation}
u \left( \Exp_t( \nabla u(f(y)) ) \right) = u(f(y)) + t \qquad \qquad \qquad \text{for} \ y \in B_0, t \in (a_y, b_y). \label{eq_B1650}
\end{equation}
The curve $t \mapsto \Exp_t( \nabla u(f(y)) )$ is a geodesic of speed one, so
\begin{equation}
|\nabla u(f(y))| = 1 \qquad \qquad \qquad \text{for} \ y \in B_0.
\label{eq_B1135_}
\end{equation}
Since $R_0$ is a
seed of a ray cluster
of continuous length, then the functions $\alpha_u, \beta_u: R_0 \rightarrow (0, +\infty]$
are continuous. Therefore $b_y = \beta_u(f(y))$ and $a_y = -\alpha_u(f(y))$ are continuous functions of $y \in B_0$, thanks
to (\ref{eq_B1234}) and the continuity of $f$.
 Consequently,
 \begin{equation}  B = \left \{ (y,t) \in \RR^{n-1} \times \RR \, ; \, y \in B_0, \ a_y < t < b_y \right \} \label{eq_B936} \end{equation}
is a parallel line-cluster. According to (\ref{eq_B1524}), (\ref{eq_B1234}) and item (i)  of Definition \ref{def_B1216},
\begin{equation}
u(f(y)) = \tilde{u}(f(y)) = r_0, \quad \nabla \tilde{u}(f(y)) =
\nabla u(f(y)) \qquad \qquad \text{for} \ y \in B_0. \label{eq_B1652} \end{equation}
For $y \in \Omega_0$ and $t \in \RR$ define
\begin{equation}  \tilde{F}(y,t) = \Exp_{t} \left( \nabla \tilde{u}(f(y)) \right), \qquad N(y,t) = \frac{\partial \tilde{F}}{\partial t}(y,t). \label{eq_B917_}
\end{equation}
Since $\cM$ is not necessarily complete, then $(y,t) \mapsto \tilde{F}(y,t)$ and $(y,t) \mapsto N(y,t)$ are well-defined on a maximal open subset of $\Omega_0 \times \RR$ that contains $\Omega_0 \times \{ 0 \}$. The functions $\tilde{u}$ and $f$ are $C^{1,1}$-maps,
and hence $$ \Omega_0 \ni y \mapsto \nabla{\tilde{u}}(f(y)) \in T \cM $$ is locally-Lipschitz.
The exponential map is smooth, and
from (\ref{eq_B917_}) we learn that $\tilde{F}$ is locally-Lipschitz. According to (\ref{eq_B1434}),
the  map $\tilde{F}$ is well-defined on the entire set $B$. Set
$$ F = \tilde{F}|_B, $$
a well-defined, locally-Lipschitz map.
From (\ref{eq_B936}), (\ref{eq_B1652}) and (\ref{eq_B917_}),
\begin{equation} F(y,t) = \tilde{F}(y,t) = \Exp_{t} \left( \nabla \tilde{u}(f(y)) \right) = \Exp_{t} \left( \nabla u(f(y)) \right)
\qquad \text{for all} \ (y,t) \in B. \label{eq_B1432} \end{equation}
We conclude from
(\ref{eq_B1242}), (\ref{eq_B1434}), (\ref{eq_B936}) and (\ref{eq_B1432}) that
$$ R = F(B). $$
Thus $F: B \rightarrow R$ is onto. We argue that for any $y_1, y_2 \in B_0$,
\begin{equation} y_1 \neq y_2 \qquad \Longrightarrow \qquad f(y_1) \not \in \cI(y_2).
\label{eq_B1119}
\end{equation}
Indeed, $u(f(y_1)) = u(f(y_2)) = r_0$ according to (\ref{eq_B1652}). Hence,
if $f(y_1) \in \cI(y_2)$ then by
(\ref{eq_B1434}) and (\ref{eq_B1650}) necessarily $f(y_1) = \Exp_t(\nabla u(f(y_2)))$ for $t = 0$.
Therefore $f(y_1) = f(y_2)$ and consequently  $y_1 = y_2$ as the function $f$ is one-to-one. This establishes (\ref{eq_B1119}). Recalling that $T^{\circ}[u]$ is a partition, we deduce from (\ref{eq_B1119}) that the union in
(\ref{eq_B1242}) is a disjoint union.
Glancing at (\ref{eq_B1434}) and (\ref{eq_B1432}),
we see that the locally-Lipschitz map
$F: B \rightarrow R$ is  one-to-one and hence invertible, as required.

\medskip Let us verify conclusion (i) of the proposition: The relation
(\ref{eq_917}) holds true in view of (\ref{eq_B936}). It follows from (\ref{eq_B1432}) that $F(y,0) = f(y)$
for all $y \in B_0$. By (\ref{eq_B1228}) and (\ref{eq_B1234}), the set $R_0 = f(B_0)$ is a seed of a ray cluster satisfying
(\ref{eq_B2030}). The definition of $a_y$ and $b_y$ above implies (\ref{eq_B1043A}),
 and (i) is proven. We move on to the proof of  conclusion (ii) of the proposition:
The fact that $t \mapsto F(y,t)$ is a minimizing geodesic whose image is the relative
interior of a transport ray follows from (\ref{eq_B1434}) and (\ref{eq_B1432}).
The relation (\ref{eq_B1104}) follows from
(\ref{eq_B1650}), (\ref{eq_B1652}) and (\ref{eq_B1432}).
Thus conclusion (ii) is proven as well.

\medskip In order to obtain
conclusion (iii) we would like to apply Lemma \ref{lem_jacobi}. To this end,
observe  that our definition (\ref{eq_B917_}) of $\tilde{F}(y,t)$ and $N(y,t)$ coincides
with that of Lemma \ref{lem_jacobi}. According to Definition \ref{def_B1216}(iii), for almost any $y_0 \in B_0 \subseteq \Omega_0$, the function $\tilde{u}$
is twice differentiable with a symmetric Hessian at $f(y_0)$. Note that the requirement (\ref{eq_B1627}) of
Lemma \ref{lem_jacobi} is satisfied in view of
(\ref{eq_B1135_}) and (\ref{eq_B1652}).
Thus, from conclusion (ii) of Lemma \ref{lem_jacobi}, for almost any Lebesgue density point $y_0 \in B_0$,
\begin{equation}  J_1(y_0, t) = \frac{\partial \tilde{F}}{\partial y_1}(y_0,t), \ldots, J_{n-1}(y_0, t) = \frac{\partial \tilde{F}}{\partial y_{n-1}}(y_0,t),
 \label{eq_B1821} \end{equation}
are well-defined Jacobi fields along the entire geodesic $t \mapsto \tilde{F}(y_0, t)$ for $t \in (a_{y_0}, b_{y_0})$.
In fact, $(y_0,t)$ is a Lebesgue density point of $B$ for any $t \in (a_{y_0}, b_{y_0})$. Recalling that $F = \tilde{F}|_B$
we conclude from
Lemma \ref{lem_jacobi}(i) that the map $F: B \rightarrow R$ is differentiable at $(y_0, t)$ whenever $t \in (a_{y_0}, b_{y_0})$.
The relation (\ref{eq_B1807}) thus follows from the validity of (\ref{eq_B1821}) for all $t \in (a_{y_0}, b_{y_0})$.
The Jacobi fields $t \mapsto J_1(y_0, t), \ldots, t \mapsto J_{n-1}(y_0, t)$
 also satisfy (\ref{eq_B1459}) and (\ref{eq_B1621}), thanks to Lemma \ref{lem_jacobi}(iii),
and the proof of (iii) is complete.

\medskip We continue with the proof of (iv). First of all, the function $F$ is locally-Lipschitz and hence
differentiable almost everywhere in $B$. According to conclusion (iii) which was proven above, for almost any $(y,t) \in B$,
\begin{equation}  T(y,t) = \left \{ \left \langle J_i(y,t), J_k(y,t) \right \rangle \right \}_{i,k=1,\ldots,n}
= \left \{ \left \langle \frac{\partial F}{\partial y_i} (y,t), \frac{\partial F}{\partial y_k} (y,t) \right \rangle \right \}_{i,k=1,\ldots,n} \label{eq_B1119_}
\end{equation}
where $\partial F / \partial y_n := \partial F / \partial t$. We will use the area formula for Lipschitz maps
from Evans and Gariepy \cite{EG}. Let us recall the relevant theory.
Let $H: \RR^n \rightarrow \RR^n$ be a Lipschitz function. The Jacobian of $H$, denoted by $J_H$,
is well-defined almost everywhere. According to \cite[Section 3.3.3]{EG}, for any measurable function $g: \RR^n \rightarrow [0, \infty)$
and a measurable set $D \subseteq \RR^n$,
\begin{equation}
 \int_{D} g(x) J_H(x) dx = \int_{\RR^n} \left[ \sum_{x \in D \cap H^{-1}(y)} g(x) \right] dy,
 \label{eq_B1719} \end{equation}
 where an empty sum is defined to be zero. We claim that
  in order to define the left-hand side and the right-hand side of (\ref{eq_B1719}),
 it suffices to know the values of $H$ in the set $D$ alone. Indeed, the Jacobian $J_H(x)$ is determined
 by $H|_D$  at any Lebesgue density point $x \in D$ in which $H$ is differentiable.
  The Kirszbraun theorem \cite[Section 3.3.1]{EG} states that any Lipschitz map from $D$ to $\RR^n$
 may be extended to a Lipschitz map from $\RR^n$ to $\RR^n$.
It therefore suffices to
 assume that $H: D \rightarrow \RR^n$ is a Lipschitz
 function in order for (\ref{eq_B1719}) to hold true.
 In fact, it is enough to assume that $H: D \rightarrow \RR^n$
 is only locally-Lipschitz. Indeed, there exist compacts $K_1 \subseteq K_2 \subseteq \ldots$
that are contained in $D$ with
$$ m \left( D \setminus \bigcup_{i=1}^{\infty} K_i \right) = 0, $$
where $m$ is the Lebesgue measure on $\RR^n$. We now apply (\ref{eq_B1719}) with the compact set $K_i$
playing the role of $D$ and use the monotone convergence theorem. This yields (\ref{eq_B1719}) for the original set $D$, even though $H$ is only
locally-Lipschitz. To summarize, when $D \subseteq \RR^n$ is a measurable set and $H: D \rightarrow \RR^n$ is a
locally-Lipschitz, one-to-one map, then for any measurable function $g: \RR^n \rightarrow [0, \infty)$,
\begin{equation}
 \int_{D} g(x) J_H(x) dx = \int_{H(D)} g(H^{-1}(y)) dy.
 \label{eq_B1736} \end{equation}
 Next, what happens if the range of $H$ is not a Euclidean space, but a Riemannian manifold $\cM$? In this case, we claim that
 for any measurable set $D \subseteq \RR^n$ and a locally-Lipschitz map $H: D \rightarrow \cM$ which is one-to-one,
\begin{equation}
 \int_{D} \vphi(x) \sqrt{\det T(x)} dx = \int_{H(D)} \vphi(H^{-1}(y)) d \lambda_{\cM}(y),
 \label{eq_B1741} \end{equation}
 for any measurable $\vphi: \RR^n \rightarrow [0, \infty)$.
 Here, $T(x) = \left( \langle \partial H/ \partial x_i, \partial H/ \partial x_j \rangle \right)_{i,j=1,\ldots,n}$.
 Note that (iv) follows from (\ref{eq_B1119_}) and (\ref{eq_B1741}), with $D = B, H = F$ and $\vphi = 1_{H^{-1}(A)}$. In order to deduce (\ref{eq_B1741})
 from (\ref{eq_B1736}) we need to work in a local chart, and observe that $\sqrt{\det T(x)}$ is the Riemannian
 volume of the parallelepiped spanned by the tangent vectors
 $$ \frac{\partial H}{\partial x_1}, \ldots, \frac{\partial H}{\partial x_n}. $$
 The usual Jacobian $J_H(x)$ is the Euclidean volume of this parallelepiped in our local chart.
 We conclude that $\sqrt{\det T(x)} / J_H(x)$ is precisely the density of the Riemannian volume measure
 $\lambda_{\cM}$ at the point $H(x)$ in our local chart. By setting $$ g(x) = \vphi(x) \sqrt{\det T(x)} / J_H(x), $$
 we deduce (\ref{eq_B1741}) from (\ref{eq_B1736}).
\end{proof}

\begin{remark}{\rm
It suffices to assume that $A \subseteq R$ is a measurable set in order for (\ref{eq_B1314A})
to hold true. In fact, denote by $\theta$ the complete measure on the set $B$ whose density
is $(y,t) \mapsto \sqrt{\det T(y,t)}$. Note also that
the restriction of $\lambda_{\cM}$ to $R$ is a complete measure on $R$.
The validity of (\ref{eq_B1314A}) for all
Borel subsets of $R$ and a standard measure-theoretic argument
show that a subset $A \subseteq R$ is $\lambda_{\cM}$-measurable if and only if $F^{-1}(A)$
is $\theta$-measurable. Therefore, $F$ pushes forward the measure $\theta$ to the restriction of $\lambda_{\cM}$ to the ray cluster $R$.
} \label{rem_1318} \end{remark}

\begin{remark}{\rm What happens if the ray cluster $R$ from
Proposition \ref{prop_B1122} is not assumed to be of {\it continuous length}?
The assumption that the ray cluster $R$ is of continuous length was mainly used
to prove that the set $B$ defined in (\ref{eq_B936}) is a parallel line-cluster.
Without the assumption that $R$ is of continuous length, the functions
$$ b_y = \beta_u(f(y)), \qquad a_y = -\alpha_u(f(y)) $$ are still measurable functions of $y \in B_0$, thanks to Lemma \ref{lem_1055}
and the continuity of $f$. Therefore $B$ is an {\it almost line-cluster}.
We thus see that only minor changes will occur in the conclusion of the proposition,
if the ray cluster $R$ is not assumed to be of continuous length.
One obvious change would be that $B$
becomes an almost line-cluster, and not a parallel line-cluster.
The only additional change
is that
\begin{center} ``for all $t \in (a_{y_0}, b_{y_0})$''
\end{center}
in the second line of (iii) and also in (\ref{eq_B1807}) will be replaced by
\begin{center} ``for almost all $t \in (a_{y_0}, b_{y_0})$''.
\end{center}  Indeed,
the function $F = \tilde{F}|_B$ is differentiable at $(y_0, t)$ and it satisfies the equality in (\ref{eq_B1807}) at any
point $(y_0, t) \in B$ which is a Lebesgue density point of $B$. By the Lebesgue density theorem,
for almost any $y_0 \in B$ and for almost any $t \in (a_{y_0}, b_{y_0})$, the point $(y_0, t) \in B$ is a
Lebesgue density point of $B$.
To conclude, we are allowed to apply Proposition \ref{prop_B1122}, with the aforementioned tiny changes, even if the ray cluster $R$ is not assumed to be of continuous length.
\label{rem_942}
}\end{remark}

For a subset $A \subseteq \cM$  define $\Ends(A) \subseteq \cM$ to be the union
of all relative {\it boundaries} of transport rays intersecting $A$. In other words,
a point $x \in \cM$
belongs to $\Ends(A)$ if and only if there exists a transport ray $\cI \in T[u]$,
whose relative boundary contains $x$, such that
$ A \cap \cI \neq \emptyset$.

\begin{lemma} Let $u :\cM \rightarrow \RR$ satisfy $\| u \|_{Lip} \leq 1$ and let $R \subseteq \Strain[u]$ be a ray cluster. Then,
$$ \lambda_{\cM}(\Ends(R)) = 0. $$
\label{lem_1054}
\end{lemma}

\begin{proof} We can assume that $R \neq \emptyset$.
We may apply Proposition \ref{prop_B1122}(ii) thanks to Remark
\ref{rem_942}. Whence,
\begin{equation} R = \left \{ F(y,t) \, ; \, y \in B_0, \  a_y < t < b_y \right \}. \label{eq_B948_} \end{equation}
Furthermore,  $F = \tilde{F}|_B$ where $\tilde{F}$ as defined in (\ref{eq_B917_}) is a locally-Lipschitz map which is well-defined in a maximal open subset
of $\Omega_0 \times \RR$ containing $\Omega_0 \times \{ 0 \}$. We claim that
\begin{equation}
\Ends(R) = \left \{ \tilde{F}(y,t) \, ; \, y \in B_0,\ t \in \RR \cap \{ a_y, b_y \}, \ \tilde{F}(y,t) \ \text{is well-defined} \right \}.
\label{eq_B949_}
\end{equation}
Indeed, fix an arbitrary point $x \in R$.
Since $R \subseteq \Strain[u]$,
then according to Lemma
\ref{lem_A1031}, there is a unique transport ray $\cI \in T[u]$ containing $x$.
The relative interior of $\cI$ contains the point $x$.
By Proposition \ref{prop_B1122}(ii), the relative interior of $\cI$ must take the form
\begin{equation}  \left \{ F(y,t) \, ; \, t \in (a_y, b_y) \right \} \label{eq_B1016A} \end{equation}
for a certain $y \in B_0$.
The transport ray $\cI \subseteq \cM$ is a closed set.
Recall that $F = \tilde{F}|_B$, and that the curve $t \mapsto F(y,t)$
 is a minimizing geodesic in $t \in (a_y, b_y)$.
  We thus deduce from
 (\ref{eq_B917_}), (\ref{eq_B1016A}) and Lemma \ref{lem_1102} that
\begin{equation} \cI = \left \{ \tilde{F}(y,t) \, ; \, t \in \RR \cap [a_y, b_y], \ \tilde{F}(y,t) \ \text{is well-defined} \right \}. \label{eq_B1731Q}
\end{equation}
Since $x \in R$ was an arbitrary point, the relation (\ref{eq_B949_})  follows
from the representation (\ref{eq_B1731Q}) of the unique transport ray $\cI$ containing $x$. Consider the set
\begin{equation}  \left \{ (y,t) \in B_0 \times \RR \, ; \, t \in \{ a_y, b_y \},  \ \tilde{F}(y,t) \ \text{is well-defined} \right \}. \label{eq_B951_}
\end{equation}
This set is contained in the union of two graphs of measurable functions, and hence it is a set of measure zero in $\RR^{n-1} \times \RR$.
Since $\Ends(R)$ is the image of the set in (\ref{eq_B951_}) under the locally-Lipschitz map $\tilde{F}$, then $\Ends(R)$ is a null-set
in the $n$-dimensional manifold $\cM$.
\end{proof}

\subsection{Decomposition into ray clusters}
\label{sec_decomp}
\setcounter{equation}{0}

As before, we write $\lambda_{\cM}$ for the Riemannian
volume measure on the geodesically-convex, Riemannian manifold $\cM$.
Our main result in this subsection is the following:

\begin{proposition} Let $u: \cM \rightarrow \RR$ satisfy $\| u \|_{Lip} \leq 1$. Then there exists a countable
family $\{ R_i \}_{i=1,\ldots,\infty}$ of disjoint ray clusters of continuous length such that
$$  \lambda_{\cM} \left( \Strain[u] \setminus \left( \bigcup_{i=1}^{\infty} R_i \right) \right) = 0. $$
\label{cfm}
\end{proposition}

We begin the proof of Proposition \ref{cfm} with the following lemma.

\begin{lemma} Let $u :\cM \rightarrow \RR$ satisfy $\| u \|_{Lip} \leq 1$. Let $R \subseteq \Strain[u]$ be any ray cluster
associated with $u$.
Then $R$ is a Borel subset of $\cM$.
\label{cor_2108}
\end{lemma}

\begin{proof} We may assume that $R \neq \emptyset$. According to Proposition \ref{prop_B1122} and Remark \ref{rem_942} we
know that $R = F(B)$
where $B$ is an almost-line cluster. Let $R_0 \subseteq \cM$ and $r_0 \in \RR$ be as in
Proposition \ref{prop_B1122}. We claim that a given point $x \in \Strain[u]$ belongs to $R$ if and only if the following two conditions are met:
\begin{enumerate}
\item[(A)] $\displaystyle r_0 - u(x) \in (- \alpha_u(x), \beta_u(x))$.
\item[(B)] $\displaystyle \Exp_{r_0 - u(x)}(\nabla u(x)) \in R_0$.
\end{enumerate}
In order to prove this claim,  assume that $x \in \Strain[u]$ satisfies conditions (A) and (B).
Since $T^{\circ}[u]$ is a partition of $\Strain[u]$, there exists $\cI \in T^{\circ}[u]$ such that $x \in \cI$. From (A) and Corollary \ref{cor_A1217} the point $\Exp_{r_0 - u(x)}(\nabla u(x))$ belongs to $\cI$,
while condition (B) shows that this point belongs to $R_0$. Hence $\cI \cap R_0 \neq \emptyset$.
From Definition \ref{def_B1215} we obtain that  $\cI \subseteq R$ and consequently $x \in R$. Conversely, assume that $x \in R$.
According to  Proposition \ref{prop_B1122} there exists $(y,t)\in B$ for which  $F(y,t) = x$ and $u(x) = t + r_0$.
Additionally, $$ \alpha_u(x) = t - a_y, \qquad \beta_u(x) = b_y - t, $$ in the notation of Proposition \ref{prop_B1122}.
Since $B$ is an almost-line cluster,  then
 $0 \in (a_y, b_y)$ and consequently $r_0 - u(x) = -t \in (a_y - t, b_y - t) = (-\alpha_u(x), \beta_u(x))$. We have thus verified condition (A).
By Proposition \ref{prop_B1122} and Corollary \ref{cor_A1217}, we have
$ R_0 \ni F(y,0) = \Exp_{r_0 - u(x)}(\nabla u(x))$, and (B) follows as well.

\medskip Recall that the set $\Strain[u]$ is Borel  according to Lemma \ref{lem_1055},
as well as the functions $\alpha_u, \beta_u: \cM \rightarrow \RR \cup \{ \pm \infty \}$. Since $u$ is continuous, then the collection of all $x \in \Strain[u]$ satisfying
condition (A) is a Borel set. As for condition (B), the set $R_0$ is a seed of a ray cluster and by definition
it is a Borel set. Consider the partially-defined function
\begin{equation}  \Strain[u] \ni x \mapsto \Exp_{r_0 - u(x)}(\nabla u(x)) \in \cM. \label{eq_B938A} \end{equation}
We claim that this function is well-defined on a Borel subset of $\Strain[u]$, and that it is a Borel map.
Indeed, Lemma \ref{lem_1046} shows that the Lipschitz function $u$ is differentiable in the Borel set $\Strain[u]$.
Consequently $\nabla u: \Strain[u] \rightarrow T \cM$ is a well-defined Borel map, as it may be represented as a pointwise
limit of Borel maps. The exponential map is continuous
and the domain of definition of the partially-defined map
$$ T \cM \times \RR \ni (v,t) \mapsto \Exp_t(v) \in \cM $$
is an open set. Hence the map in (\ref{eq_B938A}) is a Borel map
which is defined on a Borel subset of $\Strain[u]$. We conclude that the collection of all $x \in \Strain[u]$ satisfying condition (B) is Borel, being
the preimage
of the Borel set $R_0$ under the Borel map (\ref{eq_B938A}). Therefore the set $R \subseteq \Strain[u]$,
which is defined by conditions (A) and (B), is a Borel set.\end{proof}

\begin{lemma} Let $u: \cM \rightarrow \RR$ satisfy $\| u \|_{Lip} \leq 1$. Assume that $R, R_1, R_2, \ldots,  R_L \subseteq \Strain[u]$
are ray clusters. Then also $R \setminus (\bigcup_{i=1}^{L}) R_i$ is a ray cluster. \label{lem_B1747}
\end{lemma}

\begin{proof} Denote by $R_0$ the seed of the ray cluster $R$
provided by Definition \ref{def_B1215}. Then $R_0$ is a Borel set.
Lemma \ref{cor_2108} implies that $\tilde{R}_0 = R_0 \setminus (\cup_{i=1}^{L} R_i)$
is a Borel set as well. By the remark following Definition \ref{def_B1216},
the set $\tilde{R}_0$ is a seed of a ray cluster associated with $u$. In fact, the set $\tilde{R}_0$
is the seed of the ray cluster $R \setminus (\cup_{i=1}^{L} R_i)$, as follows from Definition \ref{def_B1215}
and the fact that $T^{\circ}[u]$ is a partition of $\Strain[u]$.
\end{proof}

The equality of the mixed second derivatives of $C^{1,1}$-functions, stated in the following lemma, is of
great importance to us.

\begin{lemma} Let $U \subseteq \RR^n$ be an open set and let $f: U \rightarrow \RR$
be a $C^{1,1}$-function. Then for $i,j=1,\ldots,n$, the functions $\partial_i f$ and $\partial_j f$ are differentiable almost everywhere in $U$, with
\begin{equation} \partial_i \left( \partial_j f \right) \, = \,
\partial_j \left( \partial_i f \right)
\qquad \qquad \text{almost everywhere in} \ U. \label{eq_1653} \end{equation}
\label{lem_1357}
\end{lemma}

\begin{proof} Let $x_0 \in U$. It suffices to prove the lemma in an open neighborhood
of $x_0$, in which $f$ and $\partial_1 f, \ldots, \partial_n f$ are Lipschitz functions.
By the Rademacher theorem, the functions $\partial_1 f,\ldots, \partial_n f$ are differentiable almost everywhere in $U$.
By considering slices of $U$, we see that it
suffices to prove (\ref{eq_1653}) assuming that $n=2$ and that $U$ is a rectangle parallel to the axes, of the form
$$ U = \left \{ (x,y) \in \RR^2 \, ; \, a < x < b, \, c < y < d \right \}. $$
Denote
$$ h = \frac{\partial}{\partial x} \left( \frac{\partial f}{\partial y} \right). $$
Since $\partial f / \partial y$ is Lipschitz, then $h$ is an $L^{\infty}$-function. Furthermore, for any $(x,y) \in U$,
$$ \frac{\partial f}{\partial y}(x,y) = \frac{\partial f}{\partial y}(a,y) + \int_a^x h(t,y) dt. $$
Integrating with respect to the $y$-variable we see that for any $(x,y) \in U$,
\begin{equation}
 f(x,y) = f(x,c) + \int_c^y \frac{\partial f}{\partial y}(x,s) ds = f(x,c) + \int_c^y \frac{\partial f}{\partial y}(a,s) ds +
\int_{[a,x] \times [c, y]} h, \label{eq_1651} \end{equation}
where the use of Fubini's theorem is legitimate as $h$ is an $L^{\infty}$-function on $U$.
Differentiating (\ref{eq_1651}) with respect to $x$, we deduce
that the Lipschitz function $\partial f / \partial x$ satisfies
\begin{equation}
\frac{\partial f}{\partial x}(x,y) = \frac{\partial f}{\partial x}(x,c) + \int_c^y h(x,s) ds  \label{eq_12121}
\end{equation}
almost everywhere in $U$. Both the left-hand side and the right-hand side of (\ref{eq_12121}) are differentiable
with respect to $y$ almost everywhere in $U$. Therefore, by differentiating (\ref{eq_12121}) with respect to $y$ we obtain
$$ \frac{\partial}{\partial y} \left( \frac{\partial f}{\partial x} \right) = h $$
almost everywhere in $U$. Thus (\ref{eq_1653}) is proven.
\end{proof}

\begin{corollary}
Let $f: \cM \rightarrow \RR$ be a $C^{1,1}$-function. Then
 the vector field $\nabla f$ is differentiable almost-everywhere in $\cM$, and for almost any $p \in \cM$,
\begin{equation}  \langle \nabla_v (\nabla f),  w \rangle \, = \,
\langle \nabla_w (\nabla f), v \rangle \qquad \qquad \qquad \text{for} \ v,w \in T_p \cM. \label{eq_B1357}
\end{equation}
Here,
by ``almost-everywhere''
we refer to the Riemannian volume measure $\lambda_{\cM}$.
\label{cor_B1445}
\end{corollary}

\begin{proof} Working in a local chart, we may replace $\cM$ by an open set $U \subseteq \RR^n$ equipped with a Riemannian metric tensor.
Since $f: U \rightarrow \RR$ is a $C^{1,1}$-function, Lemma \ref{lem_1357} implies that the functions $\partial_1 f, \ldots, \partial_n f$
are differentiable almost everywhere, and
\begin{equation}  \partial_i (\partial_j f) = \partial_j (\partial_i f) \label{eq_1421_} \end{equation}
almost everywhere in $U$. The Leibnitz rule applies at any point where the involved functions are differentiable
and hence,
\begin{align*} \langle \nabla_{\partial_i}  (\nabla f),  \partial_j \rangle  & -  \langle \nabla_{\partial_j} (\nabla f), \partial_i \rangle
 \\ & =  \partial_i \langle \nabla f, \partial_j \rangle  -  \partial_j \langle \nabla f, \partial_i \rangle
- \langle \nabla f, \nabla_{\partial_i} \partial_j - \nabla_{\partial_j} \partial_i \rangle  = \partial_i (\partial_j f) - \partial_j (\partial_i f)
\end{align*}
at any point in which $\partial_1 f, \ldots, \partial_n f$ are differentiable. Now (\ref{eq_B1357}) follows
from the validity of (\ref{eq_1421_}) almost everywhere in $U$.
\end{proof}

\begin{lemma} Let $u: \cM \rightarrow \RR$ satisfy $\| u \|_{Lip} \leq 1$ and let $\eps > 0$ and $p \in \Strain_{\eps}[u]$.
Then there exist an open set $V \subseteq \cM$ containing $p$ and a ray cluster $R \subseteq \cM$ such that
\begin{equation}  \Strain_{\eps}[u] \cap V \subseteq R.  \label{eq_B1557}
\end{equation}
\label{lem_B1745}
\end{lemma}

\begin{proof} Set $\eps_0 = \eps / 2$.
Applying Theorem \ref{prop_939}, we find $\delta > 0$
and a $C^{1,1}$-function $\tilde{u}: B_{\cM}(p, \delta) \rightarrow \RR$ such that
\begin{equation} x \in B_{\cM}(p,\delta) \cap \Strain_{\eps_0}[u] \qquad \Longrightarrow \qquad \tilde{u}(x) = u(x), \quad \nabla \tilde{u}(x) = \nabla u(x). \label{eq_B1401} \end{equation}
We would like to apply the implicit function theorem, in the form of Lemma \ref{lem_1345}(iii).
Decreasing $\delta$ if necessary, we may assume that $B_{\cM}(p, \delta)$  is contained in a single chart
of the differentiable manifold $\cM$. Since $p \in \Strain_{\eps}[u] \subseteq \Strain_{\eps_0}[u]$ then $p$ belongs to the relative
interior of some transport ray. From (\ref{eq_B1401}) and Lemma \ref{lem_1046},
\begin{equation}  \nabla \tilde{u}(p) = \nabla u(p) \neq 0 \qquad \text{and} \qquad \tilde{u}(p) = u(p). \label{eq_B1428} \end{equation}
We may apply Lemma \ref{lem_1345}(iii) in the local chart, thanks to (\ref{eq_B1428}). We conclude from Lemma \ref{lem_1345}(iii)
that there exist an open set
\begin{equation} U \subseteq B_{\cM}(p, \delta) \label{eq_B1835} \end{equation} containing $p$, an open set $\Omega = \Omega_0 \times (a,b) \subseteq \RR^{n-1} \times \RR$
and a $C^{1,1}$-diffeomorphism $G: \Omega \rightarrow U$ with
\begin{equation}
\tilde{u}(G(y,t)) = t \qquad \qquad \qquad \text{for} \ (y,t) \in \Omega_0 \times (a,b).
\label{eq_B1443}
\end{equation}
Since $p \in U$ and $G: \Omega \rightarrow U$ is
onto,
then (\ref{eq_B1428}) and (\ref{eq_B1443}) imply that
\begin{equation}
u(p) = \tilde{u}(p) \in (a,b). \label{eq_B1701} \end{equation}
The set $U$ is an open neighborhood of  $p$, hence there exists $0 < \eta < \eps_0$
with
\begin{equation}
B_{\cM}(p, \eta) \subseteq U. \label{eq_B1702}
\end{equation}
According to  Corollary \ref{cor_B1445}, for almost any $x \in U$, the $C^{1,1}$-function $\tilde{u}$ is twice differentiable
with a symmetric Hessian at $x$. Since $G$ is a $C^1$-diffeomorphism, then for almost any $(y,t) \in \Omega_0 \times (a,b)$,
the function $\tilde{u}$ is twice differentiable
with a symmetric Hessian at the point $G(y,t)$. From the latter fact and from (\ref{eq_B1701}) we conclude that
there exists
\begin{equation}
t_0 \in (a,b) \cap \left( u(p) - \frac{\eta}{2}, u(p) + \frac{\eta}{2} \right) \label{eq_B1703}
\end{equation} with the following property: For almost any $y \in \Omega_0 \subseteq \RR^{n-1}$, the function $\tilde{u}$ is
twice differentiable
with a symmetric Hessian at the point $G(y,t_0)$.
Denote
\begin{equation}  R_0 = \left \{ x \in U \cap \Strain_{\eps_0}[u] \, ; \, \tilde{u}(x) = t_0 \right \}.
\label{eq_B1523}
\end{equation}
Lemma \ref{lem_1055} implies that
$\Strain_{\eps_0}[u] = \{ x \in \cM \, ; \, \ell_u(x) > \eps_0 \}$ is a Borel set. From (\ref{eq_B1523}), the set $R_0 \subseteq \cM$ is also Borel.
We claim that $R_0$ is a seed of a ray cluster in the sense of Definition \ref{def_B1216}.
In order to prove our claim we define  $r_0 := t_0$ and set
$$ f(y) := G(y,t_0) \qquad \qquad \qquad (y \in \Omega_0). $$
Since $G$ is a $C^{1,1}$-diffeomorphism onto $U$, then the $C^{1,1}$-function $f$ is one-to-one with a continuous inverse. The relation (\ref{eq_B1443}) implies that
\begin{equation}
f(\Omega_0) = \left \{ x \in U \, ; \, \tilde{u}(x) = t_0 \right \}
= \left \{ x \in U \, ; \, \tilde{u}(x) = r_0 \right \}. \label{eq_B1738}
\end{equation}
Let us verify that
the numbers  $r_0 \in \RR, \eps_0 > 0$, the open sets $U \subseteq \cM, \Omega_0 \subseteq \RR^{n-1}$
and the $C^{1,1}$-functions
$\tilde{u}: U \rightarrow \RR, f: \Omega_0 \rightarrow \cM$
satisfy the requirements of Definition \ref{def_B1216}.
Indeed, by the choice of $t_0$ we verify requirement (iii) of Definition \ref{def_B1216}.
By using (\ref{eq_B1738}) and the preceding sentence we obtain Definition \ref{def_B1216}(ii).
The relation  (\ref{eq_B1523}) and the fact that $r_0 = t_0$ show that Definition \ref{def_B1216}(iv) holds as well.
From (\ref{eq_B1401}) and (\ref{eq_B1835}) we deduce Definition \ref{def_B1216}(i).
Thus $R_0$ is a seed of a ray cluster associated with $u$.
Set
\begin{equation}   R = \left \{ x \in \cM \, ; \, \exists \cI \in T^{\circ}[u] \ \text{such that} \ x \in \cI \ \text{and} \ \cI \cap R_0 \neq \emptyset \right \}. \label{eq_B1529} \end{equation}
Then $R \subseteq \Strain[u]$ is a ray cluster, according to Definition \ref{def_B1215}.
We  still need to find an open set $V \subseteq \cM$ containing $p$ for
which
(\ref{eq_B1557}) holds true. Let us define
\begin{equation}
V = \left \{ x \in B_{\cM} \left(p, \frac{\eta}{2} \right) \, ; \,  |u(x) - t_0| < \eta/2 \right \},
\label{eq_B1706} \end{equation}
which is an open set containing $p$ in view of (\ref{eq_B1703}).
In order to prove (\ref{eq_B1557}), we recall that $\eps = 2 \eps_0$
and  let
$x \in \Strain_{\eps}[u] \cap V$ be an arbitrary point. Since $\ell_u(x) > \eps$, then Corollary
\ref{cor_A1217} implies that there exist
$\cI \in T^{\circ}[u]$ and a minimizing geodesic
$\gamma:[-\eps, \eps] \rightarrow \cM$ with
\begin{equation} \gamma(0) = x
\label{eq_B1831} \end{equation} such that
\begin{equation}
\gamma \left( [-\eps, \eps] \right) \subseteq \cI, \label{eq_1630}
\end{equation}
and such that
\begin{equation}
u(\gamma(t)) = u(x) + t \qquad \qquad \qquad
\text{for} \ t \in [-\eps, \eps].
\label{eq_B1723}
\end{equation}
It follows from (\ref{eq_B1723}) and the definition of $\alpha_u, \beta_u$ and $\ell_u$ in Section \ref{transport_rays} that
\begin{equation}
\ell_u(\gamma(t)) \geq \eps - |t| \qquad \qquad \qquad \text{for} \ t \in (-\eps, \eps).
\label{eq_B1729}
\end{equation}
Since $x \in V$, then $|u(x) - t_0| < \eta/2$ according to (\ref{eq_B1706}).
Denoting $t_1 = t_0 - u(x)$, we have
\begin{equation}
|t_1| = |u(x) - t_0| < \eta/2 < \eps_0 = \eps/2,
\label{eq_B1726}
\end{equation}
where $\eta < \eps_0$ according to the line before (\ref{eq_B1702}).
From (\ref{eq_B1723}) and (\ref{eq_B1726}) we see that $u(\gamma(t_1)) = u(x) + t_1 = t_0$. From (\ref{eq_B1729}) and (\ref{eq_B1726}) it follows
that $\ell_u(\gamma(t_1)) > \eps/2 = \eps_0$. Therefore,
\begin{equation}  \gamma(t_1) \in \Strain_{\eps_0}[u] \cap \left \{ x \in \cM \, ; \, u(x) = t_0 \right \}.
\label{eq_B1731} \end{equation}
Furthermore, $x \in V$ and hence $d(x, p) < \eta/2$ by (\ref{eq_B1706}). Since $\gamma$ is a unit speed geodesic, then
from (\ref{eq_B1831}) and (\ref{eq_B1726}),
\begin{equation}
d(\gamma(t_1), p) \leq d(\gamma(0), p) + |t_1| = d(x,p) + |t_1| < \eta/2 + \eta/2 = \eta.
\label{eq_B1733}
\end{equation}
We learn from (\ref{eq_B1702}) and (\ref{eq_B1733}) that $\gamma(t_1) \in U$. From (\ref{eq_B1401}), (\ref{eq_B1835}) and (\ref{eq_B1731}),
we thus obtain that $\tilde{u}(\gamma(t_1)) = u(\gamma(t_1)) = t_0$.
By using (\ref{eq_B1523}) and (\ref{eq_B1731}), we finally obtain that
$$\gamma(t_1) \in R_0.
$$
Note also that $\gamma(t_1) \in \cI$, thanks to (\ref{eq_1630}) and (\ref{eq_B1726}).
We have thus found a point $\gamma(t_1) \in \cI \cap R_0$, and hence $\cI \cap R_0 \neq \emptyset$.
Recalling that $\cI \in T^{\circ}[u]$ we learn from (\ref{eq_B1529}) that $\cI \subseteq R$. Since $x = \gamma(0) \in \cI$
by (\ref{eq_B1831}) and (\ref{eq_1630}), then  $x \in R$. However, $x$ was an arbitrary
point in $\Strain_{\eps}[u] \cap V$, and hence the proof of
 (\ref{eq_B1557}) is complete.
\end{proof}

\begin{lemma} Let $u: \cM \rightarrow \RR$ satisfy $\| u \|_{Lip} \leq 1$. Then there exists a countable
family $\{ R_i \}_{i=1,2,\ldots}$ of disjoint ray clusters associated with $u$ such that
\begin{equation}  \Strain[u] = \bigcup_{i=1}^{\infty} R_i. \label{eq_1122}
\end{equation}
\label{lem_B952}
\end{lemma}

\begin{proof} In order to prove the lemma, it suffices to find ray clusters
$\tilde{R}_i \subseteq \cM$ for $i=1,2,\ldots$ which are not necessarily disjoint, such that
\begin{equation}  \Strain[u] \subseteq \bigcup_{i=1}^{\infty} \tilde{R}_i. \label{eq_B1022}
\end{equation}
Indeed, any ray cluster $R$ is automatically contained in $\Strain[u]$.
By setting $R_i = \tilde{R}_i \setminus \cup_{j < i} \tilde{R}_j$ and using Lemma
\ref{lem_B1747}, we deduce (\ref{eq_1122}) from (\ref{eq_B1022}).
We thus focus on the proof of (\ref{eq_B1022}). Recall from Section \ref{transport_rays} that
$ \Strain[u] = \bigcup_{k=1}^{\infty} \Strain_{1/k}[u]$. Hence, in order to prove (\ref{eq_B1022}),
it suffices to fix $\eps > 0$ and to find ray clusters $R_1, R_2, \ldots$ with
\begin{equation}
\Strain_{\eps}[u] \subseteq \bigcup_{i=1}^{\infty} R_i. \label{eq_B1014_}
\end{equation}
Let us fix $\eps > 0$. We need to find ray clusters $R_1,R_2,\ldots$ satisfying (\ref{eq_B1014_}).
For $p \in \Strain_{\eps}[u]$
let us write $V_{p, \eps} = V \subseteq \cM$ for the open set containing $p$ that is provided by
Lemma \ref{lem_B1745}. Then for any $p \in \Strain_{\eps}[u]$ there is a ray cluster
$R = R_{p, \eps} \subseteq \cM$ such that
\begin{equation}  \Strain_{\eps}[u] \cap V_{p, \eps} \subseteq R_{p, \eps}.  \label{eq_B1007_}
\end{equation}
Consider all open sets of the form $V_{p, \eps}$ where $p \in \Strain_{\eps}[u]$.
This collection is an open cover of $\Strain_{\eps}[u]$.
Recall that $\cM$ is second-countable. Hence we may find an open sub-cover of $\Strain_{\eps}[u]$ which is countable. That is, there
exist points $p_1,p_2,\ldots \in \Strain_{\eps}[u]$ such that
\begin{equation}  \Strain_{\eps}[u] \subseteq \bigcup_{i=1}^{\infty} V_{p_i, \eps}. \label{eq_B1018_}
\end{equation}
From (\ref{eq_B1007_}) and (\ref{eq_B1018_}) we conclude that the ray clusters $R_i = R_{p_i, \eps}$
satisfy (\ref{eq_B1014_}), and the lemma is proven.
\end{proof}

\begin{proof}[Proof of Proposition \ref{cfm}]
In view of Lemma \ref{cor_2108} and
Lemma \ref{lem_B952}, all that remains is to prove the following: For
any ray cluster $R \subseteq \cM$ with $\lambda_{\cM}(R) > 0$, there exist disjoint ray clusters of {\it continuous
length} $\{ R_i \}_{i=1,\ldots,\infty}$, all contained in $R$, such that
\begin{equation}  \lambda_{\cM} \left( R \setminus \left( \bigcup_{i=1}^{\infty} R_i \right) \right) = 0. \label{eq_B1122}
\end{equation}
According to Remark \ref{rem_942}, we may apply Proposition \ref{prop_B1122} for the ray cluster $R$.
Let $B$ be the almost-line cluster that is provided by Remark \ref{rem_942} and Proposition \ref{prop_B1122},
and let $F, f, B_0, a, b$ be as in Proposition \ref{prop_B1122}.
From  Proposition \ref{prop_B1122}(i), the set $R_0 = f( B_0 )$ is a seed of a ray cluster.
The set $B_0 \subseteq \RR^{n-1}$ is a measurable set, and $a: B_0 \rightarrow [-\infty, 0)$ and $b: B_0 \rightarrow
(0, +\infty]$ are measurable functions. By Luzin's theorem from real analysis,
there exist disjoint $\sigma$-compact subsets $\tilde{B}^{(k)}_0 \subseteq B_0$ for $k=1,2,\ldots$ such that
\begin{equation}  m \left( B_0 \setminus \left( \bigcup_{k=1}^{\infty} \tilde{B}^{(k)}_0 \right) \right) = 0, \label{eq_B1207}
\end{equation}
while for any $k \geq 1$, the functions $a|_{\tilde{B}^{(k)}_0}$ and $b|_{\tilde{B}^{(k)}_0}$ are continuous.
Here, $m$ is the Lebesgue measure on $\RR^{n-1}$.
Note that  $\tilde{R}^{(k)} := f( \tilde{B}^{(k)}_0 )$ is a $\sigma$-compact set for any $k \geq 1$,
being the image of a $\sigma$-compact set under a continuous map.
By the remark following Definition \ref{def_B1216},
the set  $\tilde{R}^{(k)} \subseteq R_0$ is a seed of a ray cluster.

\medskip From our  construction the functions $a_y = -\alpha_u(f(y))$ and $b_y = \beta_u(f(y))$ are
continuous functions of $y \in \tilde{B}_0^{(k)}$, for any $k \geq 1$. From Definition \ref{def_B1216}(ii), the function $f^{-1}$
is continuous on $R_0$, and therefore the functions $\alpha_u, \beta_u$ are continuous on
$\tilde{R}^{(k)} = f( \tilde{B}^{(k)}_0 )$ for any $k \geq 1$. This shows that
$\tilde{R}^{(k)}$ is actually a seed of a ray cluster of continuous length.
 The function $f$ is one-to-one, and
therefore $\tilde{R}^{(1)},\tilde{R}^{(2)},\ldots$
are pairwise-disjoint.

\medskip For $k \geq 1$, define $R_k$ to be the union of all relative interiors of transport
rays intersecting $\tilde{R}^{(k)}$. The
sets $R_1,R_2,\ldots$ are pairwise-disjoint and are contained in $R$,
according to Proposition \ref{prop_B1122}(ii).
From Definition \ref{def_B1215}, the sets $R_1,R_2,\ldots$ are ray clusters of continuous length, while
Lemma \ref{cor_2108} implies the measurability of these sets.
The desired relation
(\ref{eq_B1122}) holds true in view of (\ref{eq_B1207}) and Proposition \ref{prop_B1122}(iv).
This completes the proof.
\end{proof}

\subsection{Needles and Ricci curvature}
\label{ricci}
\setcounter{equation}{0}

We begin this section with an addendum to Proposition \ref{prop_B1122}.

\begin{lemma} We work under the notation and assumptions of Proposition \ref{prop_B1122}.
Let $y = y_0 \in B_0$ be a Lebesgue density point of $B_0$ for which the conclusions of Proposition \ref{prop_B1122}(iii)
hold true.  Then either for all $t \in (a_y, b_y)$ the vectors
$$ J_1(y,t),\ldots,J_{n-1}(y,t) \in T_{F(y,t)} \cM $$
are linearly independent, or else for all $t \in (a_y, b_y)$, these vectors are linearly dependent.
\label{lem_B1811}
\end{lemma}

\begin{proof}
Fix $\lambda_1,\ldots,\lambda_{n-1} \in \RR$ and denote
$$ J(y,t) = \sum_{i=1}^{n-1} \lambda_i J_i(y,t) \qquad \qquad \qquad \text{for} \ t \in (a_y, b_y). $$
We would like to show that the set $ \left \{ t \in (a_y, b_y) \, ; \, J(y,t) = 0 \right \} $
is an open set. Assume that $t_1 \in (a_y, b_y)$ satisfies
\begin{equation}  J(y, t_1) = 0. \label{eq_B1755} \end{equation}
We need to prove that $J(y,t) = 0$ for $t$ in a small neighborhood of $t_1$.
To this end, denote $v = (\lambda_1,\ldots, \lambda_{n-1}) \in \RR^{n-1}$. Since $y \in B_0$ is a Lebesgue density point of $B_0 \subseteq \RR^{n-1}$, then there exists
a $C^1$-curve $\gamma: (-1,1) \rightarrow \RR^{n-1}$ with $\gamma(0) = y$ and $\dot{\gamma}(0) = v$, such that
the set
$$ I = \left \{ s \in (-1,1) \, ; \, \gamma(s) \in B_0 \right \} $$
has an accumulation point at zero. We are going to
view $\gamma$ as a map from $I$ to $B_0$, and we will never use the values of $\gamma$ outside $I$.
Thus, from now on when we write $\dot{\gamma}(0) = v$, we actually mean that
$$ \lim_{I \ni s \rightarrow 0} \frac{\gamma(s) - \gamma(0)}{s} = v. $$
We plan to apply the geometric lemma of Feldman and McCann, which is
Lemma \ref{lem_849} above. Set
\begin{equation}  p = F(y, t_1) \in \cM. \label{eq_B2033A} \end{equation} Let $\delta_1 = \delta_1(p) > 0$ be the parameter
provided by Lemma \ref{lem_849}. Fix
$\eps > 0$ with
\begin{equation}  \eps < \min \{ \delta_1, b_y - t_1, t_1 - a_y \}.
\label{eq_B1038} \end{equation}
Then $a_y < t_1 - \eps$ while $b_y > t_1 + \eps$.
Since $B$ is a parallel line cluster, then the functions $a$ and $b$ are continuous on $B_0$.
Since $\gamma$ is continuous  with $\gamma(0) = y$, then for some $\eta > 0$,
\begin{equation}  a_{\gamma(s)} < t_1 - \eps, \ b_{\gamma(s)} > t_1 + \eps \qquad \qquad \text{for all} \ s \in I \cap (-\eta, \eta).
\label{eq_B1039} \end{equation}
According to Proposition \ref{prop_B1122}(iii) and the chain rule,
for any $t \in (t_1 - \eps, t_1 + \eps)$,
\begin{equation}   J(y,t) = \sum_{i=1}^{n-1} \lambda_i \frac{\partial F}{\partial y_i} (y,t) = \left. \frac{d}{ds} F(\gamma(s), t) \right|_{s=0}
\label{eq_B1129_} \end{equation}
where we only consider values $s \in I$ when computing the limit defining the derivative with respect to $s$.
Note that the use of the chain rule is legitimate, as $F$ is differentiable at $(y,t)$
while $\gamma(0) = y$ and $\dot{\gamma}(0) = v = (\lambda_1,\ldots,\lambda_{n-1})$.
From (\ref{eq_B1129_}), for any $t \in (t_1 - \eps, t_1 + \eps)$,
\begin{equation}
|J(y,t)| =
\lim_{I \ni s \rightarrow 0} \frac{d(F(\gamma(0), t), F(\gamma(s), t)  )}{|s|}
=
\lim_{I \ni s \rightarrow 0} \frac{d(F(y, t), F(\gamma(s), t) )}{|s|}.
\label{eq_B1747} \end{equation}
Fix $0 < \delta < \eps$. For $s\in (-\eta, \eta) \cap I$ and $i=0,1,2$ define
\begin{equation} x_i = F(y, t_1 + \delta (i-1)), \qquad z_i(s) = F(\gamma(s), t_1 + \delta (i-1)).
\label{eq_B1652__}
\end{equation}
The points $x_0, x_1, x_2, z_0(s), z_1(s), z_2(s) \in \cM$ are well-defined due  to (\ref{eq_B1038}) and (\ref{eq_B1039}).
According to Proposition \ref{prop_B1122}(ii),
\begin{equation} u(x_i) = t_1 + \delta (i-1) + r_0 = u(z_i(s))  \qquad \text{for} \ i=0,1,2,  \  s \in I \cap (-\eta, \eta). \label{eq_B1737}
 \end{equation}
 Recall that $\| u \|_{Lip} \leq 1$ and that $t \mapsto F(y,t)$ is a minimizing geodesic, as well as $t \mapsto F(\gamma(s), t)$. We
 thus conclude from (\ref{eq_B1652__}) and (\ref{eq_B1737}) that for any $s \in I \cap (-\eta, \eta)$ and $i,j=0,1,2$,
 \begin{equation}
  d(x_i, x_j) = d(z_i(s), z_j(s)) = \delta |i-j| = |u(x_i) - u(z_j(s))| \leq d(x_i, z_j(s)).
  \label{eq_B1739}
 \end{equation}
Furthermore, since $d(x_i, x_1) \leq \delta < \eps$ for $i=0,1,2$, then thanks to (\ref{eq_B2033A}) and (\ref{eq_B1038}),
\begin{equation}  x_0,x_1, x_2 \in \cB_{\cM}(x_1, \eps) = \cB_{\cM}(p, \eps) \subseteq \cB_{\cM}(p, \delta_1). \label{eq_B1233A}
\end{equation}
 The map $F$ is continuous, while $\gamma(s) \rightarrow y$ as $I \ni s \rightarrow 0$. Therefore, for
  $i=0,1,2$ we have that $z_i(s) \rightarrow x_i$ as $I \ni s \rightarrow 0$.
  From (\ref{eq_B1233A}) we thus conclude that
 $z_0(s), z_1(s), z_2(s) \in \cB_{\cM}(p, \delta_1)$
 for any $s \in I \cap (-\tilde{\eta}, \tilde{\eta})$ for some $0 < \tilde{\eta} < \eta$. Thanks to (\ref{eq_B1739}) we may apply Lemma
  \ref{lem_849} for the six points $$ x_0, x_1, x_2, z_0(s), z_1(s), z_2(s) \in \cB_{\cM}(p, \delta_1), $$ when $s \in I \cap (-\tilde{\eta}, \tilde{\eta})$. From the conclusion of Lemma \ref{lem_849},
\begin{equation}   \limsup_{I \ni s \rightarrow 0} \frac{d(x_0, z_0(s)) + d(x_2, z_2(s))}{|s|}  \leq 20 \cdot \limsup_{I \ni s \rightarrow 0} \frac{d(x_1, z_1(s))}{|s|}. \label{eq_B1132_} \end{equation}
By using (\ref{eq_B1747}), (\ref{eq_B1652__}) and (\ref{eq_B1132_}) we obtain
\begin{equation} |J(y, t_1 - \delta)| + |J(y, t_1 + \delta)| \leq 20 \cdot |J(y, t_1)|.
\label{eq_B1756} \end{equation}
However, $\delta > 0$ was an arbitrary number in $(0, \eps)$. From (\ref{eq_B1755}) and (\ref{eq_B1756}) we therefore conclude that
$$ |J(y,t)| = 0 \qquad \qquad \qquad \text{for all} \ t \in (t_1 - \eps, t_1 + \eps). $$
This completes the proof that the set $ \left \{ t \in (a_y, b_y) \, ; \, J(y,t) = 0 \right \} $
is an open set. Since $J$ is a smooth Jacobi field, then this set is also closed.
Therefore, either $t \mapsto J(y,t)$ never vanishes on $(a_y, b_y)$, or else it is the zero function.
In other words, for any $\lambda_1,\ldots,\lambda_{n-1} \in \RR$,
$$ \exists t \in (a_y, b_y), \ \sum_{i=1}^{n-1} \lambda_i J_i(y,t) =0
\qquad \Longrightarrow \qquad \forall t \in (a_y, b_y), \ \sum_{i=1}^{n-1} \lambda_i J_i(y,t) =0. $$
By linear algebra, either $J_1(y,t),\ldots,J_{n-1}(y,t)$ are linearly independent for all $t \in (a_y, b_y)$, or
else they are linearly dependent for all $t \in (a_y, b_y)$.
\end{proof}

Recall that $(\cM, d, \mu)$ is an $n$-dimensional weighted Riemannian manifold which is geodesically-convex.
Recall also that $\lambda_\cM$ is the Riemannian volume measure on the Riemannian manifold $\cM$. Let $\rho: \cM \rightarrow \RR$ be the smooth function for which
\begin{equation} \frac{d \mu}{d \lambda_{\cM}} = e^{-\rho}.
\label{eq_B1153}
\end{equation}

\begin{definition} A measure $\nu$ on $\cM$ is called a ``needle candidate'' of the weighted Riemannian manifold $(\cM, d, \mu)$ and the
Lipschitz function
$u$ if there
exist a non-empty subset $(a,b) \subseteq \RR$ with $a,b \in \RR \cup \{ \pm \infty \}$,
a measure $\theta$ on $(a,b)$, a minimizing geodesic $\gamma: (a,b) \rightarrow \cM$
and Jacobi fields $J_1(t),\ldots,J_{n-1}(t)$ along $\gamma$ with the following properties:
\begin{enumerate}
\item[(i)] The measure $\nu$ is the push-forward of $\theta$ under the map $\gamma$.
\item[(ii)] Denote $J_n = \dot{\gamma}$. Then the measure $\theta$ is absolutely-continuous with respect
to the Lebesgue measure in $(a,b) \subseteq \RR$, and its density is
proportional to
\begin{equation}  t \mapsto e^{-\rho(\gamma(t))} \cdot \sqrt{\det \left( \langle J_i(t), J_k(t) \rangle \right)_{i,k=1,\ldots,n}}.
\label{eq_B1327} \end{equation}
\item[(iii)] There exists $t \in (a,b)$ with
\begin{equation}  \langle J_i(t), \dot{\gamma}(t) \rangle = \langle J_i^{\prime}(t), \dot{\gamma}(t)  \rangle = 0 \qquad \qquad (i=1,\ldots,n-1),
\label{eq_B1702_} \end{equation}
and
\begin{equation}  \langle J_i'(t), J_k(t) \rangle = \langle J_k'(t), J_i(t) \rangle \qquad \qquad (i,k=1,\ldots,n-1).
\label{eq_B1041} \end{equation}
\item[(iv)] Either for all $t \in (a, b)$ the vectors
$$ J_1(t),\ldots,J_{n-1}(t) \in T_{\gamma(t)} \cM $$
are linearly independent, or else for all $t \in (a, b)$ these vectors are linearly dependent.
\item[(v)] Denote $A = (a,b) \subseteq \RR$. Then the set $\gamma(A)$ is the relative interior of a transport ray associated with $u$ and
$$ u(\gamma(t)) = t \qquad \qquad \text{for all} \ t \in A. $$
\end{enumerate}
\label{def_B1210}
\end{definition}

Assume that $\Omega_1, \Omega_2,\ldots$ are certain disjoint sets. Let  $\nu_i$ be a measure defined on $\Omega_i$
for $i \geq 1$. We may clearly consider the measure $\nu = \sum_{i \geq 1} \nu_i$ defined on $\Omega = \cup_{i \geq 1} \Omega_i$.
A subset $A \subseteq \Omega$ is $\nu$-measurable if and only if $A \cap \Omega_i$ is $\nu_i$-measurable for any $i \geq 1$.

\medskip Recall that $T^{\circ}[u]$ is a partition of $\Strain[u]$ and that $\pi: \Strain[u] \rightarrow T^{\circ}[u]$ is the partition map,
i.e., $x \in \pi(x) \in T^{\circ}[u]$ for any $x \in \Strain[u]$.
According to Lemma \ref{lem_A1031},  for any $x \in \Strain[u]$, the set $\pi(x)$ is the relative interior of the unique transport ray containing $x$.

\begin{lemma}
Let $u: \cM \rightarrow \RR$ satisfy $\| u \|_{Lip} \leq 1$. Then there exist
a measure $\nu$ on $T^{\circ}[u]$ and a family $\{ \mu_{\cI} \}_{\cI \in T^{\circ}[u]}$
of measures on $\cM$, such that the following hold true:
\begin{enumerate}
\item[(i)] If $G \subseteq T^{\circ}[u]$ is $\nu$-measurable then $\pi^{-1}(G) \subseteq \Strain[u]$
is a measurable subset of $\cM$.
For any measurable set $A \subseteq \cM$,
the map $\cI \mapsto \mu_{\cI}(A)$ is well-defined $\nu$-almost everywhere and is a $\nu$-measurable map.
\item[(ii)] For any measurable set $A \subseteq \cM$,
\begin{equation}  \mu(A \cap \Strain[u]) = \int_{T^{\circ}[u]} \mu_{\cI}(A) d \nu(\cI). \label{eq_B1227}
\end{equation}
\item[(iii)] For $\nu$-almost any $\cI \in T^{\circ}[u]$,
the measure $\mu_{\cI}$ is a needle candidate of $(\cM, d, \mu)$ and $u$ that is supported on $\cI$ and it satisfies $\mu_{\cI}(\cM) > 0$.
Furthermore, $A$ and $\gamma$ from Definition \ref{def_B1210} satisfy $\cI = \gamma(A)$.
\end{enumerate}
\label{prop_1147}
\end{lemma}

\begin{proof} The measure $\mu$ is assumed to be absolutely-continuous with respect to $\lambda_{\cM}$.
According to Proposition \ref{cfm}, there exist disjoint ray clusters of continuous length $\{ R_i \}_{i=1,2,\ldots}$ with
\begin{equation}
 \mu \left( \Strain[u] \setminus \left( \bigcup_{i=1}^{\infty} R_i \right) \right) = 0.
 \label{eq_B1143}
 \end{equation}
Recall from Definition \ref{def_B1215}
and Lemma \ref{cor_2108} that each ray cluster $R_i$ is a measurable set contained in $\Strain[u]$
of the form $R_i = \cup_{\cI \in S_i} \cI$ for some subset $S_i \subseteq T^{\circ}[u]$.
Fix $i \geq 1$. Let us apply Proposition \ref{prop_B1122}
 for $R_i$, which is a ray cluster of continuous length.
  Proposition \ref{prop_B1122} provides us with a certain
parallel line cluster $B \subseteq \RR^{n-1} \times \RR$, a locally-Lipschitz, invertible map $F: B \rightarrow R_i$,
and also with vector  fields $$ J_1(y,t),\ldots,J_{n-1}(y,t). $$
Let $J_n, r_0, B_0, a_y$ and $b_y$ be as in
Proposition \ref{prop_B1122}. Then for almost any Lebesgue density point $y \in B_0$, the vector fields $J_1(y,t),\ldots,J_{n-1}(y,t)$
are well-defined Jacobi fields along the entire geodesic $t \mapsto F(y,t)$ for $t \in (a_y, b_y)$.
 Consider the measure on $B$ whose density with respect to the Lebesgue measure on $B$ is
\begin{equation}
 (y,t) \mapsto \sqrt{ \det \left( \langle J_\ell(y,t), J_k(y,t) \rangle \right)_{\ell,k=1,\ldots,n} }.
 \label{eq_B1149} \end{equation}
According to Proposition \ref{prop_B1122}(iv)
and Remark \ref{rem_1318}, the map $F$ pushes forward the measure whose density is given by
 (\ref{eq_B1149}) to the restriction of $\lambda_{\cM}$ to the ray cluster $R_i$.
Next, consider the measure on $B$  with density
\begin{equation}
 (y,t) \mapsto e^{-
 \rho(F(y,t))} \cdot
 \sqrt{ \det \left( \langle J_\ell(y,t), J_k(y,t) \rangle \right)_{\ell,k=1,\ldots,n} }.
 \label{eq_B1155} \end{equation}
Glancing at (\ref{eq_B1153}), we see that the map $F$ pushes forward the measure whose density is given by (\ref{eq_B1155})
to the restriction of $\mu$ to $R_i$. From Proposition \ref{prop_B1122}(ii), for any $y \in B_0$ there exists
$\cI(y) \in T^{\circ}[u]$ such that
$$ \cI(y) = \left \{ F(y,t) \, ; \, a_y < t < b_y \right \}. $$
Furthermore, $\cI(y) \subseteq R_i$, and since $F$ is invertible then
$\cI(y_1) \cap \cI(y_2) = \emptyset$ for $y_1 \neq y_2$.
By
Proposition \ref{prop_B1122}(ii), for all $y \in B_0$ the map $t \mapsto F(y,t)$ is a minimizing geodesic.
Define the measure
$$ \tilde{\mu}_{\cI(y)} $$
to be the push-forward under the map $t \mapsto F(y,t)$ of the measure on $(a_y, b_y)$ whose density is given
by  (\ref{eq_B1155}). Then $\tilde{\mu}_{\cI(y)}$ is a well-defined measure supported on $\cI(y)$ for almost any $y \in B_0$.
Recall that the  map $F$ pushes forward the measure whose density is given by (\ref{eq_B1155})
to the restriction of $\mu$ to $R_i$. By Fubini's theorem, for any measurable set $A \subseteq R_i$,
\begin{equation} \mu(A) = \int_{B_0} \tilde{\mu}_{\cI(y)}(A) d y = \int_{B_0} \mu_{\cI(y)}(A) e^{-|y|} dy,
\label{eq_B1218} \end{equation}
where $\mu_{\cI(y)} := e^{|y|} \tilde{\mu}_{\cI(y)}$. Denote
\begin{equation}  \tilde{B}_0 = \left \{ y \in B_0 \, ; \, \mu_{\cI(y)}(\cM) > 0 \right \}, \label{eq_B222A} \end{equation}
which is a measurable subset of $B_0 \subseteq \RR^{n-1}$. Define the measure $\nu_i$ to be the push-forward under the map $y \mapsto \cI(y)$ of
the measure on $\tilde{B}_0$ whose density is $y \mapsto e^{-|y|}$.
 Then $\nu_i$ is a finite measure supported on $T^{\circ}[u]$.
 In fact, $\nu_i$ is supported on  $S_i \subseteq T^{\circ}[u]$
 since $\cI(y) \in S_i$ for all $y \in B_0$.
 From (\ref{eq_B1218}) and (\ref{eq_B222A}), for any measurable set $A \subseteq \cM$,
\begin{equation} \mu(A \cap R_i) = \int_{S_i} \mu_{\cI}(A \cap R_i) d \nu_i(\cI) =
\int_{S_i} \mu_{\cI}(A) d \nu_i(\cI).
\label{eq_B1811_} \end{equation}
 Furthermore, $\mu_{\cI}(\cM) > 0$ for $\nu_i$-almost any $\cI \in S_i$, by the definition of $\tilde{B}_0$.
 Recall that when we push-forward a measure, we also push-forward its $\sigma$-algebra.
 Therefore if a subset $G \subseteq S_i$ is $\nu_i$-measurable, then $\{ y \in \tilde{B}_0 \, ; \,
\cI(y) \in G \}$ is a measurable subset of $B_0$. Since $B$ is a parallel line cluster, then also
$\{ (y,t) \in B \, ; \, \cI(y) \in G \}$ is measurable in $\RR^{n-1} \times \RR$. The image of the latter
measurable set under $F$ equals $\pi^{-1}(G)$. Since $F$
is locally-Lipschitz, then $\pi^{-1}(G)$ is a measurable subset of $\Strain[u]$, whenever $G \subseteq S_i$
is $\nu_i$-measurable.

\medskip
Let us show that $\mu_{\cI}$ is a needle-candidate for $\nu_i$-almost any $\cI \in S_i$.
Since $\mu_{\cI}$ is proportional to $\tilde{\mu}_{\cI}$, it suffices to prove
that $\tilde{\mu}_{\cI(y)}$ is a needle-candidate for almost any $y \in \tilde{B}_0$.
Properties (i) and (ii)
from Definition \ref{def_B1210} hold by the definition of $\tilde{\mu}_{\cI(y)}$,
where we set
$$ J_i(t) = J_i(y,t-r_0), \quad \gamma(t) = F(y, t-r_0), \quad a = a_y + r_0, \quad b = b_y + r_0. $$
Property (v) follows from Proposition \ref{prop_B1122}(ii).
We deduce property (iii) of Definition \ref{def_B1210} (with $t = r_0$)
from Proposition \ref{prop_B1122}(iii).
Property (iv) follows from Lemma \ref{lem_B1811}. Note also that setting $A = (a,b)$ we have
\begin{equation} \cI(y) = \gamma(A). \label{eq_B1218A} \end{equation}
 Hence
$\tilde{\mu}_{\cI(y)}$ is a needle-candidate supported on $\cI(y)$ for almost any $y \in \tilde{B}_0$, and consequently
$\mu_{\cI}$ is a needle-candidate supported on $\cI$ for $\nu_i$-almost any $\cI \in S_i$.
Write $\tilde{S}_i \subseteq S_i$ for the collection of all $\cI \in S_i$ for which
$\mu_{\cI}$ is a needle-candidate supported on $\cI$ with $\mu_{\cI}(\cM) > 0$. Then $\nu_i(S_i \setminus \tilde{S}_i) = 0$.
For completeness, let us redefine $\mu_{\cI} \equiv 0$
for $\cI \in S_i \setminus \tilde{S}_i$. Note that
(\ref{eq_B1811_}) still holds true for any measurable set $A \subseteq \cM$, since
we altered the definition of $\mu_{\cI}$ only on a $\nu_i$-null set.

\medskip
To summarize, we found a family of measures $\{ \mu_\cI \}_{\cI \in S_i}$
such that (\ref{eq_B1811_}) holds true for any measurable set $A \subseteq \cM$.
We now let $i$ vary. Since the ray clusters $\{ R_i \}_{i=1,2,\ldots}$ are disjoint, then
 $S_1,S_2,\ldots \subseteq T^{\circ}[u]$ are also disjoint.
Denoting $\nu = \sum_i \nu_i$, we deduce (\ref{eq_B1227}) from   (\ref{eq_B1143}) and (\ref{eq_B1811_}).
This completes the proof of (ii), and also of the second assertion in (i).
Furthermore, for $\nu$-almost any $\cI \in T^{\circ}[u]$, we have that $\cI \in S_i$
for some $i$, and the measure $\mu_{\cI}$ is a needle-candidate supported on $\cI$ with $\mu_{\cI}(\cM) > 0$.
It thus follows from (\ref{eq_B1218A}) that  conclusion (iii) holds true.
Note that if a subset $G \subseteq T^{\circ}[u]$ is $\nu$-measurable, then
$G \cap S_i$ is $\nu_i$-measurable for any $i$, and hence $\pi^{-1}(G \cap S_i) \subseteq R_i$ is measurable in $\cM$.
Consequently $\pi^{-1}(G)$ is $\lambda_{\cM}$-measurable whenever $G \subseteq T^{\circ}[u]$ is $\nu$-measurable.
This completes the  proof of (i). The lemma is therefore proven.
\end{proof}

Recall from Section \ref{sec_intro}  the definition of the generalized Ricci tensor $\Ric_{\mu, N}$ of the weighted
Riemannian manifold $(\cM, d, \mu)$.

\begin{definition} Let $n \geq 2, N \in (-\infty, 1) \cup [n, +\infty]$ and let
$(\cM,d,\mu)$ be an $n$-dimensional weighted Riemannian manifold.
We say that a measure $\nu$ on
the Riemannian manifold $\cM$ is an  ``$N$-curvature needle'' if there exist a non-empty, connected open set $A \subseteq \RR$,
a smooth function $\Psi: A \rightarrow \RR$ and a minimizing geodesic $\gamma:A \rightarrow \cM$ such that:
\begin{enumerate}
\item[(i)] Denote by $\theta$ the measure on $A \subseteq \RR$ whose density with respect to the Lebesgue measure
is $e^{-\Psi}$. Then $\nu$ is the push-forward of $\theta$ under the map $\gamma$.
\item[(ii)] The following inequality holds in the entire set $A$:
\begin{equation}
\Psi^{\prime \prime} \geq \Ric_{\mu, N}(\dot{\gamma}, \dot{\gamma}) + \frac{(\Psi^{\prime})^2}{N-1},
\label{eq_B327}
\end{equation}
where in the case $N = \infty$, we interpret the term $(\Psi^{\prime})^2 / (N-1)$ as zero.\end{enumerate}
\label{def_curvature}
\end{definition}

The following proposition asserts that
any needle-candidate in the sense of Definition
\ref{def_B1210} is in fact an {\it $N$-curvature needle}.

\begin{proposition}
Let $n \geq 2, N \in (-\infty, 1) \cup [n, +\infty]$ and let $(\cM,d,\mu)$ be an $n$-dimensional weighted Riemannian manifold
which is geodesically-convex. Let $u: \cM \rightarrow \RR$ satisfy $\|  u \|_{Lip} \leq 1$.
Let $\nu$ be a needle-candidate of $(\cM, d, \mu)$ and $u$. Then either $\nu$ is the zero measure, or else $\nu$ is an $N$-curvature needle.
\label{prop_1134}
\end{proposition}

The proof of Proposition \ref{prop_1134} essentially boils down to a classical estimate in Riemannian geometry
from Heintze and Karcher \cite{HK} that was generalized to the case of {\it weighted} Riemannian
manifolds by Bayle \cite[Appendix E.1]{bayle} and by Morgan \cite{morgan}.
According to Gromov  \cite{G_levy}, the estimate stems from the work of Paul Levy on the isoperimetric inequality
in 1919.
We begin the proof of Proposition \ref{prop_1134} with the following trivial lemma:

\begin{lemma} Let $a,b \in \RR$ with $b > 0$ and $a \not \in [-b, 0]$. Then,
$$ \frac{x^2}{a} + \frac{y^2}{b} \geq \frac{(x- y)^2}{a+b} \qquad \qquad (x,y \in \RR).
$$ \label{lem_B521} \end{lemma}

\begin{proof} We use the inequality $|b/a| \cdot x^2 \pm 2 xy + |a/b| \cdot y^2 \geq 0$ to deduce that
$$
\frac{x^2}{a} + \frac{y^2}{b} - \frac{(x- y)^2}{a+b} = \frac{1}{a+b} \left( \frac{b}{a} x^2 + 2 xy + \frac{a}{b} y^2 \right) \geq 0, $$
whenever $b > 0$ and $a \not \in [-b, 0]$.
\end{proof}

Let us recall the familiar formulas for differentiating a determinant. If $A_t$ is an invertible  $n \times n$ matrix
that depends smoothly on $t \in \RR$, then
\begin{equation}
 \frac{d}{dt} \log \left| \det(A_t)\right| = \Trace[ A_t^{-1} \cdot \dot{A}_t ],
 \label{eq_B956}
 \end{equation}
 and
\begin{equation}
\frac{d^2}{dt^2} \log \left| \det(A_t) \right|= \Trace[ A_t^{-1} \cdot \ddot{A_t} ] - \Trace \left[ \left( A_t^{-1} \cdot \dot{A}_t \right)^2 \right].
\label{eq_B957} \end{equation}

\begin{proof}[Proof of Proposition \ref{prop_1134}] Let $\nu$ be a needle-candidate of $(\cM, d, \mu)$ and $u$.
We may assume that $\nu$ is not the zero measure.
Let $a,b,\theta,\gamma$
and $J_1,\ldots,J_{n-1}$ be as in
Definition \ref{def_B1210}. For $t \in (a,b)$ denote
\begin{equation}  f(t)  = e^{-\rho(\gamma(t))} \cdot \sqrt{\det \left( \langle J_i(t), J_k(t) \rangle \right)_{i,k=1,\ldots,n}}
\label{eq_B941} \end{equation}
where $J_n = \dot{\gamma}$. According to Definition \ref{def_B1210}(ii), the
density of the measure $\theta$ on $(a,b) \subseteq \RR$
is proportional to the function $f$.
We will prove that $f$ is smooth and positive
in $(a,b)$, and that $\Psi := -\log f$ satisfies
\begin{equation}
\Psi^{\prime \prime} \geq \Ric_{\mu, N}(\dot{\gamma}, \dot{\gamma}) + \frac{(\Psi^{\prime})^2}{N-1}, \label{eq_B943}
\end{equation}
where in the case $N = + \infty$ we interpret the term $(\Psi^{\prime})^2 / (N-1)$ as zero.
Comparing Definition
\ref{def_curvature} of $N$-curvature needles and Definition \ref{def_B1210} of needle-candidates, we
see that the proposition would follow from (\ref{eq_B943}). The rest of the proof is therefore
devoted to establishing (\ref{eq_B943}).
The Jacobi fields $J_1,\ldots,J_{n-1}$ satisfy the Jacobi equation:
\begin{equation} J_i^{\prime \prime}(t) = R(\dot{\gamma}(t), J_i(t)) \dot{\gamma}(t) \qquad \qquad \qquad \text{for} \ t \in (a,b), i=1,\ldots,n-1.
\label{eq_B1308}
\end{equation}
Since $\gamma$ is a geodesic then  $\nabla_{\dot{\gamma}} \dot{\gamma} = 0$, and for any $i=1,\ldots,n-1$ and $t \in (a,b)$,
\begin{equation}   \frac{d}{dt} \langle J_i, \dot{\gamma} \rangle = \langle J_i^{\prime}, \dot{\gamma} \rangle,
\qquad \frac{d^2}{dt^2} \langle J_i, \dot{\gamma} \rangle = \langle J_i^{\prime \prime}, \dot{\gamma} \rangle.
\label{eq_B1044} \end{equation}
From (\ref{eq_B1308}) and the symmetries of the Riemann curvature tensor
we deduce that $\langle J_i^{\prime \prime}, \dot{\gamma} \rangle \equiv 0$.
Therefore $\langle J_i(t), \dot{\gamma}(t) \rangle$ is an affine function of $t \in (a,b)$.
It thus follows from
(\ref{eq_B1702_}) and (\ref{eq_B1044}) that for any $t \in (a,b)$,
\begin{equation}  J_1(t),\ldots,J_{n-1}(t) \perp \dot{\gamma}(t). \label{eq_B1710} \end{equation}
From (\ref{eq_B941}) and (\ref{eq_B1710}) we obtain
\begin{equation}  f(t) = e^{-\rho(\gamma(t))} \cdot \sqrt{\det \left( \langle J_i(t), J_k(t) \rangle \right)_{i,k=1,\ldots,n-1}}.
\label{eq_B1717} \end{equation}
(The indices run only up to $n-1$, as $\dot{\gamma} = J_n$ is a unit vector orthogonal to $J_1,\ldots,J_{n-1}$).
Since $\theta$ is not the zero measure, there exists $t_1 \in (a,b)$ for which $f(t_1) \neq 0$. From (\ref{eq_B1717}) we learn
that the vectors
$$ J_1(t_1),\ldots,J_{n-1}(t_1) \in T_{\gamma(t)} \cM $$
are linearly independent. According to  Definition \ref{def_B1210}(iv), the vectors
$J_1(t),\ldots,J_{n-1}(t)$ are linearly independent for all $t \in (a,b)$.
Hence, (\ref{eq_B1717}) yields
\begin{equation} \forall t \in (a,b), \ \ f(t) > 0. \label{eq_B1813} \end{equation}
From the Jacobi equation (\ref{eq_B1308}), for any $t \in (a,b)$ and $i,k =1,\ldots,n-1$,
\begin{equation}
\frac{d}{dt} \left( \langle J_i^{\prime}, J_k \rangle -  \langle J_i, J_k^{\prime} \rangle \right)
= \langle J_i^{\prime \prime}, J_k \rangle - \langle J_i, J_k^{\prime \prime} \rangle =
\langle R(\dot{\gamma}, J_i) \dot{\gamma}, J_k \rangle - \langle J_i, R(\dot{\gamma}, J_k) \dot{\gamma} \rangle = 0,
\label{eq_B1316}
\end{equation}
by the symmetries of the Riemann curvature tensor.
By using (\ref{eq_B1041}) and (\ref{eq_B1316}) we deduce that in the entire interval $(a,b) \subseteq \RR$,
\begin{equation}  \langle J_i^{\prime}, J_k \rangle = \langle J_i, J_k^{\prime} \rangle \qquad \qquad \text{for} \ i,k=1,\ldots,n. \label{eq_B1713}
\end{equation}
Let $G_t = (G_t(i,k))_{i,k=1,\ldots,n-1}$ be the symmetric, positive-definite $(n-1) \times (n-1)$ matrix
whose entries are $G_t(i,k) = \langle J_i(t), J_k(t) \rangle$.
According to (\ref{eq_B1717}) and (\ref{eq_B1813}), the function $\Psi = -\log f$ satisfies,
\begin{equation}
\Psi(t)  = \rho(\gamma(t)) - \frac{1}{2} \log \det G_t \qquad \qquad \qquad \text{for} \ t \in (a,b).
\label{eq_B1042}
\end{equation}
Denote $H(t) = \dot{\gamma}(t)^{\perp} \subset T_{\dot{\gamma}(t)} \cM$, the orthogonal complement to the vector $\dot{\gamma}(t)$.
From (\ref{eq_B1044}) and (\ref{eq_B1710}),
\begin{equation}  J_i(t), J_i^{\prime}(t) \in H(t) \qquad \qquad \qquad \text{for all} \ t \in (a,b), i=1,\ldots,n-1. \label{eq_B2155_}
\end{equation}
For any $t \in (a,b)$ the linearly-independent vectors $J_1(t),\ldots,J_{n-1}(t) \in H(t)$ constitute a basis of the $(n-1)$-dimensional space $H(t)$.
In view of (\ref{eq_B2155_}), we may define
an $(n-1) \times (n-1)$ matrix
$A_t = (A_t(i,k))_{i,k=1,\ldots,n-1} $ by requiring that
\begin{equation}
J_i^{\prime}(t) = \sum_{k=1}^{n-1} A_t(i, k) J_k(t) \qquad \qquad \qquad \text{for} \ t \in (a,b), i=1,\ldots,n-1.
\label{eq_B1025}
\end{equation}
Recall that $G_t(i,k) = \langle J_i(t), J_k(t) \rangle$.
From (\ref{eq_B1713}) and (\ref{eq_B1025}),  for any $t \in (a,b)$,
$$
 \dot{G}_t(i,k) = \langle J_i^{\prime}, J_k \rangle + \langle J_i, J_k^{\prime} \rangle = 2 \langle
  J_i^{\prime}, J_k \rangle = 2 \left \langle \sum_{\ell=1}^{n-1} A_t(i, \ell) J_{\ell}, J_k \right \rangle =
  2 \sum_{\ell=1}^{n-1} A_t(i, \ell) G_t(\ell, k).  $$
  Equivalently, $\dot{G}_t = 2 A_t G_t$. Since $G_t$ is a symmetric matrix
  then also $A_t G_t = \dot{G}_t / 2$ is a symmetric matrix.
Since $G_t$ is a positive-definite matrix, from (\ref{eq_B956}), then
 \begin{equation}
\frac{d}{dt} \log \det(G_t) = \Trace \left[  G_t^{-1} \dot{G}_t \right]
= 2 \Trace \left[ G_t^{-1} A_t G_t \right] = 2 \Trace[ A_t ].
\label{eq_B1906}
\end{equation}
As for the second derivative, we use (\ref{eq_B1025}) and the Jacobi equation (\ref{eq_B1308}) and obtain,
\begin{align}
\label{eq_B1757_} \ddot{G}_t(i,k)  & = \langle J_i^{\prime \prime}, J_k \rangle + 2 \langle J_i^{\prime}, J_k^{\prime} \rangle
+ \langle J_k^{\prime \prime}, J_i \rangle \\ & = 2 \langle R(\dot{\gamma}, J_i) \dot{\gamma}, J_k \rangle + 2 \sum_{\ell, m=1}^{n-1} A_t(i, \ell) A_t(k, m) G_t(\ell, m) \nonumber \end{align}
where we used the symmetries of the Riemann curvature tensor in the last passage.
Recall that $\Ric_{\cM}(\dot{\gamma}, \dot{\gamma})$ is the trace of the linear transformation $V \mapsto -R(\dot{\gamma}, V) \dot{\gamma}$
in the linear space $H(t)$. By linear algebra, (\ref{eq_B1757_}) entails that
\begin{equation} \Trace \left[ G_t^{-1} \ddot{G}_t \right] = -2 \Ric_{\cM}(\dot{\gamma}(t), \dot{\gamma}(t))
+ \Trace \left[ 2 G_t^{-1} A_t^2 G_t \right],
\label{eq_B2157_} \end{equation}
where we used the fact that $A_t G_t A_t^* = A_t (A_t G_t)^* = A_t^2 G_t$ in the last passage, as $A_t G_t$ is symmetric.
Since $\dot{G}_t = 2 A_t G_t$ then from (\ref{eq_B957}) and (\ref{eq_B2157_}),
\begin{equation}   \frac{d^2}{dt^2} \log \det(G_t) = -2 \Ric_{\cM}(\dot{\gamma}(t), \dot{\gamma}(t))
+ 2 \Trace \left[ A_t^2  \right] - 4 \Trace \left[ G_t^{-1} A_t^2 G_t \right].
\label{eq_B1059} \end{equation}
Applying (\ref{eq_B1042}) and (\ref{eq_B1906}) yields
\begin{equation}  \Psi^{\prime}(t) = \partial_{\dot{\gamma}(t)} \rho - \Trace[ A_t ]. \label{eq_B1150} \end{equation}
Since $\gamma$ is a geodesic, the equations (\ref{eq_B1042}) and (\ref{eq_B1059}) lead to
\begin{align} \Psi^{\prime \prime}(t) & = \Hess_{\rho}(\dot{\gamma}(t), \dot{\gamma}(t)) + \Ric_{\cM}(\dot{\gamma}(t), \dot{\gamma}(t)) +
\Trace \left[ A_t^2  \right]. \label{eq_B1358} \end{align}
We will now utilize the definition of the generalized Ricci tensor with parameter $N$. Therefore,
from (\ref{eq_B1358}),
\begin{align} \Psi^{\prime \prime}(t) \geq  \Ric_{\mu, N}(\dot{\gamma}(t), \dot{\gamma}(t))
+ \frac{(\partial_{\dot{\gamma}(t)} \rho)^2}{N-n} + \Trace \left[ A_t^2  \right], \label{eq_B1411} \end{align}
where in the case where $N = \infty$
we interpret the term $(\partial_{\dot{\gamma}(t)} \rho)^2 / (N-n)$ as zero. In the case where $N = n$, we
require $\rho$ to be a constant function and the latter term is again interpreted as zero.
The matrix $\dot{G}_t = 2 A_t G_t$ is symmetric, and hence $G_t^{-1/2} A_t G_t^{1/2}$ is also symmetric.
Thus the $(n-1) \times (n-1)$ matrix $A_t$ is conjugate to a symmetric matrix and consequently it has $n-1$ real eigenvalues (repeated according
to their multiplicity).
The Cauchy-Schwartz inequality yields $\left[ \Trace(A_t) \right]^2 \leq (n-1) \Trace[A_t^2]$ and therefore,
for any $t \in (a,b)$,
\begin{equation} \Psi^{\prime \prime}(t) \geq  \Ric_{\mu, N}(\dot{\gamma}(t), \dot{\gamma}(t))
+ \frac{(\partial_{\dot{\gamma}(t)} \rho)^2}{N-n} + \frac{\left( \Trace  [A_t]  \right)^2}{n-1}.
\label{eq_B1148}
\end{equation}
In the case where $N = \infty$ or $N = n$, we deduce (\ref{eq_B943})  from (\ref{eq_B1150}) and (\ref{eq_B1148}).
Otherwise, we have $N \in \RR \setminus [1,n]$ and from (\ref{eq_B1148}) and Lemma \ref{lem_B521},
\begin{equation} \Psi^{\prime \prime}(t) \geq  \Ric_{\mu, N}(\dot{\gamma}(t), \dot{\gamma}(t))
+ \frac{(\partial_{\dot{\gamma}(t)} \rho - \Trace[ A_t ])^2}{N-1}.
\label{eq_B1418}
\end{equation}
From (\ref{eq_B1150}) and (\ref{eq_B1418}) we
conclude that
(\ref{eq_B943}) holds true for any $t \in (a,b)$, and the proof of the proposition is complete.
\end{proof}

\begin{example}{\rm Consider the example where $\rho \equiv Const$ and where $\cM \subseteq \RR^n$ is an open, convex set.
Equations (\ref{eq_B1150}) and (\ref{eq_B1358}) along with simple manipulations show that here,
\begin{equation}  \Psi^{\prime}(t) = -\Trace[A_t], \qquad \Psi''(t) = \Trace[A_t^2] \qquad \text{and} \qquad \dot{A}_t = -A_t^2.
\label{eq_B1741_} \end{equation}
The eigenvalues of $A_t$ may be viewed as ``principal curvatures'' or as ``eigenvalues of the second fundamental form'' of a level set of $u$.
Solving (\ref{eq_B1741_}), we see that the density $f(t) = e^{-\Psi(t)}$ is proportional to the function
\begin{equation} t \mapsto \prod_{i=1}^{k} |t - \lambda_i| \qquad \qquad \qquad \text{for} \ t \in (a,b), \label{eq_B2045A} \end{equation}
where $k \leq n-1$ and $\lambda_1,\ldots,\lambda_{k} \in \RR \setminus (a,b)$ are some numbers. An empty product is defined to be one.
We
 learn from (\ref{eq_B2045A}) that the
positive function $f: (a,b) \rightarrow \RR$
is a polynomial of degree at most $n-1$, all of whose roots lie in $\RR \setminus (a,b)$.
}\end{example}

\begin{theorem}
Let $n \geq 2$ and $N \in (-\infty, 1) \cup [n, +\infty ]$.
Assume that $(\cM,d,\mu)$ is an $n$-dimensional weighted Riemannian manifold
which is geodesically-convex.
 Let $u: \cM \rightarrow \RR$ satisfy $\| u \|_{Lip} \leq 1$.
Then there exist a measure $\nu$ on the set $T^{\circ}[u]$
and a family $\{ \mu_{\cI} \}_{\cI \in T^{\circ}[u]}$ of measures on $\cM$
such that:
\begin{enumerate}
\item[(i)] For any Lebesgue-measurable set $A \subseteq \cM$, the map $\cI \mapsto \mu_{\cI}(A)$ is well-defined $\nu$-almost everywhere and is a $\nu$-measurable map.
When a subset $S \subseteq T^{\circ}[u]$ is $\nu$-measurable then $\pi^{-1}(S) \subseteq \Strain[u]$
is a measurable subset of $\cM$.
\item[(ii)] For any Lebesgue-measurable set $A \subseteq \cM$,
$$ \mu(A \cap \Strain[u]) = \int_{T^{\circ}[u]} \mu_{\cI}(A) d \nu(\cI). $$
\item[(iii)] For $\nu$-almost any $\cI \in T^{\circ}[u]$, the measure $\mu_{\cI}$ is an $N$-curvature needle supported on $\cI \subseteq \cM$.
Furthermore, the set $A \subseteq \RR$ and the minimizing geodesic $\gamma:A \rightarrow \cM$ from Definition \ref{def_curvature} may be
selected so that $\cI = \gamma(A)$ and
so that
$$ u(\gamma(t)) = t \qquad \qquad \qquad \text{for all} \ t \in A. $$
\end{enumerate}
\label{thm_main2_}
\end{theorem}

\begin{proof}
Apply
Lemma \ref{prop_1147} to obtain certain measures $\nu$ and $\{ \mu_{\cI} \}_{\cI \in T^{\circ}[u]}$.
Applying Lemma \ref{prop_1147}(iii) and
Proposition \ref{prop_1134}, we learn that $\mu_{\cI}$ is an $N$-curvature needle supported on $\cI$ for $\nu$-almost any $\cI \in T^{\circ}[u]$.
Together with Definition \ref{def_B1210}(v), this proves conclusion (iii). Conclusions (i) and (ii) follow from Lemma \ref{prop_1147}(i) and
Lemma \ref{prop_1147}(ii), respectively.
\end{proof}

\medskip
\begin{proof}[Proof of Theorem \ref{thm_main2}]
Recall from Section \ref{sec_intro} that the weighted Riemannian manifold $(\cM, d, \mu)$ satisfies the curvature-dimension  condition $CD(\kappa,N)$  when
$$ \Ric_{\mu, N}(v, v) \geq \kappa \qquad \qquad \qquad \text{for any} \ p \in \cM, v \in T_p \cM, |v| = 1. $$
Glancing at Definition \ref{def_cd} and Definition \ref{def_curvature}, we see that
under curvature-dimension condition $CD(\kappa,N)$, any $N$-curvature needle is in fact a $CD(\kappa, N)$-needle.
The theorem thus follows from Theorem \ref{thm_main2_}.
\end{proof}

\section{The Monge-Kantorovich problem}
\label{sec_monge}
\setcounter{equation}{0}

In this section we prove Theorem \ref{prop_intro}, following the approach of Evans and Gangbo \cite{EG}.
We assume that $(\cM, d, \mu)$ is an $n$-dimensional, geodesically-convex, weighted Riemannian manifold
of class $CD(\kappa, N)$, where $n \geq 2, \kappa \in \RR$ and $N \in (-\infty, 1) \cup [n, + \infty]$.
Suppose that $f: \cM \rightarrow \RR$ is a $\mu$-integrable function with
\begin{equation}
\int_\cM f d \mu = 0. \label{eq_C1801}
\end{equation}
Assume also that there exists a point $x_0 \in \cM$ with
\begin{equation}  \int_\cM |f(x)| \cdot  d(x_0, x) d \mu(x) < \infty. \label{eq_C1802}
\end{equation}
It follows from (\ref{eq_C1802}) that for any $1$-Lipschitz function $v: \cM \rightarrow \RR$,
 $$ \int_{\cM} |f v| d \mu \leq |v(x_0)| \int_{\cM} |f| d \mu + \int_{\cM} |f(x)| d(x_0, x) d \mu(x) < \infty, $$
as $|v(x)| \leq |v(x_0)| + d(x_0, x)$ for all $x \in \cM$.
Conclusion (A) of Theorem \ref{prop_intro} follows from the following standard lemma:

\begin{lemma} There exists a $1$-Lipschitz function $u: \cM \rightarrow \RR$ with
\begin{equation}  \int_{\cM} u f d \mu = \sup \left \{ \int_{\cM} v f d \mu \, ; \, v: \cM \rightarrow \RR, \, \| v \|_{Lip} \leq 1 \right \}.
\label{eq_B1803} \end{equation}
\label{lem_C1815}
\end{lemma}

\begin{proof} Recall that $(\cM, d)$ is a locally-compact, separable, metric space (see, e.g., Section \ref{transport_rays}).
For $k=1,2,\ldots$ let $v_k: \cM \rightarrow \RR$ be a $1$-Lipschitz function such that
$$ \int_{\cM} v_k f d \mu \stackrel{k \rightarrow \infty} \longrightarrow \sup_{\| v \|_{Lip} \leq 1} \int_{\cM} v f d \mu. $$
Since $\int_{\cM} f d \mu = 0$, then we may add a constant to $v_k$ and assume that $v_k(x_0) = 0$ for all $k$.
By the Arzela-Ascoli theorem, there exists a subsequence $v_{k_i}$
that converges locally-uniformly to a $1$-Lipschitz function $u: \cM \rightarrow \RR$ with $u(x_0) = 0$.
Since $|v_k(x)| \leq d(x_0, x)$ for all $x \in \cM$ and $k \geq 1$, then we may apply the dominated convergence theorem
thanks to (\ref{eq_C1802}). We conclude that
\begin{align*} \int_{\cM} u f d \mu = \lim_{i \rightarrow \infty} \int_{\cM} v_{k_i} f d \mu = \sup_{\| v \|_{Lip} \leq 1} \int_{\cM} v f d \mu.
\tag*{\qedhere} \end{align*}
\end{proof}

The maximization problem in Lemma \ref{lem_C1815} is dual to the
$L^1$-Monge-Kantorovich problem in the theory of optimal transportation. For information about the Monge-Kantorovich $L^1$-transportation problem,
we refer the reader
to  the book by Kantorovich and Akilov \cite[Section VIII.4]{KA}
and to the papers by Ambrosio \cite{amb}, Evans and Gangbo \cite{EG} and Gangbo \cite{gangbo}.

\medskip Most of the remainder of this section is devoted to the proof of conclusions (B) and (C) of Theorem \ref{prop_intro}.
To that end, let us fix a $1$-Lipschitz
function $u: \cM \rightarrow \RR$ such that
\begin{equation}  \int_{\cM} u f d \mu = \sup_{\| v \|_{Lip} \leq 1} \int_{\cM} v f d \mu.
\label{eq_C1535} \end{equation}

Recall the definition of a transport ray from Section \ref{transport_rays}. The set $T[u]$ is the collection
of all transport rays associated with $u$. From the definition of a transport ray, for any $x, y \in \cM$,
\begin{equation}  |u(x) - u(y)| = d(x,y) \qquad \Longleftrightarrow \qquad \exists \cI \in T[u], \, x,y \in \cI. \label{eq_C2146A}
\end{equation}
A transport ray is called degenerate when it is  a singleton. By the maximality property of transport rays
(see Definition \ref{def_1050}),
for any $x \in \cM$,
\begin{equation}
\{ x \} \in T[u] \qquad \Longleftrightarrow \qquad \forall x \neq y \in \cM, \ |u(y) - u(x)| < d(x,y).
\label{eq_C950} \end{equation}
Define $\Inactive[u] \subseteq \cM$ to be the union of all {\it degenerate} transport rays associated with $u$.
Thus,
$$ \Inactive[u] = \left \{ x \in \cM \, ; \, \{ x \} \in T[u] \right \}. $$
By the maximality property of transport rays, for any $\cI \in T[u]$,
\begin{equation}
\cI \cap \Inactive[u] \neq \emptyset \qquad \Longleftrightarrow \qquad \exists x \in \Inactive[u], \ \cI = \{ x \}.
\label{eq_C1354}
\end{equation}
From Lemma \ref{lem_1102}, any transport ray $\cI \in T[u]$ is the image  of a minimizing geodesic.
The relative interior of $\cI \in T[u]$ is empty if and only if $\cI$ is a singleton.
Recall from Lemma \ref{lem_A317}
that $T^{\circ}[u]$ is the collection of all relative interiors of non-degenerate transport rays associated with $u$,
while
\begin{equation}  \Strain[u] = \bigcup_{\cI \in T^{\circ}[u]} \cI. \label{eq_C1736} \end{equation}
It follows from  (\ref{eq_C1354}) and (\ref{eq_C1736}) that
\begin{equation}
\Strain[u] \cap \Inactive[u] = \emptyset.
\label{eq_C954} \end{equation}
Finally, let us set $\Ends[u] = \cM \setminus \left( \Inactive[u] \cup \Strain[u] \right)$.
Thus, $\Strain[u], \Ends[u]$ and $\Inactive[u]$ are three disjoint sets whose union equals $\cM$.

\begin{lemma} $\displaystyle  \mu \left( \Ends[u]   \right) = \lambda_{\cM} \left( \Ends[u]   \right) = 0$.
\label{lem_C1047}
\end{lemma}

\begin{proof}
Recall from Section \ref{lip_sec}
that for a subset $A \subseteq \cM$, we define $\Ends(A) \subseteq \cM$ to be the union
of all relative {\it boundaries} of transport rays intersecting $A$. We claim that
\begin{equation}
\Ends[u] \subseteq \Ends\left( \Strain[u] \right).
\label{eq_C1005}
\end{equation}
Indeed, if $x \in \Ends[u]$, then $\{ x \}$ is not a transport ray as $x \not \in \Inactive[u]$.
From Definition \ref{def_1050}, there exists a non-degenerate transport ray $\cI \in T[u]$
that contains $x$. Since $x \not \in \Strain[u]$, then the point $x \in \cI$ does not
belong to the relative interior of $\cI$. Consequently, $x$ belongs to the relative
boundary of $\cI$. Since the relative interior of $\cI$ is non-empty, then $\cI \cap \Strain[u] \neq \emptyset$
and consequently $x \in \Ends\left( \Strain[u] \right)$. Thus (\ref{eq_C1005}) is proven.
Next, according to Lemma
\ref{lem_B952}, there exist ray clusters $R_1,R_2,\ldots$ such that $\Strain[u] = \cup_{i} R_i$. Hence,
\begin{equation}
\Ends(\Strain[u]) = \bigcup_{i=1}^{\infty} \Ends(R_i).
\label{eq_C1053_}
\end{equation}
However, Lemma \ref{lem_1054} asserts that $\lambda_{\cM}(\Ends(R_i)) = 0$ for any $i \geq 1$.
Consequently, from (\ref{eq_C1005}) and (\ref{eq_C1053_}) we conclude that
$$ \lambda_{\cM} \left( \Ends[u] \right) = 0. $$
Since $\mu$ is absolutely-continuous with respect to $\lambda_{\cM}$, the lemma is proven.
\end{proof}

The following lemma, just like  our entire proof of conclusion (B), is similar  to the mass balance lemma of Evans and Gangbo \cite[Lemma 5.1]{EG}.
For a set $K$ we write $1_K$ for the function that equals one on $K$ and vanishes elsewhere.

\begin{lemma} Let $K \subseteq \cM$ be a compact set. For $\delta > 0$ denote
\begin{equation}  u_{\delta}(x) = \inf_{y \in \cM} \left[ u(y) + d(x,y) - \delta \cdot 1_K(y) \right]
\qquad \qquad \text{for} \ x \in \cM.
 \label{eq_C2132} \end{equation}
 Let $A \subseteq \cM$ be the union of all transport
rays $\cI \in T[u]$ that intersect $K$.
 Then there exists a function $v: \cM \rightarrow [0,1]$ such that
 \begin{equation}  \lim_{\delta \rightarrow 0^+} \frac{u(x) - u_{\delta}(x)}{\delta} = \left \{ \begin{array}{cl} 0 & x \in \cM \setminus A \\ v(x) & x \in A \setminus K \\ 1 & x \in K \end{array} \right. \label{eq_C2127}
\end{equation}
 Moreover, for any $x \in \cM$ and $\delta > 0$ we have that $0 \leq u(x) - u_{\delta}(x) \leq \delta$.
\label{lem_1532}
\end{lemma}

\begin{proof} Since $\| u \|_{Lip} \leq 1$ then for all $x \in \cM$,
\begin{equation}
u_{\delta}(x) = \inf_{y \in \cM} \left[ u(y) + d(x,y) - \delta \cdot 1_K(y) \right] \geq
\inf_{y \in \cM} \left[ u(y) + d(x,y) \right] - \delta \geq u(x) - \delta. \label{eq_C1521} \end{equation}
The ``Moreover'' part of the lemma follows from (\ref{eq_C1521}) and from the simple inequality  $u_{\delta}(x) \leq u(x)$.
For any $x,y \in \cM$ we have that $u(x) - u(y) - d(x,y) \leq 0$ as $u$ is $1$-Lipschitz. Therefore, for any $x \in \cM$, the function
$$ \delta \mapsto \frac{u(x) - u_{\delta}(x)}{\delta} = \sup_{y \in \cM} \left[ \frac{u(x) - u(y) - d(x,y)}{\delta} +  1_K(y) \right] $$
is non-decreasing in $\delta > 0$. Hence the limit in (\ref{eq_C2127}) exists and belongs to $[0,1]$ for all $x \in \cM$.
Next, fix a point $x \in \cM \setminus A$.
 Then for any $y \in K$, the points $x$ and $y$ do not belong to the same transport ray.
Therefore $|u(x) - u(y)| < d(x,y)$ and hence $u(y) + d(x,y) > u(x)$ for any $y \in K$. By the compactness of $K$, there exists $\delta_x > 0$ such that
\begin{equation}  \inf_{y \in K} \left[ u(y) + d(x,y) \right]
= \min_{y \in K} \left[ u(y) + d(x,y) \right]
> u(x) + \delta_x.
\label{eq_B2129}
\end{equation}
Since $u$ is $1$-Lipschitz, then $u(y) + d(x,y) \geq u(x)$ for all $y \in \cM$. Consequently, from (\ref{eq_C2132})
and (\ref{eq_B2129}),
$$ u_{\delta}(x) =  u(x) \qquad \qquad \qquad \text{when} \ 0 < \delta < \delta_x. $$
This proves (\ref{eq_C2127}) in the case where $x \in \cM \setminus A$.
Consider now the case where $x \in K$. Then,
\begin{equation}  u_{\delta}(x) = \inf_{y \in \cM} \left[ u(y) + d(x,y) - \delta \cdot 1_K(y) \right] \leq u(x) + d(x,x) - \delta = u(x) - \delta. \label{eq_C1520}
\end{equation}
From (\ref{eq_C1521}) and (\ref{eq_C1520}) we learn that $u_{\delta}(x) = u(x) - \delta$ for any $x \in K$ and $\delta > 0$. This proves (\ref{eq_C2127}) for the case where $x \in K$.
\end{proof}

Following Evans and Gangbo \cite[Lemma 5.1]{EG}, we say that a measurable subset $A \subseteq \cM$ is a {\it transport set}
associated with $u$ if for any $x \in A \setminus \Ends[u]$ and $\cI \in T[u]$,
\begin{equation}  x \in \cI \qquad \Longrightarrow \qquad \cI \subseteq A. \label{eq_C1055} \end{equation}
In other words, a transport set $A$ is a measurable set that contains all transport
rays intersecting $A \setminus \Ends[u]$.

\begin{lemma} Let $A \subseteq \cM$ be a transport set associated with $u$. Then,
$$ \int_A f d \mu \geq 0. $$
\label{lem_C1529}
\end{lemma}

\begin{proof} It suffices to prove that $\int_A f d \mu > -\eps$ for any $\eps > 0$. To this end, let us fix $\eps > 0$.
According to Lemma \ref{lem_C1047}, the set $\Ends[u]$ is of $\mu$-measure zero. Therefore,
\begin{equation}  \int_{A \setminus \Ends[u]} |f| d \mu  = \int_A |f| d \mu  < \infty. \label{eq_C1016} \end{equation}
Since $\mu$ is a Borel measure, it follows from (\ref{eq_C1016}) that there exists a compact $K \subseteq A \setminus \Ends[u]$ such that
\begin{equation}  \int_{A \setminus K} |f| d \mu < \eps. \label{eq_C1544} \end{equation}
For $\delta > 0$ we define $u_{\delta}: \cM \rightarrow \RR$ as in (\ref{eq_C2132}).
Then $u_{\delta}$ is a $1$-Lipschitz function, since it is the infimum of a family of $1$-Lipschitz functions.
From (\ref{eq_C1535}),
\begin{equation}  \int_{\cM} \frac{u - u_{\delta}}{\delta} \cdot f \cdot d \mu \geq 0  \qquad \qquad \qquad \text{for all} \ \delta > 0. \label{eq_C1536}
\end{equation}
For $k=1,2,\ldots$ denote
\begin{equation}  v_k(x) = \frac{u(x) - u_{1/k}(x)}{1/k}  \qquad \qquad \qquad (x \in \cM).
\label{eq_C2208} \end{equation}
From the ``Moreover'' part of Lemma \ref{lem_1532} we know that $0 \leq v_k(x) \leq 1$ for all $x \in \cM$ and $k \geq 1$.
According to Lemma \ref{lem_1532}, there exists
a function $v: \cM \rightarrow [0,1]$ such that $v_k(x) \longrightarrow v(x)$ for all $x \in \cM$.
Furthermore,  by (\ref{eq_C2127}),
\begin{equation}  v(x) = \left \{ \begin{array}{rl} 0 & x \in \cM \setminus A \\ 1 & x \in K \end{array} \right.
\label{eq_C2211} \end{equation}
where we used the fact that $A$ is a transport set and hence $A$ contains all transport rays intersecting $K \subseteq A \setminus \Ends[u]$.
Since $f$ is $\mu$-integrable and $|v_k(x)| \leq 1$ for all $k$ and $x$, then we may use the dominated convergence theorem and conclude from (\ref{eq_C1536})
and (\ref{eq_C2211}) that
\begin{equation}  0 \leq \int_{\cM} v_k  f d \mu \stackrel{k \rightarrow \infty}\longrightarrow \int_{\cM} v f d \mu
= \int_A v f d \mu = \int_{A \setminus K} v f d \mu + \int_K f d \mu. \label{eq_C1546}
\end{equation}
Since $v(x) \in [0,1]$ for all $x \in \cM$, then according to  (\ref{eq_C1544}) and (\ref{eq_C1546}),
$$ \int_K f d \mu \geq -\int_{A \setminus K} v f d \mu \geq -\int_{A \setminus K} |f| d \mu > - \eps, $$
and the lemma is proven.
\end{proof}

\begin{corollary} Let $A \subseteq \cM$ be a transport set associated with $u$. Then,
$$ \int_A f d \mu = 0. $$
\label{cor_C1529}
\end{corollary}

\begin{proof} In view of Lemma \ref{lem_C1529} we only need to prove that $\int_A f d \mu \leq 0$.
Note that the supremum of $\int v (-f) d \mu$ over all $1$-Lipschitz functions $v$ is attained for $v = -u$. Furthermore,
 $T[u] = T[-u]$ and $\Ends[u] = \Ends[-u]$. Therefore  $A$ is also a transport set associated with $-u$.
We may therefore apply Lemma \ref{lem_C1529} with $f$ replaced by $-f$ and with $u$ replaced by $-u$.
 By the conclusion of Lemma \ref{lem_C1529},
 $\int_A (-f) d \mu \geq 0$, and the corollary is proven.
 \end{proof}

Recall that $T^{\circ}[u]$ is a partition of $\Strain[u]$, and that $\pi: \Strain[u] \rightarrow T^{\circ}[u]$
is the partition map, i.e., $x \in \pi(x) \in T^{\circ}[u]$ for all $x \in \Strain[u]$.

\begin{lemma} Let $S \subseteq T^{\circ}[u]$. Assume that $\pi^{-1}(S) \subseteq \Strain[u]$ is a measurable subset of $\cM$. Then,
$$ \int_{\pi^{-1}(S)} f d \mu = 0. $$
\label{lem_C1046}
\end{lemma}

\begin{proof} Recall that $\Strain[u], \Inactive[u]$ and $\Ends[u]$ are three disjoint sets whose union equals $\cM$.
In view
of Lemma \ref{lem_C1047} and Corollary \ref{cor_C1529}, it suffices to show that there exists
a transport set $A \subseteq \cM$ with
\begin{equation}
\pi^{-1}(S) \subseteq A \qquad \text{and} \qquad A \setminus \pi^{-1}(S) \subseteq \Ends[u].
\label{eq_C1029} \end{equation}
Any $\cJ \in T^{\circ}[u]$ is the relative interior of a non-degenerate transport ray.
Since transport rays are closed sets,  it follows from Lemma \ref{lem_1102} that the closure $\overline{\cJ}$
of any $\cJ \in T^{\circ}[u]$ is  a transport ray.
We claim that for any $\cJ \in T^{\circ}[u]$,
\begin{equation}
 \overline{\cJ} \setminus \cJ \subseteq \cM \setminus (\Inactive[u] \cup \Strain[u]) = \Ends[u].
 \label{eq_C933}
 \end{equation}
Indeed, it follows from (\ref{eq_C1354}) that $\overline{\cJ}$ is contained in $\cM \setminus \Inactive[u]$ since it is a transport ray whose relative interior is non-empty.
Any point $x \in \overline{\cJ} $ belonging to $\Strain[u]$ must lie in $\cJ$,
according to Lemma \ref{lem_A1031}. Hence $\overline{\cJ} \setminus \cJ$ is disjoint from $\Strain[u]$, and (\ref{eq_C933})
is proven. Denote
\begin{equation}  A = \bigcup_{\cJ \in S} \overline{\cJ}. \label{eq_C1053} \end{equation}
Clearly $A \supseteq \bigcup_{\cJ \in S} \cJ = \pi^{-1}(S)$.
It follows from (\ref{eq_C933}) that
\begin{equation}  A \setminus \pi^{-1}(S) = \left \{\bigcup_{\cJ \in S} \overline{\cJ}  \right\} \setminus \left \{\bigcup_{\cJ \in S} \cJ  \right\}
\subseteq \bigcup_{\cJ \in S} (\overline{\cJ} \setminus \cJ) \subseteq \Ends[u]. \label{eq_C1046}
\end{equation}
Now (\ref{eq_C1029}) follows from (\ref{eq_C1046}) and from the fact that
$A \supseteq \bigcup_{\cJ \in S} \cJ = \pi^{-1}(S)$.
 All that remains
is to show that $A \subseteq \cM$ is a transport set. Since $\pi^{-1}(S)$ is assumed to be measurable
and $\Ends[u]$ is a null set, then the measurability of $A$ follows from (\ref{eq_C1029}).
In order to prove condition (\ref{eq_C1055}) and conclude that $A$ is a transport set, we choose $x \in A \setminus \Ends[u]$ and $\cI \in T[u]$ with
\begin{equation}  x \in \cI. \label{eq_C1050} \end{equation}
Since $x \in A \setminus \Ends[u]$, then
necessarily
$x \in \pi^{-1}(S) \subseteq \Strain[u]$ according to (\ref{eq_C1046}).
Denote by $\cJ$ the relative interior of the transport ray $\cI$.
 From (\ref{eq_C1050}) and Lemma \ref{lem_A1031}
we deduce that $\cI$ is the unique transport ray containing $x$, and that $x \in \cJ$.
Since $x \in \pi^{-1}(S)$, we  learn that $\cJ \in S$. From (\ref{eq_C1053}) we conclude that $\cI = \overline{\cJ} \subseteq A$.
We have thus verified condition (\ref{eq_C1055}) and proved that $A$ is a transport set associated with $u$.
The lemma is proven. \end{proof}

\begin{proof}[Proof of Theorem \ref{prop_intro}(B)]
The measurability of $\Strain[u]$ follows from Lemma \ref{lem_1055}.
We would like to show that
\begin{equation}
f(x) = 0 \qquad \qquad \qquad \text{for} \ \mu\text{-almost any point} \ x \in \cM \setminus \Strain[u].
\label{eq_C1715}
\end{equation}
We learn from (\ref{eq_C1354}) and from the definition
(\ref{eq_C1055})
that  any measurable set $S \subseteq \Inactive[u]$ is a transport set associated with $u$.
From Corollary \ref{cor_C1529}, for any measurable set $S \subseteq \Inactive[u]$,
$$ \int_{S} f d \mu = 0. $$
This implies that $f$ vanishes $\mu$-almost everywhere in $\Inactive[u]$.
Recall that $\cM \setminus \Strain[u] = \Inactive[u] \cup \Ends[u]$. In view of Lemma \ref{lem_C1047}, we
conclude (\ref{eq_C1715}).

\medskip Next, let $\nu$ and $\{ \mu_{\cI} \}_{\cI \in T^{\circ}[u]}$ be measures on $T^{\circ}[u]$
and $\cM$, respectively, satisfying conclusions (i), (ii) and (iii) of Theorem \ref{thm_main2}. Thus,
for $\nu$-almost any $\cI \in T^{\circ}[u]$, the measure $\mu_{\cI}$ is a $CD(\kappa,N)$-needle supported on $\cI$. Additionally,
 for any measurable set $A \subseteq \cM$,
\begin{equation}  \mu(A \cap \Strain[u]) = \int_{T^{\circ}[u]} \mu_{\cI}(A) d \nu(\cI),
\label{eq_C1713_}
\end{equation}
and in particular, the map $\cI \mapsto \mu_{\cI}(A)$ is $\nu$-measurable. It follows from (\ref{eq_C1713_}) that for any $\mu$-integrable function
$g: \cM \rightarrow \RR$,
\begin{equation}  \int_{\Strain[u]} g d \mu = \int_{T^{\circ}[u]} \left( \int_{\cI} g(x) d \mu_{\cI}(x) \right) d \nu(\cI).
\label{eq_C1713}
\end{equation}
In order to complete the proof, we need to show that
\begin{equation}
\int_{\cI} f d \mu_{\cI} = 0 \qquad \qquad \qquad \text{for} \ \nu\text{-almost any}  \ \cI \in T^{\circ}[u]. \label{eq_C1726} \end{equation}
Since $f$ is $\mu$-integrable, from (\ref{eq_C1713}) the map $\cI \mapsto \int_{\cI} f d \mu_{\cI}$ is
$\nu$-integrable, and in particular, it is well-defined for $\nu$-almost any $\cI \in T^{\circ}[u]$. The desired conclusion
(\ref{eq_C1726}) would follow once we show that for any $\nu$-measurable subset $S \subseteq T^{\circ}[u]$,
\begin{equation}
\int_S \left( \int_{\cI} f d \mu_{\cI} \right) d \nu(\cI) = 0.
\label{eq_C920}
\end{equation}
Thus, let us fix a $\nu$-measurable subset $S \subseteq T^{\circ}[u]$. From Theorem \ref{thm_main2}(i),
the set $\pi^{-1}(S)$ is a measurable subset of $\cM$. According to Lemma \ref{lem_C1046},
\begin{equation} 0 = \int_{\pi^{-1}(S)} f d \mu = \int_{\Strain[u]} f(x) \cdot 1_{\pi^{-1}(S)}(x) d \mu(x). \label{eq_C944}
\end{equation}
By using (\ref{eq_C1713}) and (\ref{eq_C944}),
$$ 0 = \int_{\Strain[u]} f \cdot 1_{\pi^{-1}(S)} d \mu =
\int_{T^{\circ}[u]} 1_S(\cI) \cdot \left( \int_{\cI} f d \mu_{\cI} \right) d \nu(\cI) = \int_S \left( \int_{\cI} f d \mu_{\cI} \right) d \nu(\cI). $$
Recalling that $S \subseteq T^{\circ}[u]$ was an arbitrary $\nu$-measurable set, we see that (\ref{eq_C920}) is proven.
The proof is complete.
\end{proof}

\begin{proof}[Proof of Theorem \ref{thm_main}]
From Theorem \ref{thm_main2}, Theorem \ref{prop_intro}(A) and Theorem \ref{prop_intro}(B) we obtain a $1$-Lipschitz function $u: \cM \rightarrow \RR$,
a certain measure $\nu$ on $T^{\circ}[u]$ and a family of measures $\{ \mu_{\cI} \}_{\cI \in T^{\circ}[u]}$ on the manifold $\cM$.
We make the following formal manipulations:
Let $\Omega$ be the partition of $\cM$ obtained by adding the singletons $\left \{ \{ x \} \, ; \, x \in \cM \setminus \Strain[u] \right \}$
to the partition $T^{\circ}[u]$ of $\Strain[u]$. Let $\tilde{\nu}$ be the push-forward of $\mu|_{\cM \setminus \Strain[u]}$ under the map
$x \mapsto \{ x \}$ to the set $\Omega$. Define
$$ \nu_1 = \nu + \tilde{\nu}, $$
a measure on $\Omega$. Finally, for $x \in \cM \setminus \Strain[u]$ write $\mu_{\{x \}}$ for Dirac's delta measure at $x$.
From Theorem \ref{thm_main2}, for any measurable subset $A \subseteq \cM$,
\begin{align*}  \mu(A) & = \mu(A \cap \Strain[u]) + \mu(A \setminus \Strain[u]) \\ & = \int_{T^{\circ}[u]} \mu_{\cI}(A) d \nu(\cI) + \int_{\cM \setminus \Strain[u]} \mu_{ \{ x \}} (A) d \mu(x)  =
\int_{\Omega} \mu_{\cI}(A) d \nu_1(\cI).
\end{align*}
Thus conclusion (i) holds true with $\nu$ replaced by $\nu_1$.
For $\nu_1$-almost any $\cI \in \Omega$, we have that either $\cI$
is a singleton, or else $\cI$ is the relative interior of a transport ray on which the $CD(\kappa,N)$-needle $\mu_{\cI}$ is supported.
We have thus verified conclusion (ii).
Theorem \ref{prop_intro}(B) shows that
$f$ vanishes almost everywhere in $\cM \setminus \Strain[u]$.
Conclusion (iii) thus follows from Theorem \ref{prop_intro}(B).
\end{proof}

\begin{proof}[Proof of Theorem \ref{prop_intro}(C)] This follows from Theorem \ref{thm_main2}(iii)
and the previous proof.
\end{proof}

\begin{corollary}[``Uniqueness of maximizer''] Let $(\cM, d, \mu)$ be an $n$-dimensional, geodesically-convex, weighted Riemannian manifold.
Suppose that $f: \cM \rightarrow \RR$ is a $\mu$-integrable function with $\int_{\cM} f d \mu = 0$ and
that there exists $x_0 \in \cM$ with $\int_{\cM} d(x_0,x) |f(x)| d \mu(x) < +\infty$. Assume furthermore that
\begin{equation}  \mu \left( \{ x \in \cM \, ; \, f(x) = 0 \} \right) = 0. \label{eq_C354} \end{equation}
Let $u_1, u_2: \cM \rightarrow \RR$ be $1$-Lipschitz functions with
\begin{equation}  \int_{\cM} u_1 f d \mu = \int_{\cM} u_2 f d \mu = \sup \left \{ \int_{\cM} u f d \mu \, ; \, u: \cM \rightarrow \RR, \, \| u \|_{Lip} \leq 1 \right \}. \label{eq_C357}
\end{equation}
Then $u_1 - u_2$ is a constant function.
\end{corollary}

\begin{proof} A $1$-Lipschitz function $u: \cM \rightarrow \RR$ for which the supremum in (\ref{eq_C357}) is attained
is called here a {\it maximizer}. According to (\ref{eq_C354}) and
Theorem \ref{prop_intro}(B),
the set $\cM \setminus Strain[u]$ is a Lebesgue-null set
for any maximizer $u$.
From Lemma \ref{lem_1046} we deduce that for any maximizer $u: \cM \rightarrow \RR$,
$$ |\nabla u(x)| = 1 \qquad \qquad \qquad \text{for almost any} \ x \in \cM. $$
Suppose now that $u_1$ and $u_2$ are two maximizers. Then
also $(u_1 + u_2) / 2$ is a  maximizer. Therefore for almost any $x \in \cM$,
$$ |\nabla u_1(x)| = |\nabla u_2(x)| = \left|\frac{\nabla u_1(x)  + \nabla u_2(x)}{2} \right| = 1. $$
Consequently $\nabla u_1 = \nabla u_2$ almost everywhere, and hence $u_1 - u_2 \equiv Const$.
\end{proof}

The $CD(\kappa, N)$ curvature-dimension condition
was used in our argument only in order to deduce that $N$-curvature needles are $CD(\kappa, N)$-needles.
The ``$N$-curvature needle'' variant of Theorem \ref{thm_main2} is rendered as Theorem \ref{thm_main2_} above.
Next we formulate an $N$-curvature variant of Theorem \ref{prop_intro}:

\begin{theorem}
Let $n \geq 2, \kappa \in \RR$ and $N \in (-\infty, 1) \cup [n, +\infty ]$.
Assume that $(\cM,d,\mu)$ is an $n$-dimensional weighted Riemannian manifold
which is geodesically-convex.
Let $f: \cM \rightarrow \RR$ be a $\mu$-integrable function with $ \int_\cM f d \mu = 0$.
Assume that there exists a point $x_0 \in \cM$ with $\int_\cM |f(x)| \cdot  d(x_0, x) d \mu(x) < \infty$.
Then,
\begin{enumerate}
\item[(A)] There exists a $1$-Lipschitz function $u: \cM \rightarrow \RR$ such that
$$ \int_{\cM} u f d \mu = \sup_{\| v \|_{Lip} \leq 1} \int_{\cM} v f d \mu.
$$
\item[(B)] For any such function $u$, the function $f$ vanishes $\mu$-almost everywhere in
$\cM \setminus \Strain[u]$. Furthermore, let $\nu$ and $\{ \mu_{\cI} \}_{\cI \in T^{\circ}[u]}$ be measures on $T^{\circ}[u]$
and $\cM$, respectively, satisfying conclusions (i), (ii) and (iii) of Theorem \ref{thm_main2_}.
Then for $\nu$-almost any $\cI \in T^{\circ}[u]$,
$$  \int_{\cI} f d \mu_{\cI} = 0.
$$
 \end{enumerate}
 \label{prop_intro_}
\end{theorem}

The proof of Theorem \ref{prop_intro_} is almost identical to the proof of Theorem \ref{prop_intro}.
The only difference is that one needs to appeal to Theorem \ref{thm_main2_} rather than
to Theorem \ref{thm_main2} rather than, and to replace the words
``$CD(\kappa, N)$-needle'' by ``$N$-curvature needle'' throughout the proof.

\medskip \begin{remark}{\rm
Similarly,
Theorem \ref{thm_main}
and Theorem \ref{prop_four_functions} remain valid without the $CD(\kappa, N)$-assumption,
yet one has to replace the words ``$CD(\kappa, N)$-needle'' by ``$N$-curvature needle''.
\label{rem_BL}
}\end{remark}

\section{Some applications}
\label{applications}
\setcounter{equation}{0}

One-dimensional log-concave needles are quite well-understood.
Theorem \ref{thm_main} allows us to reduce certain questions pertaining to Riemannian manifolds whose Ricci curvature is non-negative,
to analogous questions for one-dimensional log-concave needles.

\subsection{The inequalities of Buser, Ledoux and  E. Milman}
\setcounter{equation}{0}

Let $\cM$ be a Riemannian manifold with distance function $d$.
For a subset $S \subseteq \cM$ and $\eps > 0$  denote
$$ S_\eps = \left \{ x \in \cM \, ; \, \inf_{y \in S} d(x,y) < \eps \right \}, $$
the $\eps$-neighborhood of the set $S$.
The next proposition was proven by E. Milman \cite{e_milman_invent},
improving upon earlier results by Buser \cite{buser} and by Ledoux \cite{ledoux}:

\begin{proposition} Let $n \geq 2, R > 0$.
Assume that $(\cM,d,\mu)$ is an $n$-dimensional weighted Riemannian manifold
of class $CD(0, \infty)$ which is geodesically-convex with $\mu(\cM) = 1$.
Assume that for any $1$-Lipschitz function $u: \cM \rightarrow \RR$,
\begin{equation} \inf_{\alpha \in \RR} \int_{\cM} |u(x) - \alpha| d \mu(x) < R. \label{eq_E1139} \end{equation}
Then for any measurable set $S \subseteq \cM$ and $0 < \eps < R$,
$$ \mu(S_{\eps} \setminus S) \geq  c \cdot \frac{\eps}{R} \cdot \mu(S) \cdot (1 - \mu(S)), $$
where $c > 0$ is a universal constant.
\label{prop_1116}
\end{proposition}

It is well-known that the optimal choice of $\alpha$ in (\ref{eq_E1139}) is the median of the function $u$.
The expectation $E = \int_{\cM} u d \mu$ is also a reasonable choice for the parameter $\alpha$, since
$\int_{\cM} |u - E| d \mu$ is at most twice as large as the actual infimum in (\ref{eq_E1139}).
We begin the proof of Proposition \ref{prop_1116} with the following
standard estimate from the theory of
 one-dimensional log-concave measures:

\begin{lemma} Let $R > 0$, let $A \subseteq \RR$ be a non-empty, open connected set, let $\Psi: A \rightarrow \RR$ be a
convex function
with $\int_A e^{-\Psi} < \infty$, and let $\eta$ be the measure supported on $A$ whose density is $e^{-\Psi}$.
Suppose that $R =  \int_{A} |t| d \eta(t) / \eta(\RR)$.
Then for any  $0 < t < 1,  0 < \eps < 2 R$
and a measurable subset $S \subseteq \RR$,
\begin{equation}
 \eta(S) = t \cdot \eta(\RR)  \qquad \Longrightarrow \qquad
 \eta(S_{\eps} \setminus S) \geq c \cdot \frac{\eps}{ R } \cdot  t (1 -t) \cdot \eta(\RR),
 \label{eq_E1236} \end{equation}
 where $c > 0$ is a universal constant.
 \label{lem_1258}
\end{lemma}

\begin{proof} We may add a constant to $\Psi$ and stipulate that $\eta(\RR) = 1$.
We may rescale and assume furthermore that $R = \int_A |t| d \eta(t) = 1$.
According to Bobkov  \cite[Proposition 2.1]{bobkov},
it suffices to prove (\ref{eq_E1236}) under the additional assumption that $S$ is a half-line in $\RR$ with $\eta(S) = t$.
Reflecting $\Psi$ if necessary, we may suppose that  $S$ takes the form $S = (-\infty, a)$ for some $a \in A$.
Furthermore, we may assume that
\begin{equation} \eta\left( (a, a + \eps) \right) \leq  \min \{ t, 1 - t \}/2. \label{eq_E1450}
\end{equation}
Indeed, if (\ref{eq_E1450}) fails then $\eta(S_\eps \setminus S) = \eta \left( (a, a + \eps) \right) \geq (\eps/R) \cdot t(1-t)/4$
and (\ref{eq_E1236}) holds true. For $x \in \RR$ and $0 < s < 1$ denote
$$ \Phi(x) = \int_{-\infty}^x e^{-\Psi}, \qquad \qquad I(s) = \exp(-\Psi(\Phi^{-1}(s))). $$
Since $\Psi$ is convex, then $I: (0,1) \rightarrow (0, \infty)$ is a well-defined concave function according
to Bobkov \cite[Lemma 3.2]{bobkov2}. Furthermore, since $\int_{A} |t| d \eta(t) = 1$
then $I(1/2) \geq c$ where $c > 0$ is a universal constant, as is shown in \cite[Section 3]{bobkov2}.
Therefore, by the concavity of the non-negative function $I: (0,1) \rightarrow \RR$,
\begin{equation}  I(t) \geq 2c \cdot \min \{ t, 1 - t \} \qquad \qquad \text{for all} \ 0 < t < 1. \label{eq_E1453} \end{equation}
According to (\ref{eq_E1450}) and (\ref{eq_E1453}),
$$ \eta\left( (a, a + \eps) \right) \geq \eps \cdot \inf_{x \in (a, a+ \eps) \cap A} e^{-\Psi(x)} \geq
\eps \cdot \inf_{s \in [t, t + \min \{ t, 1-t \}/2]} I(s) \geq \eps \cdot c \cdot \min \{ t, 1-t \}, $$
and (\ref{eq_E1236}) is proven.
\end{proof}

\begin{proof}[Proof of Proposition \ref{prop_1116}]
Denote $t = \mu(S) \in [0,1]$. We may assume that $t \in (0,1)$, as otherwise there is nothing to prove.
Set $f(x) = 1_S(x) - t$ for $x \in \cM$. Then $\int_{\cM} f d\mu = 0$, and certainly for any $x_0 \in \cM$,
$$ \int_{\cM} |f(x)| \cdot d(x_0, x) d \mu(x) \leq |t + 1| \cdot \int_{\cM}  d(x_0, x) d \mu(x) < \infty, $$
where the integrability of the $1$-Lipschitz function $x \mapsto d(x_0, x)$ follows from (\ref{eq_E1139}).
Applying Theorem \ref{prop_intro}, we obtain a certain $1$-Lipschitz function $u: \cM \rightarrow \RR$
and measures $\nu$ and $\{ \mu_{\cI} \}_{\cI \in T^{\circ}[u]}$ on $T^{\circ}[u]$ and $\cM$ respectively.
It follows from (\ref{eq_E1139}) that after adding an appropriate constant to the
$1$-Lipschitz function $u$, we have
\begin{equation}  \int_{\cM} |u| d \mu \leq R. \label{eq_D1333} \end{equation}
For $\nu$-almost any $\cI \in T^{\circ}[u]$ we know that $\int_{\cI} f d \mu_{\cI} = 0$. Consequently,
for $\nu$-almost any $\cI \in T^{\circ}[u]$,
\begin{equation}
\mu_{\cI}(S) = t \cdot \mu_{\cI}(\cM) < \infty. \label{eq_E1212}
\end{equation}
From Theorem \ref{prop_intro}(B), the function $f$ vanishes $\mu$-almost everywhere outside $\Strain[u]$, but
our function $f(x) = 1_S(x) - t$ never vanishes in $\cM$.
Hence $\Strain[u]$ is a set of a full $\mu$-measure.
From Theorem \ref{thm_main2}(ii) and from (\ref{eq_D1333}) we thus obtain that
\begin{equation}  \int_{T^{\circ}[u]} \left( \int_{\cI} |u| d \mu_{\cI} \right) d \nu(\cI) = \int_{\Strain[u]} |u| d \mu = \int_{\cM} |u| d \mu \leq R.
\label{eq_E1143} \end{equation}
Denote
\begin{equation} B = \left \{ \cI \in T^{\circ}[u] \, ; \, \int_{\cI} |u| d \mu_{\cI} \leq 2 R \cdot \mu_{\cI}(\cM) \right \}. \label{eq_E1144}
\end{equation}
Since $\mu(\cM) = \mu(\Strain[u]) = 1$ then $\int_{T^{\circ}[u]} \mu_{\cI}(\cM) d \nu(\cI) = 1$. From (\ref{eq_E1143}) and
the Markov-Chebyshev inequality,
\begin{equation}
\int_B \mu_{\cI}(\cM) d \nu(\cI) \geq \frac{1}{2}. \label{eq_E1144_}
\end{equation}
Furthermore, $\mu_{\cI}$ is a log-concave needle (i.e., a $CD(0, \infty)$-needle) for $\nu$-almost any $\cI \in B$.
We would like to show that for $\nu$-almost any $\cI \in B$ and any $0 < \eps < R$,
\begin{equation}
\mu_{\cI}(S_{\eps} \setminus S) \geq c \cdot \frac{\eps}{R} \cdot  t (1 -t) \cdot \mu_{\cI}(\cM),
\label{eq_E1215}
\end{equation}
for a universal constant $c > 0$.
Let us fix $\cI \in B$ such that $\mu_{\cI}$ is a log-concave needle for which (\ref{eq_E1212}) holds true.
Let $A \subseteq \RR, \Psi: A \rightarrow \RR$ and $\gamma: A \rightarrow \cM$ be as in Definition \ref{def_cd}.
Then $A \subseteq \RR$ is a non-empty, open, connected set and $\Psi: A \rightarrow \RR$ is smooth and convex.
From Theorem \ref{thm_main2}(iii) we know that $\cI = \gamma(A)$ and
\begin{equation}  u(\gamma(t)) = t \qquad \qquad \qquad \text{for all} \ t \in A. \label{eq_E1336} \end{equation}
Since $\cI \in B$, we may apply Lemma
\ref{lem_1258} thanks to (\ref{eq_E1212}), (\ref{eq_E1144}) and (\ref{eq_E1336}).
The conclusion of Lemma \ref{lem_1258} implies (\ref{eq_E1215}). Consequently, for any $0 < \eps < R$,
$$ \mu(S_{\eps} \setminus S) = \int_{T^{\circ}[u]} \mu_{\cI}(S_{\eps} \setminus S) d \nu(\cI)
\geq \int_{B} \mu_{\cI}(S_{\eps} \setminus S) d \nu(\cI) \geq c \frac{\eps}{R} \cdot t (1 -t) \cdot \int_B \mu_{\cI}(\cM) d \nu(\cI). $$
The proposition now follows from (\ref{eq_E1144_}).
\end{proof}

Proposition \ref{prop_1116} is stated and proved in the particular case where $\kappa = 0$ and $N = \infty$.
For general $\kappa$ and $N$, an appropriate $CD(\kappa,N)$-variant of the one-dimensional
Lemma \ref{lem_1258} would lead to a $CD(\kappa,N)$-variant of the $n$-dimensional Proposition \ref{prop_1116}.

\subsection{A Poincar\'e inequality for geodesically-convex domains}
\setcounter{equation}{0}

For $\kappa \in \RR, 1 \neq N \in \RR \cup \{ +\infty \}$ and $D \in (0, +\infty)$
write $\cF_{\kappa, N, D}$ for the collection of all measures $\nu$ supported on the interval $(0, D) \subseteq \RR$
which are $CD(\kappa, N)$-needles. According to Definition \ref{def_cd}, a measure $\nu$  belongs to
$\cF_{\kappa, N, D}$ if and only if $\nu$ is supported on a non-empty, open interval $A \subseteq (0, D)$ with density
$e^{-\Psi}$, where $\Psi: A \rightarrow \RR$ is a smooth function that satisfies
\begin{equation}
\Psi^{\prime \prime} \geq \kappa + \frac{(\Psi^{\prime})^2}{N-1}.
\label{eq_E2030} \end{equation}
The term $(\Psi^{\prime})^2 / (N-1)$ in (\ref{eq_E2030}) is interpreted as zero when $N = +\infty$.
In order to include the case $D = +\infty$, we write $\cF_{\kappa, N, +\infty}$ for the collection of
all measures $\nu$ on $\RR$ which are $CD(\kappa, N)$-needles.
Define
$$ \lambda_{\kappa, N, D} = \inf \left \{ \frac{\int_{\RR} |u'|^2 d \nu}{\int_{\RR} u^2 d \nu} \, ; \, \nu \in \cF_{\kappa, N, D}, \,
u \in C^1 \cap L^{1 \cap 2}(\nu), \, \int_{\RR} u d \nu =0, \, \int_{\RR} u^2 d \nu > 0 \right \}, $$
where $L^{1 \cap 2}(\nu)$ is an abbreviation for $L^1(\nu) \cap L^2(\nu)$.
There are some cases where $\lambda_{\kappa, N, D}$ may be computed explictely. For example, for
$N \in (-\infty, -1] \cup (1, +\infty ]$, the simple one-dimensional lemma of Payne and Weinberger \cite{PW} shows that
\begin{equation}  \lambda_{0, N, D} = \frac{\pi^2}{D^2}. \label{eq_E2209} \end{equation}
We refer the reader to Bakry and Qian \cite{BQ} and references therein for generalizations
of the following proposition:

\begin{proposition} Let $n \geq 2, \kappa \in \RR$ and $N \in (-\infty, 1) \cup [n, +\infty ]$.
Assume that $(\cM,d,\mu)$ is an $n$-dimensional weighted Riemannian manifold
of class $CD(\kappa, N)$ which is geodesically-convex. Denote
$$ D = \Diam(\cM) = \sup_{x, y \in \cM} d(x,y) \in (0, +\infty] $$
the diameter of $\cM$. Then for any $C^1$-function $f: \cM \rightarrow \RR$ with $f \in L^1(\mu) \cap L^2(\mu)$,
\begin{equation}  \int_{\cM} f d \mu = 0 \qquad \Longrightarrow \qquad \lambda_{\kappa, N, D} \cdot \int_{\cM} f^2 d \mu \leq \int_{\cM} |\nabla f|^2 d \mu.
\label{eq_E1141} \end{equation}
\label{prop_1144}
\end{proposition}

\begin{proof} Let $f: \cM \rightarrow \RR$ be a $C^1$-function with $f \in L^{1 \cap 2}(\mu)$ and $\int_{\cM} f d \mu = 0$.
Applying Theorem \ref{thm_main}, we see that (\ref{eq_E1141}) would follow from the following inequality:
for any measure $\nu$ on $\cM$ which is a $CD(\kappa, N)$-needle,
\begin{equation}  \left[ f \in L^{1 \cap 2}(\nu) \ \ \  \text{and}  \ \ \  \int_{\cM} f d \nu = 0 \right] \qquad \Longrightarrow \qquad \lambda_{\kappa, N, D} \cdot \int_{\cM} f^2 d \nu \leq \int_{\cM} |\nabla f|^2 d \nu.
\label{eq_E1141_} \end{equation}
Thus, let us fix a $CD(\kappa, N)$-needle $\nu$ for which
$f \in L^{1 \cap 2}(\nu)$
and
$\int_{\cM} f d \nu = 0$.
Let $A \subseteq \RR, \Psi: A \rightarrow \RR$ and $\gamma: A \rightarrow \cM$ be as in Definition \ref{def_cd}. Denoting $g = f \circ \gamma$, we see that
$$
|g^{\prime}(t)| \leq |\nabla f(\gamma(t))| \qquad \qquad \qquad \text{for} \ t \in A,
$$
as $\gamma$ is a unit speed geodesic. Hence
(\ref{eq_E1141_}) would follow from
the inequality
\begin{equation}  \int_{A} g e^{-\Psi}  = 0 \qquad \Longrightarrow \qquad \lambda_{\kappa, N, D} \cdot \int_{A} g^2 e^{-\Psi} \leq  \int_{A} (g^{\prime})^2 e^{-\Psi},
\label{eq_E1141__} \end{equation}
where $g: A \rightarrow \RR$ is a $C^1$-function with $\int_A \left( |g| + g^2 \right) e^{-\Psi} < \infty$.
The set $A$ is open and connected, and  since $\gamma: A \rightarrow \cM$ is a minimizing geodesic then $A$ is an open interval whose length is at most $D$.
The smooth function $\Psi: A \rightarrow \RR$ satisfies (\ref{eq_E2030}),
and the desired inequality (\ref{eq_E1141__}) holds in view of the definition of $\lambda_{\kappa, N, D}$.
This completes the proof.
\end{proof}

The case $\kappa = 0$
of Proposition \ref{prop_1144}, with the constant $\lambda_{0, N, D}$ given
by (\ref{eq_E2209}),
appears in  Payne-Weinberger \cite{PW} in the Euclidean case,
and in Li-Yau \cite{LY} and Yang-Zhong \cite{YZ} in the Riemannian case.

\subsection{The isoperimetric inequality and its relatives}
\setcounter{equation}{0}

Recall the definition of
$\cF_{\kappa, N, D}$ from the previous subsection. Recall that $A_\eps$ stands for the $\eps$-neighborhood of the set $A$.
For $0 < t < 1$ and $\eps > 0$ define
\begin{equation}  I_{\kappa, N, D}(t, \eps) = \inf \left \{ \nu(A_\eps) \, ; \, \nu \in \cF_{\kappa, N, D}, \, A \subseteq \RR, \, \nu(\RR) = 1, \, \nu(A) = t \right \}. \label{eq_E2237}
\end{equation}
That is, $I_{\kappa,N, D}(t,\eps)$ is the infimal measure of an $\eps$-neighborhood  of a subset of measure $t$.
There
are cases where the function $I_{\kappa, N, D}$ may be computed explicitly.
For example, when $\kappa > 0, N = D = \infty$, the infimum in (\ref{eq_E2237})
is attained when $A$ is a half-line and $\nu$ is a Gaussian measure on the real line of variance $1/\kappa$.
See E. Milman \cite{E_milman_model} and references therein for more information
about the function $I_{\kappa, N, D}$.

\begin{proposition} Let $n \geq 2, \kappa \in \RR$ and $N \in (-\infty, 1) \cup [n, +\infty ]$.
Assume that $(\cM,d,\mu)$ is an $n$-dimensional weighted Riemannian manifold
of class $CD(\kappa, N)$ which is geodesically-convex. Assume that $\mu(\cM) = 1$.
Denote $ D = \Diam(\cM)$, the diameter of $\cM$. Then for any measurable set $A \subseteq \cM$ and $\eps > 0$, denoting $t = \mu(A)$,
$$ \mu(A_\eps) \geq I_{\kappa, N, D}(t, \eps). $$
\label{prop_2241}
\end{proposition}

\begin{proof} Denote $f(x) = 1_A(x) - t$. Then $\int_{\cM} f d \mu = 0$. The proposition follows
by applying Theorem \ref{thm_main} and arguing similarly to the proof of Proposition \ref{prop_1144}.
\end{proof}

Similarly, one may reduce the proof of log-Sobolev or transportation-cost inequalities to the one-dimensional case
by using Theorem \ref{thm_main}, as well as the proof of the inequalities of
Cordero-Erausquin, McCann and Schmuckenschl\"ager \cite{CMS1, CMS2}.
By using Theorem \ref{prop_intro_}, it is also straightforward to reduce the proof
of the Brascamp-Lieb inequality and its dimensional variants to the one-dimensional case.
We will end this section
with the proof of the {\it four functions theorem},
rendered as Theorem  \ref{prop_four_functions} above.

\begin{proof}[Proof of Theorem \ref{prop_four_functions}]
By approximation, we may assume that the function $f_3: \cM \rightarrow [0, +\infty)$ does not vanish in $\cM$
(for example, replace $f_3$ by $f_3 + \eps g$ where $g$ is a positive function with suitable integrability properties,
and then let $\eps$ tend to zero). We claim that for any $CD(\kappa, N)$-measure $\eta$ on the Riemannian manifold $\cM$ for which $f_1,f_2,f_3, f_4 \in L^1(\eta)$,
\begin{equation}  \left( \int_{\cM} f_1 d \eta \right)^{\alpha} \left( \int_{\cM} f_2 d \eta \right)^{\beta}
\leq \left( \int_{\cM} f_3 d \eta \right)^{\alpha} \left( \int_{\cM} f_4 d \eta \right)^{\beta}. \label{eq_E951_}
\end{equation}
Indeed, inequality (\ref{eq_E951_}) appears in the assumptions of the theorem, but under the additional assumption that $\eta$
is a probability measure. By homogeneity, (\ref{eq_E951_}) holds true under the additional assumption that  $\eta$ is a finite measure.
In the general case, we may select a sequence of finite $CD(\kappa,N)$-measures $\eta_\ell$ such that $\eta_\ell \nearrow \eta$, and use
the monotone convergence theorem. Thus (\ref{eq_E951_}) is proven.

\medskip
 Next, denote $\lambda = \int_{\cM} f_1 d \mu / \int_{\cM} f_3 d \mu$,
define $f = f_1 - \lambda f_3$,
and apply Theorem \ref{thm_main}.
Let $\Omega, \{ \mu_{\cI} \}_{\cI \in \Omega}, \nu$ be as in Theorem \ref{thm_main}. Then for $\nu$-almost any $\cI \in \Omega$
we have that $f_1,f_2,f_3,f_4 \in L^1(\mu_{\cI})$ and
\begin{equation} \left( \int_{\cI} f_1 d \mu_{\cI} \right)^{\alpha} \left( \int_{\cI} f_2 d \mu_{\cI} \right)^{\beta}
\leq \left( \int_{\cI} f_3 d \mu_{\cI} \right)^{\alpha} \left( \int_{\cI} f_4 d \mu_{\cI} \right)^{\beta} \label{eq_E1002} \end{equation}
as follows from (\ref{eq_E951_}) and from the pointwise inequality
$f_1^{\alpha} f_2^{\beta} \leq f_3^{\alpha} f_4^{\beta}$ that holds almost-everywhere in $\cM$.
However, $\int_{\cI} f_1 d \mu_{\cI} = \lambda \int_{\cI} f_3 d \mu_{\cI}$ for $\nu$-almost any $\cI \in \Omega$.
Thus (\ref{eq_E1002}) implies that for $\nu$-almost any $\cI \in \Omega$,
\begin{equation} \lambda^{\alpha  / \beta} \int_{\cI} f_2 d \mu_{\cI}
\leq  \int_{\cI} f_4 d \mu_{\cI}.  \label{eq_E1006} \end{equation}
Integrating (\ref{eq_E1006}) with respect to the measure $\nu$ yields
$$ \lambda^{\alpha  / \beta} \int_{\cM} f_2 d \mu = \lambda^{\alpha  / \beta} \int_{\Omega} \left( \int_{\cI}
f_2 d \mu_{\cI} \right) d \nu(\cI) \leq \int_{\Omega} \left( \int_{\cI}
 f_4 d \mu_{\cI} \right) d \nu(\cI) = \int_{\cM} f_4 d \mu. $$
 From the definition of $\lambda$ we thus obtain
$$  \left( \int_{\cM} f_1 d \mu \right)^{\alpha} \left( \int_{\cM} f_2 d \mu \right)^{\beta}
\leq \left( \int_{\cM} f_3 d \mu \right)^{\alpha} \left( \int_{\cM} f_4 d \mu \right)^{\beta}, $$
and the theorem is proven.
\end{proof}

\section{Further research}
\setcounter{equation}{0}
\label{sec_future}

This section contains ideas and conjectures for possible
extensions of the results in this manuscript. First, we conjecture that the results and the arguments presented above  may be generalized
to the case of a smooth Finsler manifold. Another interesting generalization involves {\it several constraints}.
That is, suppose that we are given a
weighted Riemannian manifold $(\cM, d, \mu)$ and a $\mu$-integrable function $f: \cM \rightarrow \RR^k$ with
$$ \int_{\cM} f d \mu = 0.
$$
We would like to understand whether the measure $\mu$ may be decomposed into $k$-dimensional pieces
in a way analogous to Theorem \ref{thm_main}.

\begin{definition} Let $\cM$ and $\cN$ be geodesically-convex Riemannian manifolds. We declare that
``$\cM \rightarrow \cN$ has the isometric extension property''
if for any subset $A \subseteq \cM$ and a distance-preserving map $f: A \rightarrow \cN$, there exists
a geodesically-convex subset $B \subseteq \cM$ containing $A$ and an extension of $f$ to a
distance-preserving map $f: B \rightarrow \cN$.
\end{definition}

Lemma \ref{lem_1522} shows that $\RR \rightarrow \cM$ has the isometric extension property whenever
$\cM$ is a geodesically-convex Riemannian manifold. If $\cM \subseteq \RR^n$ is a convex set then for any $k \leq n$,
$$ \RR^k \rightarrow \cM $$
has the isometric extension property. Also $S^k \rightarrow S^n$ has the isometric extension property, as well as $S^k \rightarrow \cM$
when $\cM$ is a geodesically-convex subset of the sphere $S^n$. These facts have direct proofs which do not
rely on the Kirszbraun theorem.
Let us discuss in greater detail the case where $\cM \subseteq \RR^n$ is an open, convex
set. Suppose that $u: \cM \rightarrow \RR^k$ is a $1$-Lipschitz map. We may generalize Definition
\ref{def_1050} as follows: A subset
 $\cS \subseteq \cM$ is a {\it leaf} associated with $u$ if
$$  |u(x) - u(y)| = |x- y|  \qquad \qquad \qquad \text{for all} \, x, y \in \cS,
$$
and if for any $\cS_1 \supsetneq \cS$ there exist
$x,y \in \cS_1$ with $|u(x) - u(y)| < |x - y|$. For any leaf $\cS \subseteq \cM$, the set
$$ u(\cS) = \left \{ u(x) \, ; \, x \in \cS \right \} $$
is a closed, convex subset of $\RR^k$. This follows from the isometric extension property of
$\RR^k \rightarrow \cM$. Let us define $\Strain[u]$ to be the union of all relative interiors of leafs.
Write $T^{\circ}[u]$ for the collection of all non-empty relative interiors of leafs. Suppose that $\mu$
is a measure on the convex set $\cM \subseteq \RR^n$ such that $(\cM,| \cdot |,\mu)$ is an $n$-dimensional weighted Riemannian manifold
of class $CD(\kappa, N)$. We conjecture that there exists a measure $\nu$ on $T^{\circ}[u]$
and a family of measures $\{ \mu_{\cS} \}_{\cS \in T^{\circ}[u]}$ such that
$$ \mu(A \cap \Strain[u]) = \int_{T^{\circ}[u]} \mu_{\cS}(A) d \nu(\cS) \qquad \qquad \text{for any measurable} \ A \subseteq \cM. $$
Additionally,  for $\nu$-almost any $\cS \in T^{\circ}[u]$, the measure $\mu_{\cS}$ is supported on $\cS$ and
$$ (\cS, | \cdot |, \mu_{\cS}) $$
is a weighted Riemannian manifold of class $CD(\kappa, N)$.
In other words, at least in the Euclidean setting, we conjecture that Theorem
\ref{thm_main2} admits a direct generalization to functions $u: \cM \rightarrow \RR^k$.
Perhaps the generalization works whenever $u: \cM \rightarrow \cN$ is $1$-Lipschitz,
where $\cN \rightarrow \cM$ has the isometric extension property, and we require certain
bounds on sectional curvatures.
Moreover, in the Euclidean setting, we believe that Theorem \ref{prop_intro} may be generalized as follows:
Assume that $f: \cM \rightarrow \RR^k$ satisfies $\int_{\cM} f d \mu = 0$ and also
$\int_{\cM} |f(x)| \cdot d(x_0, x) d \mu(x) < +\infty$ for a certain $x_0 \in \cM$.
Let us maximize
\begin{equation}  \int_{\cM} \langle f,  u \rangle d \mu \label{eq_I1116} \end{equation}
among all $1$-Lipschitz functions $u: \cM \rightarrow \RR^k$. One may use Kirszbraun's theorem
and prove that for any maximizer $u: \cM \rightarrow \RR^k$  and for $\nu$-almost any leaf $\cS \in T^{\circ}[u]$,
$$ \int_{\cS} f d \mu_{\cS} = 0 \quad \text{and} \quad  \int_{\cM} \langle f,  u \rangle d \mu_{\cS} = \sup \left \{ \int_{\cS} \langle f,  v \rangle d \mu_{\cS} \, ; \,  v: \cS \rightarrow \RR^k, \, \| v \|_{Lip} \leq 1 \right \}.
$$

\begin{remark}{\rm The bisection method outlined in Section \ref{sec_intro} has one significant
advantage compared to our results. The methods discussed in this manuscript are very much {\it linear}, as we obtain a geodesic foliation from
the linear maximization problem (\ref{eq_I1116}). In comparison, the bisection method works only in symmetric spaces such as $\RR^n$ or $S^n$,
but in these spaces it offers more flexibility, since one may devise various  linear and non-linear rules
for the bisection procedure. This flexibility is exploited artfully by Gromov \cite{gr_waist}. It is currently
unclear to us whether one may arrive at an {\it integrable} foliation in the situations considered by Gromov \cite{gr_waist}.
}\end{remark}

\medskip Another possible research direction is concerned with $CD(\kappa, N+1)$-needles in one dimension.
It seems that many concepts and results
from convexity theory admit  generalizations to the class of $CD(\kappa, N+1)$-needles.
For example, when $0 \neq N \in \RR$ and $\kappa/N > 0$, we may define a Legendre-type transform
of a function $f: \RR \rightarrow [0, +\infty]$ by setting
\begin{equation} f^*(s) = \inf_{t ; f(t) < +\infty} \frac{g(s +t)}{f(t)} \qquad \qquad \text{for} \ s \in \RR, \label{eq_E1026} \end{equation}
where $$ g(t) =
\left \{ \sin \left( \sqrt{\frac{\kappa}{N}} \cdot t \right) \cdot 1_{[0, \pi]} \left(\sqrt{\frac{\kappa}{N}} \cdot t \right) \right \}^N $$
and we agree that $g(s + t) / 0 \equiv +\infty$
and that $0^{N} = 0$ when $N \in (0, +\infty)$ and $0^N = +\infty$ when $N \in (-\infty, 0)$.
It seems that the function $f^*$ is either a density of a $CD(\kappa, N+1)$-needle in $\RR$,
or else it is a  limit of such densities.
We say that a function $f: \RR \rightarrow [0, +\infty]$ is
 $(\kappa, N+1)$-concave if the set
$$  \{ t \in \RR \, ; \, f(t) > R \cdot g(s + t)  \} $$
is connected for all $R > 0, s \in \RR$.
Perhaps the transform (\ref{eq_E1026}) is an order-reversing involution on the class of
upper semi-continuous $(\kappa, N+1)$-concave functions on $\RR$.

\medskip One reason for investigating one-dimensional $CD(\kappa, N)$-needles
is that $CD(\kappa, N)$-needles may be further decomposed into needles of a simpler form
that satisfy a certain linear constraint.
This was already discovered by Lov\'asz and Simonovits \cite{LS} in the most interesting case $\kappa = 0, N = n$.

\begin{definition} Let $\kappa \in \RR, 1 \neq N \in \RR \cup \{ \infty \}$ and let $\nu$ be  a measure on
a certain Riemannian manifold $\cM$ which is a $CD(\kappa, N)$-needle. Let $A, \Psi$ and $\gamma$ be as in Definition \ref{def_cd}.
We say that $\nu$ is a ``$CD(\kappa,N)$-affine needle'' if the following inequality holds true in the entire set $A$:
$$
\Psi^{\prime \prime} = \kappa + \frac{(\Psi^{\prime})^2}{N-1},
$$
where in the case $N = \infty$, we interpret the term $(\Psi^{\prime})^2 / (N-1)$ as zero.
\label{def_cd2}
\end{definition}

For $x \in \RR$ write $x_+ = \max \{x, 0 \}$.
The class of $CD(\kappa, N)$-affine needles may be described explicitly, as follows:

\begin{enumerate}
\item The {\it exponential needles} are $CD(0, \infty)$-affine needles, for which the function
 $e^{-\Psi}$ is an exponential function restricted
to the open, connected set  $A$.
That is, the function $e^{-\Psi}$ takes the form $$ A \ni t \mapsto \alpha \cdot e^{\beta \cdot t} $$ for certain $\beta \in \RR, \alpha > 0$.
The {\it $\kappa$-log-affine needles} are $CD(\kappa, \infty)$-affine needles, for which  $\Psi(t) - \kappa t^2/2$ is an affine function
in the open, connected set $A$.

\item  The {\it $N$-affine needles} are  $CD(0, N+1)$-affine needles with $0 \neq N \in \RR$,
for which  $f^{1/N}$ is an affine function in the open, connected set  $A$.

\item  For $0 \neq \kappa \in \RR$ and $0 \neq N \in \RR$, the $CD(\kappa, N+1)$-affine needles satisfy, for all $t \in A$,
$$ e^{-\Psi(t)} = \left \{ \begin{array}{lr}
       \left \{ \alpha \cdot \sin \left( \sqrt{\frac{\kappa}{N}} t  - \beta \right) \cdot 1_{[0, \pi]} \left(\sqrt{\frac{\kappa}{N}} t - \beta\right) \right \}_+^N & \kappa/N > 0 \vspace{5pt} \\
    (\alpha + t \beta)_+^N & \kappa = 0 \vspace{5pt} \\
     \left(\alpha \cdot \sinh \left( \sqrt{|\frac{\kappa}{N}|} \cdot t \right) + \beta \cdot \cosh \left( \sqrt{|\frac{\kappa}{N}|} \cdot t \right) \right)^N_+
     & \kappa/N <0
     \end{array} \right. $$
for some $\alpha, \beta \in \RR$.
\end{enumerate}

In the case where $N \in (0, +\infty]$ and $\kappa \geq 0$ it seems pretty safe to make the following:

\begin{conjecture} Let $\mu$ be a probability measure on
$\RR$ which is a $CD(\kappa,N+1)$-needle. Let $\vphi: \RR \rightarrow \RR$ be a
continuous,   $\mu$-integrable  function with
$ \int_{\RR} \vphi d \mu = 0$.
Then there exist probability measures $\{ \mu_{\alpha} \}_{\alpha \in \Omega}$ on $\RR$ and a probability measure $\nu$ on the set $\Omega$ such that:
\begin{enumerate}
\item[(i)] For any Lebesgue-measurable set $A \subseteq \RR$ we have $\mu(A) = \int_{\Omega} \mu_{\alpha}(A) d \nu(\alpha)$.
\item[(ii)] For $\nu$-almost any $\alpha \in \Omega$, the measure $\mu_{\alpha}$ is either supported on a singleton, or else it is a $CD(\kappa, N+1)$-affine needle with $  \int_{\RR}
\vphi d \mu_{\alpha} = 0$.
\end{enumerate}
\label{prop_551}
\end{conjecture}

Conjecture \ref{prop_551} reduces certain questions on $CD(\kappa, N+1)$-needles
to an inequality involving only two or three real parameters. A proof of Conjecture \ref{prop_551} in the case where $N = +\infty$ or $\kappa = 0$
follows from Choquet's integral representation theorem and the results of Fradelizi and Gu\'edon \cite{FG}. We are not sure what should be the correct formulation of Conjecture
\ref{prop_551} in the case where $N < 0$ and $\kappa < 0$.

\bigskip \bigskip
\noindent {\large \bf Appendix: The Feldman-McCann proof of Lemma \ref{lem_849}}
\addcontentsline{toc}{section}{Appendix: The Feldman-McCann proof of Lemma \ref{lem_849}}
\setcounter{equation}{0}
\bigskip \smallskip

In this appendix we describe the Feldman-McCann proof of Lemma \ref{lem_849}.
Let $\cM$ be a Riemannian manifold with distance function $d$.
Fix $p \in \cM$ and let $\delta_0 = \delta_0(p) > 0$ be the constant provided by Lemma \ref{lem_1315}.
Thus, $U = B_{\cM}(p, \delta_0/2)$ is a strongly-convex set. As in Section \ref{charts},
for $a \in U$ we write $$ U_a = \exp_a^{-1}(U) \subseteq T_a \cM, $$
a convex subset of $T_a \cM$.
For $a \in U$ and $X, Y \in U_a$, denoting $x = \exp_a(X), y = \exp_a(Y)$ we set
$$ F_a(X,Y) = \exp_{x}^{-1}(y) \in T_x \cM, $$
and also
$$ \Phi_a(X, Y) = \text{The parallel translate of} \ F_a(X,Y) \ \text{along the unique geodesic from} \ x \ \text{to} \ a. $$
The map $\Phi_a: U_a \times U_a \rightarrow T_a \cM$ satisfies
\begin{equation} |\Phi_a(X, Y)| = |F_a(X,Y)| = d(\exp_a X, \exp_a Y). \label{eq_1738_} \end{equation}
The behavior of $\Phi_{a}$ on lines through the origin is quite simple: Since $\exp_a(s X)$
and $\exp_a(t X)$ lie on the same geodesic emanating from $a$, then for any $X \in T_a \cM$ and $s,t \in \RR$,
\begin{equation}
\Phi_a(s X, t X) = (t -s) X \qquad \qquad \qquad \text{when} \ s X, t X \in U_a.
\label{eq_2224_}
\end{equation}
See \cite[Section 3.2]{FM} for more details about $\Phi_a$. Our next lemma is precisely
Lemma 14 in \cite{FM}. The proof given in \cite[Lemma 14]{FM} is very simple and uses essentially the same notation
as ours, and it is not reproduced here. In fact, the argument is similar to the proof of Lemma \ref{lem_1737} above,
and it relies only on  the smoothness of $\Phi_a$ and on the relation $\Phi_a(0,Y) = Y$ that follows from (\ref{eq_2224_}).

\medskip
\noindent {\bf Lemma A.1. } {\it
Let $a \in U$ and $X, Y_1, Y_2 \in U_a$. Then,
$$ \left| \Phi_a(X, Y_2) - \Phi_a(X, Y_1) \, - \, (Y_2 - Y_1) \right| \leq \bar{C}_{p} \cdot |X| \cdot |Y_1 - Y_2|,
$$
where $\bar{C}_{p} > 0$ is a constant depending only on $p$.}

\medskip
\begin{proof}[Proof of Lemma \ref{lem_849} (due to Feldman and McCann \cite{FM})]
Define
\begin{equation}
\delta_1 = \delta_1(p) = \min \left \{ \frac{1}{2000 \cdot \bar{C}_{p}}, \frac{\delta_0}{2} \right \}, \label{eq_A1038}
\end{equation}
where $\bar{C}_{p} > 0$ is the constant from Lemma A.1.
Both the assumptions and the conclusion of the lemma are not altered
if we replace $x_i, y_i$ by $x_{2-i}, y_{2-i}$ for $i=0,1,2$. Applying this replacement if necessary,
we assume from now on that
\begin{equation} d(x_0, y_0) \leq d(x_2, y_2). \label{eq_1413} \end{equation}
The points $x_0,x_1,x_2,y_0, y_1,y_2$ belong to $B_{\cM}(p, \delta_1) \subseteq U$. Recall that the main assumption of the Lemma is that
\begin{equation}  d(x_i, x_j) = d(y_i, y_j) = \sigma |i-j| \leq d(x_i, y_j) \qquad \qquad \text{for} \ i,j=0,1,2. \label{eq_0845_} \end{equation}
Define
\begin{equation}
\eps := d(x_1, y_1).
\label{eq_1746}
\end{equation}
Denote $a = x_0$ and let $X_0,X_1, X_2, Y_0, Y_1, Y_2 \in U_a$
be such that $x_i = \exp_a(X_i)$ and $y_i = \exp_a(Y_i)$ for $i=0,1,2$. Since $a = x_0$ then
$$ X_0 = 0. $$
For $i=0,1,2$ we know that $x_i, y_i \in B_{\cM}(p, \delta_1)$ and $X_i, Y_i \in U_a$.
It follows from (\ref{eq_1738_}), (\ref{eq_2224_})
and (\ref{eq_0845_}) that
\begin{equation}
|X_i| = |\Phi_a(X_0, X_i)| = d(x_0, x_i) \leq 2 \delta_1, \quad
|Y_i| = |\Phi_a(X_0, Y_i)| = d(x_0, y_i) \leq 2 \delta_1.
 \label{eq_921}
\end{equation}
By using (\ref{eq_921})  and Lemma A.1,
for any $R, Z, W \in \{ 0= X_0, X_1, X_2, Y_0, Y_1, Y_2 \}$,
\begin{equation}
\left| \Phi_a(R, Z) - \Phi_a(R, W)  - (Z - W)  \right| \leq \bar{C}_{p} \cdot |R| \cdot |Z - W|  \leq 2 \bar{C}_{p} \delta_1 |Z - W|
\leq \frac{|Z - W|}{10},
\label{eq_1001}
\end{equation}
where we used  (\ref{eq_A1038}) in the last passage.
By using (\ref{eq_1738_}), (\ref{eq_1746}) and also (\ref{eq_1001}) with $R = Z = X_1$ and $W = Y_1$,
\begin{equation}
|Y_1 - X_1| \leq \frac{10}{9}  \cdot|\Phi_a(X_1, Y_1) - \Phi_a(X_1, X_1)| = \frac{10}{9} \cdot|\Phi_a(X_1, Y_1) |
= \frac{10}{9} \cdot d(x_1, y_1) = \frac{10}{9} \cdot \eps, \label{eq_923}
\end{equation}
where $\Phi_a(X_1, X_1) = 0$ by (\ref{eq_2224_}). From (\ref{eq_2224_}), (\ref{eq_0845_}) and the fact that $X_0 = 0$,
\begin{equation}  2\sigma \leq d(x_0, y_2) = |\Phi_a(X_0, Y_2)| = |Y_2| = |(Y_2 - X_2) + (X_2 - X_0)|. \label{eq_1532}
\end{equation}
Note that $|X_2 - X_0| = |\Phi_a(X_0, X_2)| = 2 \sigma$ from (\ref{eq_1738_}), (\ref{eq_2224_}) and (\ref{eq_0845_}). Hence, by squaring (\ref{eq_1532}),
\begin{equation}  (2 \sigma)^2 \leq |Y_2 - X_2|^2 + 2 \langle Y_2 - X_2, X_2 - X_0 \rangle + (2 \sigma)^2. \label{eq_1540} \end{equation}
According to (\ref{eq_0845_}), the point
$x_1$ is the midpoint of the geodesic between $x_0 = a$ and $x_2$. Therefore
$x_2 = \exp_a(X_2) = \exp_a(2 X_1)$ and by strong-convexity  $2 X_1 = X_2$.
Consequently $X_2 - X_0 = 2 (X_2 - X_1)$, and from (\ref{eq_1540}) we deduce that
\begin{equation}  \langle Y_2 - X_2, X_2 - X_1 \rangle = \frac{1}{2} \langle Y_2 - X_2, X_2 - X_0 \rangle
\geq -\frac{1}{4} |Y_2 - X_2|^2.
\label{eq_1525} \end{equation}
Our next goal, like in \cite[Lemma 16]{FM}, is to prove that
\begin{equation}
 \langle Y_2 - X_2, Y_1 - Y_2 \rangle \geq -\frac{1}{3} |Y_2 - X_2|^2. \label{eq_1545}
 \end{equation}
 Begin by applying (\ref{eq_2224_}) and (\ref{eq_0845_}), in order to obtain
\begin{equation} 2\sigma \leq d(y_0, x_2) = \left |\Phi_a(Y_0, X_2) \right| = \left| \left( \Phi_a(Y_0, X_2) - \Phi_a(Y_0, Y_2) \right) \,  +\, \Phi_a(Y_0, Y_2) \right|. \label{eq_1729} \end{equation}
From (\ref{eq_0845_}), the point $y_1$ is the midpoint of the geodesic between $y_0$ and $y_2$. This implies that
$F_a(Y_0, Y_2) = 2 F_a(Y_0, Y_1)$ and therefore
$\Phi_a(Y_0, Y_2) = 2 \Phi_a(Y_0, Y_1)$.
Recall that $|\Phi_a(Y_0, Y_2)| = d(y_0, y_2) = 2 \sigma$, according to (\ref{eq_0845_}). Thus, by squaring (\ref{eq_1729}) and rearranging,
\begin{align} \nonumber
- |\Phi_a(Y_0, X_2) & - \Phi_a(Y_0, Y_2)|^2 \leq 2 \langle \Phi_a(Y_0, X_2) - \Phi_a(Y_0, Y_2), \Phi_a(Y_0, Y_2) \rangle
\\ & = 4 \langle \Phi_a(Y_0, X_2) - \Phi_a(Y_0, Y_2), \Phi_a(Y_0, Y_2) - \Phi_a(Y_0, Y_1) \rangle. \label{eq_1733} \end{align}
The deduction of (\ref{eq_1545}) from (\ref{eq_1733}) involves several approximations.
Begin by using (\ref{eq_1733}) and also (\ref{eq_1001}) with $R = Y_0, Z = X_2, W = Y_2$, to obtain
\begin{equation}
- \left(11/10 \right)^2 \cdot |X_2 - Y_2|^2 \leq 4 \langle \Phi_a(Y_0, X_2) - \Phi_a(Y_0, Y_2), \Phi_a(Y_0, Y_2) - \Phi_a(Y_0, Y_1) \rangle. \label{eq_1740} \end{equation}
Applying (\ref{eq_1413}), together
with (\ref{eq_1001}) for $R = Z = X_2, W = Y_2$, we obtain
\begin{equation}
|Y_0| = |\Phi_a(X_0, Y_0)| \leq |\Phi_a(X_2, Y_2)| = |\Phi_a(X_2, Y_2) -\Phi_a(X_2, X_2)| \leq \frac{11}{10} \cdot |Y_2 - X_2|.
\label{eq_953} \end{equation}
According to Lemma A.1
and (\ref{eq_953}), for any $Z, W \in \{ 0=X_0, X_1, X_2, Y_0, Y_1, Y_2 \}$,
\begin{equation}
\left| \Phi_a(Y_0, Z) - \Phi_a(Y_0, W)  - (Z - W)  \right| \leq \bar{C}_{p} \cdot |Y_0| \cdot |Z - W| \leq 2 \bar{C}_{p} \cdot |Y_2 - X_2| \cdot |Z - W|.
\label{eq_1000}
\end{equation}
It follows from (\ref{eq_1740}) and from the case $Z = X_2, W = Y_2$ in (\ref{eq_1000})  that
\begin{align} \label{eq_1742}
-&(11/10)^2  \cdot |X_2 - Y_2|^2 \\ & \leq 4 \langle X_2 - Y_2, \Phi_a(Y_0, Y_2) - \Phi_a(Y_0, Y_1) \rangle
+ 8 \bar{C}_{p} |X_2 - Y_2|^2 \cdot |\Phi_a(Y_0, Y_2) - \Phi_a(Y_0, Y_1)|. \nonumber \end{align}
Note that $|\Phi_a(Y_0, Y_2) - \Phi_a(Y_0, Y_1)| \leq 2 |Y_2 - Y_1|$,
as follows from an application of (\ref{eq_1001}) with $R = Y_0, Z = Y_2, W = Y_1$.
We now use (\ref{eq_1000}) with $Z = Y_2$ and $W = Y_1$, and upgrade (\ref{eq_1742}) to
\begin{align} \label{eq_1745}
- \left(11/10 \right)^2  |X_2 - Y_2|^2 & \leq 4 \langle X_2 - Y_2, Y_2 - Y_1 \rangle
+ 30 \cdot  \bar{C}_{p} |X_2 - Y_2|^2 \cdot |Y_2 - Y_1|. \end{align}
The next step is to use that $|Y_2 - Y_1| \leq |Y_2| + |Y_1| \leq 4 \delta_1 \leq 1 / (300 \bar{C}_{p})$
according to (\ref{eq_A1038}) and (\ref{eq_921}). Thus (\ref{eq_1745}) implies
$$ - \left( 11/10 \right)^2 \cdot |X_2 - Y_2|^2 \leq 4 \langle X_2 - Y_2, Y_2 - Y_1 \rangle
+ \frac{|X_2 - Y_2|^2}{10}, $$
and (\ref{eq_1545}) follows. From (\ref{eq_1525}) and (\ref{eq_1545}),
\begin{align}  \nonumber
\langle Y_2 - X_2, & Y_1 - X_1 \rangle  = \langle Y_2 - X_2, (Y_1 - Y_2) + (Y_2 - X_2) + (X_2 - X_1) \rangle \\ & \geq
-\frac{|X_2 - Y_2|^2}{3} + |X_2 - Y_2|^2 - \frac{|X_2 - Y_2|^2}{4} \geq \frac{1}{3} \cdot |Y_2 - X_2|^2. \label{eq_909}
\end{align}
According to (\ref{eq_923}), (\ref{eq_909}) and the Cauchy-Schwartz inequality,
\begin{equation} \frac{10}{9} \cdot \eps \cdot |Y_2 - X_2| \geq |Y_2 - X_2| \cdot |Y_1 - X_1| \geq \langle Y_2 - X_2, Y_1 - X_1 \rangle \geq \frac{1}{3} \cdot |Y_2 - X_2|^2.
\label{eq_1512}
\end{equation}
From (\ref{eq_1512}),
\begin{equation}
 |Y_2 - X_2| \leq 4 \eps. \label{eq_927}
\end{equation}
We may summarize (\ref{eq_923}), (\ref{eq_953}) and (\ref{eq_927}) by
\begin{equation}
|Y_i - X_i| \leq 5 \eps \qquad \qquad \qquad (i=0,1,2).
\label{eq_1114}
\end{equation}
For $i=0,1,2$, we use (\ref{eq_1738_}), (\ref{eq_1114}) and also (\ref{eq_1001}) with $R = Z = X_i$ and $W = Y_i$. This yields
\begin{equation} d(x_i, y_i) = |\Phi_a(X_i, Y_i)| = |\Phi_a(X_i, Y_i) - \Phi_a(X_i, X_i)| \leq (11/10) \cdot |Y_i - X_i| \leq 6 \eps, \label{eq_1538} \end{equation}
where $\Phi_a(X_i, X_i) = 0$ according to (\ref{eq_2224_}). The lemma follows
from (\ref{eq_1746}) and (\ref{eq_1538}).
\end{proof}

{
}


\begin{thebibliography}{99}
\addcontentsline{toc}{section}{References}
\setlength{\itemsep}{1pt}

\bibitem{amb} Ambrosio, L.,
{\it
Lecture notes on optimal transport problems.} Mathematical aspects of evolving interfaces (Funchal, 2000).
Lecture Notes in Math., Vol. 1812, Springer, Berlin, (2003), 1--52.

\bibitem{BE} Bakry, D., \'Emery, M., {\it Diffusions hypercontractives}. S\'eminaire de probabilit\'es XIX, 1983/84, Lecture Notes in Math., Vol. 1123, Springer, Berlin, (1985), 177–-206.

\bibitem{BQ}
 Bakry, D., Qian, Z., {\it Some new results on eigenvectors via dimension, diameter, and Ricci curvature. }
 Adv. Math., Vol. 155, No. 1, (2000), 98–-153.


\bibitem{BGL}
 Bakry, D., Gentil, I., Ledoux, M., {\it Analysis and geometry of Markov diffusion operators.}
  Grundlehren der Mathematischen Wissenschaften, Vol. 348, Springer, 2014.


\bibitem{bayle} Bayle, V.,
{\it Propri\'et\'es de Concavit\'e du Profil Isop\'erim\'etrique et Applications}.
Ph.D. thesis, Institut Joseph Fourier, Grenoble, 2004. \\ Available at \verb"http://www.youscribe.com"


\bibitem{bobkov}  Bobkov, S. G., {\it Extremal properties of half-spaces for log-concave distributions. }
Ann. Probab., Vol. 24, No. 1, (1996), 35–-48.

\bibitem{bobkov2} Bobkov, S. G., {\it On concentration of distributions of random weighted sums. }
Ann. Prob., Vol. 31, No. 1, (2003), 195–-215.

\bibitem{buser} Buser, P., {\it A note on the isoperimetric constant. }
Ann. Sci. \'Ecole Norm. Sup. (4), Vol. 15, No. 2, (1982), 213–-230.

\bibitem{CFM} Caffarelli, L., Feldman, M., McCann, R. J., {\it
Constructing optimal maps for Monge's transport problem as a limit of strictly convex costs.}
J. Amer. Math. Soc., Vol. 15, No. 1, (2002), 1–-26.

\bibitem{cayley} Cayley, A., {\it On Monge's ``M\'emoire sur la Th\'eorie des D\'eblais et des Remblais.''}
Proc. London Math. Soc., Vol. s1-14, Issue 1, (1882), 139--143. \\ Available
at \verb"http://dx.doi.org/10.1112/plms/s1-14.1.139"



\bibitem{Ch} Chavel, I., {\it  Riemannian geometry. A modern introduction. } Cambridge Studies in Advanced Mathematics, Vol. 98.
Cambridge University Press, Cambridge, 2006.


\bibitem{CE} Cheeger, J., Ebin, D. G., {\it  Comparison theorems in Riemannian geometry}. AMS Chelsea Publishing, Providence, RI, 2008.


\bibitem{CMS1} Cordero-Erausquin, D., McCann, R. J., Schmuckenschl\"eger, M., {\it A Riemannian interpolation inequality \'a la Borell, Brascamp and Lieb.}
Invent. Math., Vol. 146, No. 2, (2001),  219–-257.

\bibitem{CMS2} Cordero-Erausquin, D., McCann, R. J., Schmuckenschl\"eger, M., {\it Pr\'ekopa-Leindler type inequalities on Riemannian manifolds, Jacobi fields and optimal transport}. Ann. Fac. Sci. Toulouse Math. (6), Vol. 15, No. 4, (2006),  613--635.



\bibitem{edwards} Edwards, C. H., {\it  Advanced calculus of several variables. }
Dover Publications, Inc., New York, 1994.


\bibitem{EGang} Evans, L. C., Gangbo, W.,
 {\it Differential equations methods for the Monge-Kantorovich mass transfer problem.}
  Mem. Amer. Math. Soc., Vol. 137, No. 65, (1999), 1--66.

\bibitem{EG} Evans, L. C., Gariepy, R. F., {\it
 Measure theory and fine properties of functions. } Studies in Advanced Mathematics. CRC Press, Boca Raton, FL, 1992.

\bibitem{FM} Feldman, M., McCann, R. J., {\it
 Monge's transport problem on a Riemannian manifold. } Trans. Amer. Math. Soc., Vol. 354, No. 4, (2002), 1667–-1697.


\bibitem{FG}  Fradelizi, M., Gu\'edon, O., {\it The extreme points of subsets of s-concave probabilities and a geometric localization theorem. }
 Discrete Comput. Geom., Vol. 31, No. 2, (2004),  327–-335.

\bibitem{gangbo}
 Gangbo, W., {\it The Monge mass transfer problem and its applications.} Monge-Amp\`ere equation: applications to geometry and optimization (Deerfield Beach, FL, 1997).  Contemp. Math., Vol. 226, Amer. Math. Soc., Providence, RI, (1999), 79–-104.

\bibitem{GM} Gromov, M., Milman, V. D., {\it
 Generalization of the spherical isoperimetric inequality to uniformly convex Banach spaces. }
 Compositio Math., Vol. 62, No. 3, (1987), 263–-282.

\bibitem{G_levy} Gromov, M., {\it Paul Levy's isoperimetric inequality}.
Appendix C in the book {\it Metric structures for Riemannian and non-Riemannian spaces}
by M. Gromov.
 Modern Birkh\"auser Classics. Birkh\"auser Boston, Inc., Boston, MA, 2007.

\bibitem{G_MS} Gromov, M., {\it Isoperimetric inequalities in Riemannian manifolds}.
Appendix I in the book {\it Asymptotic Theory of Finite Dimensional Normed Spaces}
by V. D. Milman and G. Schechtman. Lecture notes in Math., Vol. 1200, Springer-Verlag, Berlin, 1986.


\bibitem{gr_waist} Gromov, M.,
{\it Isoperimetry of waists and concentration of maps. } Geom. Funct. Anal. (GAFA), Vol. 13, No. 1, (2003),  178–-215.

\bibitem{HK} Heintze, E., Karcher, H.,
{\it
A general comparison theorem with applications to volume estimates for submanifolds.}
Ann. Sci. \'Ecole Norm. Sup. (4), Vol. 11, No. 4, (1978), 451–-470.

\bibitem{KLS} Kannan, R., Lov\'asz, L.,  Simonovits, M., {\it Isoperimetric problems for convex bodies and a localization lemma. }
Discrete Comput. Geom., Vol. 13, No. 3--4, (1995), 541–-559.

\bibitem{KA} Kantorovich, L. V., Akilov, G. P., {\it Functional analysis.}   Second edition.
Pergamon Press, Oxford-Elmsford, NY, 1982.

\bibitem{ledoux} Ledoux, M., {\it Spectral gap, logarithmic Sobolev constant, and geometric bounds. }
Surveys in differential geometry. Vol. IX,
Int. Press, Somerville, MA, (2004), 219–-240.

\bibitem{LY} Li, P., Yau, S. T., {\it Eigenvalues of a compact Riemannian manifold.}
Amer. Math. Soc., Proc. Symp. Pure Math., Vol. 36, (1980), 205--239.

\bibitem{LS} Lov\'asz, L., Simonovits, M., {\it Random walks in a convex body and an improved volume algorithm. }
Random Structures Algorithms, Vol. 4, No. 4, (1993), 359–-412.


\bibitem{E_milman_model} Milman, E.,
{\it Sharp isoperimetric inequalities and model spaces for curvature-dimension-diameter condition}.
Available on arXiv. To appear in J. Eur. Math. Soc.

\bibitem{e_milman_invent}
Milman, E., {\it
On the role of convexity in isoperimetry, spectral gap and concentration.}
Invent. Math., Vol. 177, No. 1, (2009),  1–-43.

\bibitem{morgan} Morgan, F., {\it Manifolds with density}. Notices Amer. Math. Soc., Vol. 52, No. 8, (2005), 853–-858.

\bibitem{PW}
 Payne, L. E., Weinberger, H. F., {\it  An optimal Poincar\'e inequality for convex domains. }
 Arch. Rational Mech. Anal., Vol. 5, (1960), 286–-292.

\bibitem{stein} Stein, E. M.,
{\it  Singular integrals and differentiability properties of functions. }
Princeton Mathematical Series, No. 30,  Princeton University Press, Princeton, NJ, 1970.

\bibitem{TW} Trudinger, N. S., Wang, X.-J., {\it  On the Monge mass transfer problem. }
Calc. Var. Partial Differential Equations, Vol.  13, No. 1, (2001), 19–-31.

\bibitem{W} Whitney, H., {\it Analytic extensions of functions defined in closed sets}.
Trans. Amer. Math. Soc., Vol. 36, No. 1, (1934), 63-–89.


\bibitem{YZ} Yang, H. C., Zhong, J. Q., {\it On the estimate of the first eigenvalue of a compact
Riemannian manifold.} Sci. Sinica Ser., Vol. 27, No. 12, (1984), 1265--1273.

\end{thebibliography}
\end{document}